\documentclass[a4paper,12pt]{article}
\usepackage[top=2.5cm, bottom=2.5cm, left=2.5cm, right=2.5cm]{geometry}

\def\cleverefoptions{capitalize}
\usepackage{amsmath}
\usepackage{amsthm}
\usepackage{amssymb}
\usepackage[colorlinks=true,citecolor=black,linkcolor=black,urlcolor=blue]{hyperref}
\usepackage{cleveref}

\usepackage[utf8]{inputenc}
\usepackage[T1]{fontenc}
\usepackage[dvipsnames]{xcolor}
\usepackage{mathtools} 
\usepackage{hyphsubst}
\HyphSubstLet{english}{usenglishmax}
\usepackage{longtable}
\usepackage{booktabs}
\usepackage{dashrule}
\usepackage[width=.85\textwidth]{caption}
\usepackage[inline]{enumitem}
\usepackage{siunitx}
\usepackage{tikz}
\usetikzlibrary{matrix,decorations.pathmorphing,decorations.markings,decorations.pathreplacing,arrows.meta}
\usepackage{stackengine}
\usepackage{adjustbox}
\usepackage{float}
\usepackage{pgfplots}
\pgfplotsset{compat=1.5}
\usepackage{lscape}
\usepackage{xifthen}
\makeatletter

\pgfdeclaredecoration{complete sines}{initial}
{
        \state{initial}[
        width=+0pt,
        next state=sine,
        persistent precomputation={\pgfmathsetmacro\matchinglength{
                        \pgfdecoratedinputsegmentlength / int(\pgfdecoratedinputsegmentlength/\pgfdecorationsegmentlength)}
                \setlength{\pgfdecorationsegmentlength}{\matchinglength pt}
        }] {}
        \state{sine}[width=\pgfdecorationsegmentlength]{
                \pgfpathsine{\pgfpoint{0.25\pgfdecorationsegmentlength}{0.5\pgfdecorationsegmentamplitude}}
                \pgfpathcosine{\pgfpoint{0.25\pgfdecorationsegmentlength}{-0.5\pgfdecorationsegmentamplitude}}
                \pgfpathsine{\pgfpoint{0.25\pgfdecorationsegmentlength}{-0.5\pgfdecorationsegmentamplitude}}
                \pgfpathcosine{\pgfpoint{0.25\pgfdecorationsegmentlength}{0.5\pgfdecorationsegmentamplitude}}
        }
        \state{final}{}
}
\makeatletter 

\makeatletter
\DeclareRobustCommand{\cev}[1]{%
        {\mathpalette\do@cev{#1}}%
}
\newcommand{\do@cev}[2]{%
        \vbox{\offinterlineskip
                \sbox\z@{$\m@th#1 x$}%
                \ialign{##\cr
                        \hidewidth\reflectbox{$\m@th#1\vec{}\mkern4mu$}\hidewidth\cr
                        \noalign{\kern-\ht\z@}
                        $\m@th#1#2$\cr
                }%
        }%
}
\makeatother
\def\ov#1{\overline{#1}}
\def\wh#1{\widehat{#1}}
\def\wt#1{\widetilde{#1}}
\newcommand{\whd}[1]{\stackon[-6.5pt]{$\dot {#1}$}{$\widehat{}$}}
\def\mod{\mathrm{mod}}
\DeclareMathSymbol{\shortminus}{\mathbin}{AMSa}{"39}
\newcommand{\ab}{\allowbreak}
\newcommand{\eqdef}{\coloneqq}
\renewcommand{\AA}{\mathbb{A}}
\newcommand{\bAA}{\mathbf{A}}
\newcommand{\DD}{\mathbb{D}}
\newcommand{\EE}{\mathbb{E}}
\newcommand{\ZZ}{\mathbb{Z}}
\newcommand{\NN}{\mathbb{N}}
\newcommand{\QQ}{\mathbb{Q}}
\newcommand{\RR}{\mathbb{R}}
\newcommand{\PP}{\mathbb{P}}
\newcommand{\CCC}{\mathbb{C}}
\newcommand{\MM}{\mathbb{M}}
\newcommand{\CA}{\mathcal{A}}
\newcommand{\CD}{\mathcal{D}}
\newcommand{\CH}{\mathcal{H}}
\newcommand{\CF}{\mathcal{F}}
\newcommand{\CN}{\mathcal{N}}
\newcommand{\bh}{\mathbf{h}}
\DeclareMathOperator{\crk}{\mathbf{crk}}
\DeclareMathOperator{\specc}{{\mathbf{specc}}}
\newcommand{\sgn}{\mathrm{sgn}}
\newcommand{\Dyn}{\mathrm{Dyn}}
\newcommand{\Gl}{\mathrm{Gl}}
\newcommand{\Ad}{\mathrm{Ad}}
\DeclareMathOperator{\Ker}{Ker}

\providecommand{\tikzsetnextfilename}[1]{}%
\providecommand{\tikzexternalenable}{}%
\providecommand{\tikzexternaldisable}{}%
\newcommand{\mppms}{\scalebox{0.75}[1]{$\scriptstyle -$}}
\newcommand{\mppps}{\raisebox{.5pt}{\scalebox{0.75}{$\scriptstyle +$}}}
\newcommand{\mppmss}{\scalebox{0.75}[1]{$\scriptscriptstyle -$}}
\newcommand{\mpppss}{\raisebox{.375pt}{\scalebox{0.75}{$\scriptscriptstyle +$}}}
\newcommand{\grafcrkzA}[2]{\tikzsetnextfilename{grafcrkzA_#1_#2}%
        \begin{tikzpicture}[baseline=(current bounding box),label distance=-2pt, xscale=#1, yscale=#2]
                \node[circle, fill=black, inner sep=0pt, minimum size=3.5pt, label=above:{\mscript{1}}] (n3) at (2, 0.5) {};
                \node (n4) at (3, 0.5) {{\mscript{}}};
                \node (n5) at (4, 0.5) {{\mscript{}}};
                \node[circle, fill=black, inner sep=0pt, minimum size=3.5pt, label=above:{\mscript{p-1}}] (n6) at (5, 0.5) {};
                \node[circle, fill=black, inner sep=0pt, minimum size=3.5pt, label=above:{\mscript{p}}] (n7) at (6, 0.5) {};
                \node[inner sep=0pt, minimum size=3.5pt, label=above:{\mscript{r+1}}] (n8) at (7, 0.5) {$*$};
                \node[circle, fill=black, inner sep=0pt, minimum size=3.5pt, label=below:{\mscript{p+1}}] (n11) at (2.5  , 0) {};
                \node (n12) at (3.5  , 0) {{\mscript{}}};
                \node (n13) at (4.5  , 0) {{\mscript{}}};
                \node[circle, fill=black, inner sep=0pt, minimum size=3.5pt, label=below:{\mscript{r-1}}] (n14) at (5.5  , 0) {};
                \node[circle, fill=black, inner sep=0pt, minimum size=3.5pt, label=below:{\mscript{r}}] (n15) at (6.5  , 0) {};
                \foreach \x/\y in {6/7, 7/8,  14/15}
                \draw [-stealth, shorten <= 2.50pt, shorten >= 2.50pt] (n\x) to  (n\y);
                \foreach \x/\y in {4/5, 12/13}
                \draw [dotted, -, shorten <= -2.50pt, shorten >= -2.50pt] (n\x) to  (n\y);
                \foreach \x/\y in {3/4, 11/12}
                \draw [-stealth, shorten <= 2.50pt, shorten >= -2.50pt] (n\x) to  (n\y);
                \foreach \x/\y in {5/6, 13/14}
                \draw [-stealth, shorten <= -2.50pt, shorten >= 2.50pt] (n\x) to  (n\y);
                \draw [-stealth, shorten <= 1.5pt, shorten >= 0.50pt] (n15) to  (n8);
        \end{tikzpicture}
}

\newcommand{\grafcrkzDi}[2]{\tikzsetnextfilename{grafcrkzDi_#1_#2}%
        \begin{tikzpicture}[baseline=(current bounding box),label distance=-2pt, xscale=#1, yscale=#2]
                \node[circle, fill=black, inner sep=0pt, minimum size=3.5pt, label=above:{\mscript{2}}] (n4) at (2, 0.5) {};
                \node (n5) at (3, 0.5) {{\mscript{}}};
                \node (n6) at (4, 0.5) {{\mscript{}}};
                \node[circle, fill=black, inner sep=0pt, minimum size=3.5pt, label=above:{\mscript{r-1}}] (n7) at (5, 0.5) {};
                \node[circle, fill=black, inner sep=0pt, minimum size=3.5pt, label=above:{\mscript{\phantom{1-}r\phantom{-1}}}] (n8) at (6, 0.5) {};
                \node[inner sep=0pt, minimum size=3.5pt, label=above:{\mscript{r+1}}] (n9) at (7, 0.5) {$*$};
                \node[circle, fill=black, inner sep=0pt, minimum size=3.5pt, label=left:{\mscript{1}}] (n1) at (5.5  , 0) {};
                \foreach \x/\y in { 6/7, 7/8, 8/9}
                \draw [-stealth, shorten <= 2.50pt, shorten >= 2.50pt] (n\x) to  (n\y);
                \draw [dotted, -, shorten <= -2.50pt, shorten >= -2.50pt] (n5) to  (n6);
                \draw [-stealth, shorten <= 2.50pt, shorten >= -2.50pt] (n4) to  (n5);
                \draw [-stealth, shorten <= -2.50pt, shorten >= 2.50pt] (n6) to  (n7);
                \draw [-stealth, shorten <= 1.5pt, shorten >= 2.00pt] (n1) to  (n8);
        \end{tikzpicture}
}

\newcommand{\grafcrkzDii}[2]{\tikzsetnextfilename{grafcrkzDii_#1_#2}%
        \begin{tikzpicture}[baseline=(current bounding box),label distance=-2pt, xscale=#1, yscale=#2]
                \node[circle, fill=black, inner sep=0pt, minimum size=3.5pt, label=above:{\mscript{1}}] (n1) at (2, 0.5) {};
                \node (n4) at (3, 0.5) {{\mscript{}}};
                \node (n5) at (4, 0.5) {{\mscript{}}};
                \node[circle, fill=black, inner sep=0pt, minimum size=3.5pt, label=above:{\mscript{p-1}}] (n6) at (5, 0.5) {};
                \node[circle, fill=black, inner sep=0pt, minimum size=3.5pt, label=above:{\mscript{p}}] (n7) at (6, 0.5) {};
                \node[inner sep=0pt, minimum size=3.5pt, label=above:{\mscript{r+1}}] (n8) at (7, 0.5) {$*$};
                \node[circle, fill=black, inner sep=0pt, minimum size=3.5pt, label=below:{\mscript{p+1}}] (n11) at (2.5  , 0) {};
                \node (n12) at (3.5  , 0) {{\mscript{}}};
                \node (n13) at (4.5  , 0) {{\mscript{}}};
                \node[circle, fill=black, inner sep=0pt, minimum size=3.5pt, label=below:{\mscript{r-1}}] (n14) at (5.5  , 0) {};
                \node[circle, fill=black, inner sep=0pt, minimum size=3.5pt, label=below:{\mscript{r}}] (n15) at (6.5  , 0) {};
                \foreach \x/\y in {6/7, 7/8, 14/15}
                \draw [-stealth, shorten <= 2.50pt, shorten >= 2.50pt] (n\x) to  (n\y);
                \foreach \x/\y in {4/5, 12/13}
                \draw [dotted, -, shorten <= -2.50pt, shorten >= -2.50pt] (n\x) to  (n\y);
                \foreach \x/\y in {1/4, 11/12}
                \draw [-stealth, shorten <= 2.50pt, shorten >= -2.50pt] (n\x) to  (n\y);
                \foreach \x/\y in {5/6, 13/14}
                \draw [-stealth, shorten <= -2.50pt, shorten >= 2.50pt] (n\x) to  (n\y);
                \draw [-stealth, shorten <= 1.5pt, shorten >= 0.50pt] (n15) to  (n8);
                \draw [-stealth,shorten <= 2.50pt,shorten >= 2.50pt] ([yshift=-2.5pt]n1.east) to  ([yshift=3.5pt]n15.west);
        \end{tikzpicture}
}

\newcommand{\grafcrkzDiii}[2]{\tikzsetnextfilename{grafcrkzDiii_#1_#2}%
        \begin{tikzpicture}[baseline=(current bounding box),label distance=-2pt, xscale=#1, yscale=#2]
                \node[circle, fill=black, inner sep=0pt, minimum size=3.5pt, label=below:{\mscript{s\phantom{1}}}] (n1) at (2, 0.50) {};
                \node[circle, fill=black, inner sep=0pt, minimum size=3.5pt, label={[xshift=-0.6ex]left:{\mscript{s\mppps 1}}}] (n2) at (3, 1  ) {};
                \node[circle, fill=black, inner sep=0pt, minimum size=3.5pt, label=right:{\mscript{s\mppps 2}}] (n3) at (3, 0  ) {};
                \node[circle, fill=black, inner sep=0pt, minimum size=3.5pt, label={[xshift=0.6ex]above:{\mscript{s\mppps 3}}}] (n4) at (4, 0.50) {};
                \node[circle, fill=black, inner sep=0pt, minimum size=3.5pt, label=below:$\scriptstyle 1$] (n5) at (0, 0.50) {};
                \node (n6) at (1, 0.50) {};
                \node (n7) at (5, 0.50) {};
                \node[circle, fill=black, inner sep=0pt, minimum size=3.5pt, label=above:{\mscript{\phantom{1}p\phantom{1}}}] (n8) at (6, 0.50) {};
                \node[inner sep=0pt, minimum size=3.5pt, label=above:{\mscript{r+1}}] (n9) at (7, 0.50) {$*$};
                \node[circle, fill=black, inner sep=0pt, minimum size=3.5pt, label=below:{\mscript{\phantom{1}r\phantom{1}}}] (n10) at (6.50, 0  ) {};
                \node (n11) at (5.50, 0  ) {};
                \node[circle, fill=black, inner sep=0pt, minimum size=3.5pt, label=below:{\mscript{p\mppps 1}}] (n12) at (4.50, 0  ) {};
                \foreach \x/\y in {6/1, 11/10}
                \draw [-, densely dotted, gray!90, -, shorten <= -2.50pt, shorten >= 2.50pt] (n\x) to  (n\y);
                \draw [-, densely dotted, gray!90, -, shorten <= 2.50pt, shorten >= -2.50pt] (n4) to  (n7);
                \foreach \x/\y in {1/2, 1/3, 2/4, 3/4}
                \draw [-stealth, shorten <= 2.50pt, shorten >= 2.50pt] (n\x) to  (n\y);
                \foreach \x/\y in {5/6, 12/11}
                \draw [-stealth, shorten <= 2.50pt, shorten >= -2.50pt] (n\x) to  (n\y);
                \draw [-stealth, shorten <= -2.50pt, shorten >= 2.50pt] (n7) to  (n8);
                \draw [-stealth, shorten <= 1.50pt, shorten >= 1.00pt] (n10) to  (n9);
                \draw [-stealth, shorten <= 1.50pt, shorten >= 1.00pt] (n8) to  (n9);
        \end{tikzpicture}
}

\newcommand{\ringv}[6]{\begin{adjustbox}{margin=0.5ex}\tikzsetnextfilename{ringv_#1_#2_#3_#4_#5_#6}\begin{tikzpicture}[baseline={([yshift=-2.75pt]current bounding box)},label distance=-2pt,
                xscale=0.95, yscale=0.95]
                \foreach \n/\x/\y in {1/0.95106/0.30902, 2/0/1.0000, 3/-0.95106/0.30902, 4/-0.58779/-0.80902, 5/0.58779/-0.80902}
                \node[draw, fill=green!30, inner sep=0pt, minimum size=3.5pt] (n\n) at (\x, \y) {};
                \draw[gray,dashed,line width=0.7] (0,0) circle [radius=0.80902];
                \ifthenelse{#2=0}{%
                        \draw[-{Stealth[scale=1.5,open,fill=white]}, shorten <= .50pt, shorten >= .50pt] (n1) to  (n2);\node[circle, draw, fill=white, inner sep=0pt, minimum size=5.5pt,line width=0.7pt] at (0.47553, 0.65451) {};%
                }{%
                        \draw [{Stealth[scale=1.5]}-, shorten <= .50pt, shorten >= .50pt] (n1) to  (n2);\node[circle, fill, inner sep=0pt, minimum size=5.5pt] at (0.47553, 0.65451) {};%
                }
                \ifthenelse{#3=0}{%
                        \draw[-{Stealth[scale=1.5,open,fill=white]}, shorten <= .50pt, shorten >= .50pt] (n2) to  (n3);\node[circle, draw, fill=white, inner sep=0pt, minimum size=5.5pt,line width=0.7pt] at (-0.47553, 0.65451) {};%
                }{%
                        \draw [{Stealth[scale=1.5]}-, shorten <= .50pt, shorten >= .50pt] (n2) to  (n3);\node[circle, fill, inner sep=0pt, minimum size=5.5pt] at (-0.47553, 0.65451) {};%
                }
                \ifthenelse{#4=0}{%
                        \draw[-{Stealth[scale=1.5,open,fill=white]}, shorten <= .50pt, shorten >= .50pt] (n3) to  (n4);\node[circle, draw, fill=white, inner sep=0pt, minimum size=5.5pt,line width=0.7pt] at (-0.76942, -0.25000) {};%
                }{%
                        \draw [{Stealth[scale=1.5]}-, shorten <= .50pt, shorten >= .50pt] (n3) to  (n4);\node[circle, fill, inner sep=0pt, minimum size=5.5pt] at (-0.76942, -0.25000) {};%
                }
                \ifthenelse{#5=0}{%
                        \draw[-{Stealth[scale=1.5,open,fill=white]}, shorten <= .50pt, shorten >= .50pt] (n4) to  (n5);\node[circle, draw, fill=white, inner sep=0pt, minimum size=5.5pt,line width=0.7pt] at (0, -0.80902) {};%
                }{%
                        \draw [{Stealth[scale=1.5]}-, shorten <= .50pt, shorten >= .50pt] (n4) to  (n5);\node[circle, fill, inner sep=0pt, minimum size=5.5pt] at (0, -0.80902) {};%
                }
                \ifthenelse{#6=0}{%
                        \draw[-{Stealth[scale=1.5,open,fill=white]}, shorten <= .50pt, shorten >= .50pt] (n5) to  (n1);\node[circle, draw, fill=white, inner sep=0pt, minimum size=5.5pt,line width=0.7pt] at (0.76942, -0.25000) {};%
                }{%
                        \draw [{Stealth[scale=1.5]}-, shorten <= .50pt, shorten >= .50pt] (n5) to  (n1);\node[circle, fill, inner sep=0pt, minimum size=5.5pt] at (0.76942, -0.25000) {};%
                }
\end{tikzpicture}\end{adjustbox}}%
\newcommand{\ringvi}[7] {\begin{adjustbox}{margin=0.5ex}\tikzsetnextfilename{ringvi_#1_#2_#3_#4_#5_#6_#7}\begin{tikzpicture}[baseline={([yshift=-2.75pt]current bounding box)},
                label distance=-2pt,
                xscale=0.95, yscale=0.95]
                \foreach \n/\x/\y in {1/1.0000/0, 2/0.50000/0.86602, 3/-0.50000/0.86602, 4/-1.0000/0, 5/-0.50000/-0.86602, 6/0.50000/-0.86602}
                \node[draw, fill=green!30, inner sep=0pt, minimum size=3.5pt] (n\n) at (\x, \y) {};
                \draw[gray,dashed,line width=0.7] (0,0) circle [radius=0.86602];
                \ifthenelse{#2=0}{%
                        \draw[-{Stealth[scale=1.5,open,fill=white]}, shorten <= .50pt, shorten >= .50pt] (n1) to  (n2);\node[circle, draw, fill=white, inner sep=0pt, minimum size=5.5pt,line width=0.7pt] at (0.75000, 0.43301) {};%
                }{%
                        \draw [{Stealth[scale=1.5]}-, shorten <= .50pt, shorten >= .50pt] (n1) to  (n2);\node[circle, fill, inner sep=0pt, minimum size=5.5pt] at (0.75000, 0.43301) {};%
                }
                \ifthenelse{#3=0}{%
                        \draw[-{Stealth[scale=1.5,open,fill=white]}, shorten <= .50pt, shorten >= .50pt] (n2) to  (n3);\node[circle, draw, fill=white, inner sep=0pt, minimum size=5.5pt,line width=0.7pt] at (0, 0.86602) {};%
                }{%
                        \draw [{Stealth[scale=1.5]}-, shorten <= .50pt, shorten >= .50pt] (n2) to  (n3);\node[circle, fill, inner sep=0pt, minimum size=5.5pt] at (0, 0.86602) {};%
                }
                \ifthenelse{#4=0}{%
                        \draw[-{Stealth[scale=1.5,open,fill=white]}, shorten <= .50pt, shorten >= .50pt] (n3) to  (n4);\node[circle, draw, fill=white, inner sep=0pt, minimum size=5.5pt,line width=0.7pt] at (-0.75000, 0.43301) {};%
                }{%
                        \draw [{Stealth[scale=1.5]}-, shorten <= .50pt, shorten >= .50pt] (n3) to  (n4);\node[circle, fill, inner sep=0pt, minimum size=5.5pt] at (-0.75000, 0.43301) {};%
                }
                \ifthenelse{#5=0}{%
                        \draw[-{Stealth[scale=1.5,open,fill=white]}, shorten <= .50pt, shorten >= .50pt] (n4) to  (n5);\node[circle, draw, fill=white, inner sep=0pt, minimum size=5.5pt,line width=0.7pt] at (-0.75000, -0.43301) {};%
                }{%
                        \draw [{Stealth[scale=1.5]}-, shorten <= .50pt, shorten >= .50pt] (n4) to  (n5);\node[circle, fill, inner sep=0pt, minimum size=5.5pt] at (-0.75000, -0.43301) {};%
                }
                \ifthenelse{#6=0}{%
                        \draw[-{Stealth[scale=1.5,open,fill=white]}, shorten <= .50pt, shorten >= .50pt] (n5) to  (n6);\node[circle, draw, fill=white, inner sep=0pt, minimum size=5.5pt,line width=0.7pt] at (0.000, -0.86602) {};%
                }{%
                        \draw [{Stealth[scale=1.5]}-, shorten <= .50pt, shorten >= .50pt] (n5) to  (n6);\node[circle, fill, inner sep=0pt, minimum size=5.5pt] at (0.000, -0.86602) {};%
                }
                \ifthenelse{#7=0}{%
                        \draw[-{Stealth[scale=1.5,open,fill=white]}, shorten <= .50pt, shorten >= .50pt] (n6) to  (n1);\node[circle, draw, fill=white, inner sep=0pt, minimum size=5.5pt,line width=0.7pt] at (0.75000, -0.43301) {};%
                }{%
                        \draw [{Stealth[scale=1.5]}-, shorten <= .50pt, shorten >= .50pt] (n6) to  (n1);\node[circle, fill, inner sep=0pt, minimum size=5.5pt] at (0.75000, -0.43301) {};%
                }
\end{tikzpicture}\end{adjustbox}}%
\newcommand{\mscript}[1]{{$\scriptscriptstyle #1$}}

\newcommand{\grapheAn}[2]{\tikzexternaldisable\begin{tikzpicture}[label distance=-2pt, xscale=#1, yscale=#2]
                \node[circle, fill=black, inner sep=0pt, minimum size=3.5pt, label=above:\mscript{1}] (n1) at (0, 0  ) {};
                \node[circle, fill=black, inner sep=0pt, minimum size=3.5pt, label=above:\mscript{2}] (n2) at (1, 0  ) {};
                \node (n3) at (2, 0  ) {\mscript{}};
                \node (n4) at (3, 0  ) {\mscript{}};
                \node[circle, fill=black, inner sep=0pt, minimum size=3.5pt, label=above:\mscript{n-1}] (n5) at (4, 0  ) {};
                \node[circle, fill=black, inner sep=0pt, minimum size=3.5pt, label=above:\mscript{n}] (n6) at (5, 0  ) {};
                \draw [-, shorten <= -2.50pt, shorten >= 2.50pt] (n4) to  (n5);
                \draw [dotted, -, shorten <= -2.50pt, shorten >= -2.50pt] (n3) to  (n4);
                \foreach \x/\y in {1/2, 5/6}
                \draw [-, shorten <= 2.50pt, shorten >= 2.50pt] (n\x) to  (n\y);
                \draw [-, shorten <= 2.50pt, shorten >= -2.50pt] (n2) to  (n3);
\end{tikzpicture}\tikzexternalenable}

\newcommand{\grapheDn}[2]{\tikzexternaldisable\begin{tikzpicture}[label distance=-2pt, xscale=#1, yscale=#2]
                \node[circle, fill=black, inner sep=0pt, minimum size=3.5pt, label=above:\mscript{1}] (n1) at (0, 0  ) {};
                \node[circle, fill=black, inner sep=0pt, minimum size=3.5pt, label=right:\mscript{2}] (n2) at (1, 0.6) {};
                \node[circle, fill=black, inner sep=0pt, minimum size=3.5pt, label=above right:\mscript{3}] (n3) at (1, 0  ) {};
                \node (n4) at (2, 0  ) {\mscript{}};
                \node (n5) at (3, 0  ) {\mscript{}};
                \node[circle, fill=black, inner sep=0pt, minimum size=3.5pt, label=above:\mscript{n-1}] (n6) at (4, 0  ) {};
                \node[circle, fill=black, inner sep=0pt, minimum size=3.5pt, label=above:\mscript{n}] (n7) at (5, 0  ) {};
                \draw [-, shorten <= -2.50pt, shorten >= 2.50pt] (n5) to  (n6);
                \draw [dotted, -, shorten <= -2.50pt, shorten >= -2.50pt] (n4) to  (n5);
                \foreach \x/\y in {1/3, 2/3, 6/7}
                \draw [-, shorten <= 2.50pt, shorten >= 2.50pt] (n\x) to  (n\y);
                \draw [-, shorten <= 2.50pt, shorten >= -2.50pt] (n3) to  (n4);
\end{tikzpicture}\tikzexternalenable}

\newcommand{\grapheEsix}[2]{\tikzexternaldisable\begin{tikzpicture}[label distance=-2pt, xscale=#1, yscale=#2]
                \node[circle, fill=black, inner sep=0pt, minimum size=3.5pt, label=above:\mscript{1}] (n1) at (0, 0  ) {};
                \node[circle, fill=black, inner sep=0pt, minimum size=3.5pt, label=above:\mscript{2}] (n2) at (1, 0  ) {};
                \node[circle, fill=black, inner sep=0pt, minimum size=3.5pt, label=above left:\mscript{3}] (n3) at (2, 0  ) {};
                \node[circle, fill=black, inner sep=0pt, minimum size=3.5pt, label=right:\mscript{4}] (n4) at (2, 0.6  ) {};
                \node[circle, fill=black, inner sep=0pt, minimum size=3.5pt, label=above:\mscript{5}] (n5) at (3, 0  ) {};
                \node[circle, fill=black, inner sep=0pt, minimum size=3.5pt, label=above:\mscript{6}] (n6) at (4, 0  ) {};
                \foreach \x/\y in {1/2, 2/3, 3/5, 4/3, 5/6}
                \draw [-, shorten <= 2.50pt, shorten >= 2.50pt] (n\x) to  (n\y);
\end{tikzpicture}\tikzexternalenable}

\newcommand{\grapheEseven}[2]{\tikzexternaldisable\begin{tikzpicture}[label distance=-2pt, xscale=#1, yscale=#2]
                \node[circle, fill=black, inner sep=0pt, minimum size=3.5pt, label=above:\mscript{1}] (n1) at (0, 0  ) {};
                \node[circle, fill=black, inner sep=0pt, minimum size=3.5pt, label=above:\mscript{2}] (n2) at (1, 0  ) {};
                \node[circle, fill=black, inner sep=0pt, minimum size=3.5pt, label=above:\mscript{3}] (n3) at (2, 0  ) {};
                \node[circle, fill=black, inner sep=0pt, minimum size=3.5pt, label=above right:\mscript{4}] (n4) at (3, 0) {};
                \node[circle, fill=black, inner sep=0pt, minimum size=3.5pt, label=right:\mscript{5}] (n5) at (3, 0.6  ) {};
                \node[circle, fill=black, inner sep=0pt, minimum size=3.5pt, label=above:\mscript{6}] (n6) at (4, 0  ) {};
                \node[circle, fill=black, inner sep=0pt, minimum size=3.5pt, label=above:\mscript{7}] (n7) at (5, 0  ) {};
                \foreach \x/\y in {1/2, 2/3, 3/4, 4/5, 4/6, 6/7}
                \draw [-, shorten <= 2.50pt, shorten >= 2.50pt] (n\x) to  (n\y);
\end{tikzpicture}\tikzexternalenable}

\newcommand{\grapheEeight}[2]{\tikzexternaldisable\begin{tikzpicture}[label distance=-2pt, xscale=#1, yscale=#2]
                \node[circle, fill=black, inner sep=0pt, minimum size=3.5pt, label=above:\mscript{1}] (n1) at (0, 0  ) {};
                \node[circle, fill=black, inner sep=0pt, minimum size=3.5pt, label=above:\mscript{2}] (n2) at (1, 0  ) {};
                \node[circle, fill=black, inner sep=0pt, minimum size=3.5pt, label=above right:\mscript{3}] (n3) at (2, 0  ) {};
                \node[circle, fill=black, inner sep=0pt, minimum size=3.5pt, label=right:\mscript{4}] (n4) at (2, 0.6) {};
                \node[circle, fill=black, inner sep=0pt, minimum size=3.5pt, label=above:\mscript{5}] (n5) at (3, 0  ) {};
                \node[circle, fill=black, inner sep=0pt, minimum size=3.5pt, label=above:\mscript{6}] (n6) at (4, 0  ) {};
                \node[circle, fill=black, inner sep=0pt, minimum size=3.5pt, label=above:\mscript{7}] (n7) at (5, 0  ) {};
                \node[circle, fill=black, inner sep=0pt, minimum size=3.5pt, label=above:\mscript{8}] (n8) at (6, 0  ) {};
                \foreach \x/\y in {1/2, 2/3, 3/5, 4/3, 5/6, 6/7, 7/8}
                \draw [-, shorten <= 2.50pt, shorten >= 2.50pt] (n\x) to  (n\y);
\end{tikzpicture}\tikzexternalenable}
\newcommand{\grafPosDfiveCrkII}[2]{\tikzsetnextfilename{grafPosDfiveCrkII_#1_#2}
        \begin{tikzpicture}[baseline=(current bounding box),label distance=-2pt, xscale=#1, yscale=#2]
                \node[circle, fill=black, inner sep=0pt, minimum size=3.5pt, label=above:$\scriptscriptstyle 1$] (n1) at (0, 1  ) {};
                \node[circle, fill=black, inner sep=0pt, minimum size=3.5pt, label=above:$\scriptscriptstyle 2$] (n2) at (1, 1  ) {};
                \node[circle, fill=black, inner sep=0pt, minimum size=3.5pt, label=above:$\scriptscriptstyle 3$] (n3) at (2, 1  ) {};
                \node[circle, fill=black, inner sep=0pt, minimum size=3.5pt, label=below:$\scriptscriptstyle 4$] (n4) at (0, 0  ) {};
                \node[circle, fill=black, inner sep=0pt, minimum size=3.5pt, label=below:$\scriptscriptstyle 5$] (n5) at (1, 0  ) {};
                \node[circle, fill=black, inner sep=0pt, minimum size=3.5pt, label=below:$\scriptscriptstyle 6$] (n6) at (2, 0  ) {};
                \node[circle, fill=black, inner sep=0pt, minimum size=3.5pt, label=below:$\scriptscriptstyle 7$] (n7) at (3, 0  ) {};
                \node (n8) at (4, 0  ) {$\scriptscriptstyle \phantom{8}$};
                \foreach \x/\y in {1/6, 3/6}
                \draw [dashed, -, shorten <= 1.50pt, shorten >= 1.50pt] (n\x) to  (n\y);
                \foreach \x/\y in {1/2, 1/4, 2/5, 4/5, 5/6, 6/7}
                \draw [-, shorten <= 1.50pt, shorten >= 1.50pt] (n\x) to  (n\y);
        \end{tikzpicture}
}
\newcommand{\grafPosDsixCrkII}[2]{\tikzsetnextfilename{grafPosDsixCrkII_#1_#2}\tikzsetnextfilename{grafPosDsixCrkII}
        \begin{tikzpicture}[baseline=(current bounding box),label distance=-2pt, xscale=#1, yscale=#2]
                \node[circle, fill=black, inner sep=0pt, minimum size=3.5pt, label=above:$\scriptscriptstyle 1$] (n1) at (2, 1  ) {};
                \node[circle, fill=black, inner sep=0pt, minimum size=3.5pt, label=above:$\scriptscriptstyle 2$] (n2) at (1, 1  ) {};
                \node[circle, fill=black, inner sep=0pt, minimum size=3.5pt, label=above:$\scriptscriptstyle 3$] (n3) at (3, 1  ) {};
                \node[circle, fill=black, inner sep=0pt, minimum size=3.5pt, label=below:$\scriptscriptstyle 4$] (n4) at (0, 0  ) {};
                \node[circle, fill=black, inner sep=0pt, minimum size=3.5pt, label=below:$\scriptscriptstyle 5$] (n5) at (1, 0  ) {};
                \node[circle, fill=black, inner sep=0pt, minimum size=3.5pt, label=below:$\scriptscriptstyle 6$] (n6) at (2, 0  ) {};
                \node[circle, fill=black, inner sep=0pt, minimum size=3.5pt, label=below:$\scriptscriptstyle 7$] (n7) at (3, 0  ) {};
                \node[circle, fill=black, inner sep=0pt, minimum size=3.5pt, label=below:$\scriptscriptstyle 8$] (n8) at (4, 0  ) {};
                \foreach \x/\y in {1/7, 3/7}
                \draw [dashed, -, shorten <= 1.50pt, shorten >= 1.50pt] (n\x) to  (n\y);
                \foreach \x/\y in {2/5, 1/5, 4/5, 5/6, 6/7, 7/8}
                \draw [-, shorten <= 1.50pt, shorten >= 1.50pt] (n\x) to  (n\y);
        \end{tikzpicture}
}

\newcommand{\grapheRAn}[2]{\tikzexternaldisable\begin{tikzpicture}[label distance=-2pt, xscale=#1, yscale=#2]
                \node[circle, fill=black, inner sep=0pt, minimum size=3.5pt, label=above:\mscript{1}] (n1) at (0  , 0  ) {};
                \node[circle, fill=black, inner sep=0pt, minimum size=3.5pt, label=above right:\mscript{2}] (n2) at (1  , 0  ) {};
                \node (n3) at (2  , 0  ) {\mscript{}};
                \node (n4) at (3  , 0  ) {\mscript{}};
                \node[circle, fill=black, inner sep=0pt, minimum size=3.5pt, label=above:\mscript{n+1}] (n5) at (2.5, 0.6) {};
                \node[circle, fill=black, inner sep=0pt, minimum size=3.5pt, label=above left:\mscript{n-1}] (n6) at (4  , 0  ) {};
                \node[circle, fill=black, inner sep=0pt, minimum size=3.5pt, label=above:\mscript{n}] (n7) at (5  , 0  ) {};
                \foreach \x/\y in {1/5, 5/7}
                \draw [bend left=7.0, -, shorten <= 2.50pt, shorten >= 2.50pt] (n\x) to  (n\y);
                \draw [-, shorten <= -2.50pt, shorten >= 2.50pt] (n4) to  (n6);
                \draw [dotted, -, shorten <= -2.50pt, shorten >= -2.50pt] (n3) to  (n4);
                \foreach \x/\y in {1/2, 6/7}
                \draw [-, shorten <= 2.50pt, shorten >= 2.50pt] (n\x) to  (n\y);
                \draw [-, shorten <= 2.50pt, shorten >= -2.50pt] (n2) to  (n3);
\end{tikzpicture}\tikzexternalenable}

\newcommand{\grapheRDn}[2]{\tikzexternaldisable\begin{tikzpicture}[label distance=-2pt, xscale=#1, yscale=#2]
                \node[circle, fill=black, inner sep=0pt, minimum size=3.5pt, label=above:\mscript{1}] (n1) at (0  , 0  ) {};
                \node[circle, fill=black, inner sep=0pt, minimum size=3.5pt, label=right:\mscript{2}] (n2) at (1  , 0.6) {};
                \node[circle, fill=black, inner sep=0pt, minimum size=3.5pt, label=above left:\mscript{3}] (n3) at (1  , 0  ) {};
                \node (n4) at (2  , 0  ) {\mscript{}};
                \node (n5) at (3  , 0  ) {\mscript{}};
                \node[circle, fill=black, inner sep=0pt, minimum size=3.5pt, label=right:\mscript{n+1}] (n6) at (4  , 0.6) {};
                \node[circle, fill=black, inner sep=0pt, minimum size=3.5pt, label=above left:\mscript{n-1}] (n7) at (4  , 0  ) {};
                \node[circle, fill=black, inner sep=0pt, minimum size=3.5pt, label=above:\mscript{n}] (n8) at (5  , 0  ) {};
                \draw [-, shorten <= -2.50pt, shorten >= 2.50pt] (n5) to  (n7);
                \draw [dotted, -, shorten <= -2.50pt, shorten >= -2.50pt] (n4) to  (n5);
                \foreach \x/\y in {1/3, 2/3, 6/7, 7/8}
                \draw [-, shorten <= 2.50pt, shorten >= 2.50pt] (n\x) to  (n\y);
                \draw [-, shorten <= 2.50pt, shorten >= -2.50pt] (n3) to  (n4);
\end{tikzpicture}\tikzexternalenable}

\newcommand{\grapheREsix}[2]{\tikzexternaldisable\begin{tikzpicture}[label distance=-2pt, xscale=#1, yscale=#2]
                \node[circle, fill=black, inner sep=0pt, minimum size=3.5pt, label=above:\mscript{1}] (n1) at (0  , 0  ) {};
                \node[circle, fill=black, inner sep=0pt, minimum size=3.5pt, label=above:\mscript{2}] (n2) at (1  , 0  ) {};
                \node[circle, fill=black, inner sep=0pt, minimum size=3.5pt, label=above right:\mscript{3}] (n3) at (2  , 0  ) {};
                \node[circle, fill=black, inner sep=0pt, minimum size=3.5pt, label=right:\mscript{4}] (n4) at (2  , 0.6) {};
                \node[circle, fill=black, inner sep=0pt, minimum size=3.5pt, label=right:\mscript{5}] (n5) at (2  , 1.2) {};
                \node[circle, fill=black, inner sep=0pt, minimum size=3.5pt, label=above:\mscript{6}] (n6) at (3  , 0  ) {};
                \node[circle, fill=black, inner sep=0pt, minimum size=3.5pt, label=above:\mscript{7}] (n7) at (4  , 0  ) {};
                \foreach \x/\y in {1/2, 2/3, 3/6, 4/3, 5/4, 6/7}
                \draw [-, shorten <= 2.50pt, shorten >= 2.50pt] (n\x) to  (n\y);
\end{tikzpicture}\tikzexternalenable}

\newcommand{\grapheREseven}[2]{\tikzexternaldisable\begin{tikzpicture}[label distance=-2pt, xscale=#1, yscale=#2]
                \node[circle, fill=black, inner sep=0pt, minimum size=3.5pt, label=above:\mscript{1}] (n1) at (0  , 0  ) {};
                \node[circle, fill=black, inner sep=0pt, minimum size=3.5pt, label=above:\mscript{2}] (n2) at (1  , 0  ) {};
                \node[circle, fill=black, inner sep=0pt, minimum size=3.5pt, label=above right:\mscript{3}] (n3) at (2  , 0  ) {};
                \node[circle, fill=black, inner sep=0pt, minimum size=3.5pt, label=above right:\mscript{4}] (n4) at (3  , 0  ) {};
                \node[circle, fill=black, inner sep=0pt, minimum size=3.5pt, label=right:\mscript{5}] (n5) at (3  , 0.6) {};
                \node[circle, fill=black, inner sep=0pt, minimum size=3.5pt, label=above:\mscript{6}] (n6) at (4  , 0  ) {};
                \node[circle, fill=black, inner sep=0pt, minimum size=3.5pt, label=above:\mscript{7}] (n7) at (5  , 0  ) {};
                \node[circle, fill=black, inner sep=0pt, minimum size=3.5pt, label=above:\mscript{8}] (n8) at (6  , 0  ) {};
                \foreach \x/\y in {1/2, 2/3, 3/4, 4/6, 5/4, 6/7, 7/8}
                \draw [-, shorten <= 2.50pt, shorten >= 2.50pt] (n\x) to  (n\y);
\end{tikzpicture}\tikzexternalenable}

\newcommand{\grapheREeight}[2]{\tikzexternaldisable\begin{tikzpicture}[label distance=-2pt, xscale=#1, yscale=#2]
                \node[circle, fill=black, inner sep=0pt, minimum size=3.5pt, label=above:\mscript{1}] (n1) at (0  , 0  ) {};
                \node[circle, fill=black, inner sep=0pt, minimum size=3.5pt, label=above:\mscript{2}] (n2) at (1  , 0  ) {};
                \node[circle, fill=black, inner sep=0pt, minimum size=3.5pt, label=above right:\mscript{3}] (n3) at (2  , 0  ) {};
                \node[circle, fill=black, inner sep=0pt, minimum size=3.5pt, label=right:\mscript{4}] (n4) at (2  , 0.6) {};
                \node[circle, fill=black, inner sep=0pt, minimum size=3.5pt, label=above:\mscript{5}] (n5) at (3  , 0  ) {};
                \node[circle, fill=black, inner sep=0pt, minimum size=3.5pt, label=above:\mscript{6}] (n6) at (4  , 0  ) {};
                \node[circle, fill=black, inner sep=0pt, minimum size=3.5pt, label=above:\mscript{7}] (n7) at (5  , 0  ) {};
                \node[circle, fill=black, inner sep=0pt, minimum size=3.5pt, label=above:\mscript{8}] (n8) at (6  , 0  ) {};
                \node[circle, fill=black, inner sep=0pt, minimum size=3.5pt, label=above:\mscript{9}] (n9) at (7  , 0  ) {};
                \foreach \x/\y in {1/2, 2/3, 3/5, 4/3, 5/6, 6/7, 7/8, 8/9}
                \draw [-, shorten <= 2.50pt, shorten >= 2.50pt] (n\x) to  (n\y);
\end{tikzpicture}\tikzexternalenable}

\makeatletter
\let\c@figure\c@equation
\let\c@table\c@equation
\let\c@algorithm\c@equation
\let\ftype@table\ftype@figure 
\makeatother

\newtheorem{theorem}[equation]{Theorem}
\newtheorem{lemma}[equation]{Lemma}
\newtheorem{corollary}[equation]{Corollary}
\newtheorem{proposition}[equation]{Proposition}
\newtheorem{conjecture}[equation]{Conjecture}
\newtheorem{open}[equation]{Open problem}
\newtheorem{fact}[equation]{Fact}
\theoremstyle{definition}
\newtheorem{definition}[equation]{Definition}

\newtheorem{remark}[equation]{Remark}
\newtheorem{example}[equation]{Example}

\numberwithin{equation}{section}
\numberwithin{table}{section}
\numberwithin{figure}{section}

\title{Structure of non-negative posets of Dynkin type $\AA_n$}

\author{Marcin G\k{a}siorek\\
\small Faculty of Mathematics and Computer Science\\[-0.8ex]
\small Nicolaus Copernicus University\\[-0.8ex] 
\small ul. Chopina 12/18, 87-100 Toru\'n, Poland\\
\small\tt mgasiorek@mat.umk.pl}

\begin{document}

\maketitle

\begin{abstract}
A poset $I=(\{1,\ldots, n\}, \leq_I)$  is called \textit{non-negative} if the symmetric Gram matrix $G_I:=\frac{1}{2}(C_I + C_I^{tr})\in\MM_n(\RR)$ is positive semi-definite, where $C_I\in\MM_n(\ZZ)$ is the $(0,1)$-matrix encoding the relation $\leq_I$. Every such a connected poset $I$, up to the $\ZZ$-congruence of the $G_I$ matrix, is determined by a unique simply-laced Dynkin diagram $\Dyn_I\in\{\AA_m, \DD_m,\EE_6,\EE_7,\EE_8\}$.
We show that $\Dyn_I=\AA_m$ implies that the matrix $G_I$ is of rank  $n$ or $n-1$. Moreover, we  depict explicit shapes of Hasse digraphs $\CH(I)$ of all such posets~$I$  and devise formulae for their number.\medskip

\end{abstract}

\section{Introduction}\label{sec:intro}
By a finite partially ordered set (\textit{poset}) $I$ of \textit{size} $n$ we mean a pair $I=(V, \leq_I)$, where $V\eqdef \{1,\ldots, n\}$ and \(\leq_I\,\subseteq V\times V\) is a reflexive, antisymmetric and transitive binary relation. Every poset $I$ is uniquely determined by its \textit{incidence matrix} 
\begin{equation}\label{df:incmat}
C_{I} = [c_{ij}] \in\MM_{n}(\ZZ),\textnormal{ where } c_{ij} = 1 \textnormal{ if } i \leq_I j\textnormal{ and } c_{ij} = 0\textnormal{ otherwise},
\end{equation}
i.e., a  square $(0,1)$-matrix that encodes the relation \(\leq_I\). It is known that various mathematical classification problems can be solved by a reduction to the classification of indecomposable $K$-linear representations ($K$~is a field) of finite digraphs or matrix representations of finite posets, see~\cite{Si92}. Inspired by these results, here we study posets that are non-negative in the following sense. 
A poset $I$\ is defined to be \textit{non-negative} of \textit{rank $m$} if its \textit{symmetric Gram matrix} $G_I\eqdef\tfrac{1}{2}(C_I+C_I^{tr})\in\MM_n(\RR)$ is positive semi-definite of rank~$m$. 

Non-negative posets are classified by means of signed simple graphs as follows. One associates with a poset $I=(V, \leq_I)$\ the signed graph $\Delta_I=(V,E,\sgn)$ with the set  of edges  $E=\{\{i,j\};\ i<_I j \textnormal{ or } j <_I i\}$ and the sign function $\sgn(e)\eqdef1$ for every edge (i.e., signed graph with \textit{positive} edges only), see~\cite{SimZaj_intmms} and \Cref{rmk:graphbigraph}.
In particular, $I$ is called connected, if $\Delta_I$ is connected. We note that $\Delta_I$ is uniquely determined by its adjacency matrix $\Ad_{\Delta_I}\eqdef 2(G_I-\mathrm{id}_n)$, where $\mathrm{id}_n\in\MM_n(\ZZ)$ is an identity matrix.
Analogously as in the case of posets, a signed graph $\Delta$ is defined to be \textit{non-negative} of rank $m$ if its \textit{symmetric Gram matrix} $G_\Delta\eqdef \frac{1}{2}\Ad_\Delta + \mathrm{id}_n$ is positive semi-definite of rank $m$. Following \cite{simsonCoxeterGramClassificationPositive2013}, we call two signed graphs  $\Delta_1$ and $\Delta_2$  \textit{weakly Gram $\ZZ$-congruent} if $G_{\Delta_1}$ and $G_{\Delta_2}$ are \textit{$\ZZ$-congruent}, i.e., $G_{\Delta_2}=B^{tr}G_{\Delta_1}B$ for some $B\in\Gl_n(\ZZ)\eqdef\{A\in\MM_n(\ZZ);\,\det A=\pm 1\}$. It is easy to check that this relation preserves  definiteness and rank. 

We recall from \cite{simsonSymbolicAlgorithmsComputing2016} and~\cite{zajacStructureLoopfreeNonnegative2019} that every connected non-negative signed simple graph $\Delta$ of rank $m=n-r$ is weakly Gram $\ZZ$-congruent with the canonical $r$-vertex extension of  simply laced Dynkin diagram $\Dyn_\Delta \in \{\AA_m,\ab \DD_m,\ab \EE_6,\ab \EE_7,\ab \EE_8\}$, called the \textit{Dynkin type} of~$\Delta$. In particular, every \textit{positive} (i.e.,~of rank~$n$) connected  $\Delta$ is weakly Gram $\ZZ$-congruent with a unique simply-laced Dynkin diagram $\Dyn_\Delta$ of Table \ref{tbl:Dynkin_diagrams}.

\begin{longtable}{@{}r@{\,}l@{\,}l@{\quad}r@{\,}l@{}}
    $\AA_n\colon$ & \grapheAn{0.80}{1} & $\scriptstyle (n\geq 1);$\\[0.2cm]
    $\DD_n\colon$ & \grapheDn{0.80}{1} & $\scriptstyle (n\geq 1);$ & $\EE_6\colon$ & \grapheEsix{0.80}{1}\\[0.2cm]
    $\EE_7\colon$ & \grapheEseven{0.80}{1} &  & $\EE_8\colon$ & \grapheEeight{0.80}{1}\\[0.2cm]
    \caption{Simply-laced Dynkin diagrams}\label{tbl:Dynkin_diagrams}
\end{longtable}
\noindent Analogously, every \textit{principal} (i.e.,~of rank~$n-1$) connected bigraph $\Delta$ is weakly Gram $\ZZ$-congruent with $\widetilde{\mathrm{D}}\mathrm{yn}_\Delta \in \{\widetilde{\AA}_n,\ab \widetilde{\DD}_n,\ab \widetilde{\EE}_6,\ab \widetilde{\EE}_7,\ab \widetilde{\EE}_8\}$ diagram of Table \ref{tbl:Euklid_diag}, which is a one point extension of a diagram of \Cref{tbl:Dynkin_diagrams}.

\begin{longtable}{@{}r@{$\colon$}l@{\ \ \ }r@{$\colon$}l@{}}
    $\widetilde{\AA}_n$ & \grapheRAn{0.80}{1}\ {$\scriptstyle (n\geq 1)$;}\vspace{-0.3cm} \\
    $\widetilde{\DD}_n$ & \grapheRDn{0.80}{1}\ {$\scriptstyle (n\geq 4)$;} & $\widetilde{\EE}_6$ & \grapheREsix{0.80}{1} \\
    $\widetilde{\EE}_7$ & \grapheREseven{0.80}{1} &  $\widetilde{\EE}_8$ & \grapheREeight{0.80}{1}\\[0.2cm]
    \caption{Simply-laced Euclidean diagrams}\label{tbl:Euklid_diag}
\end{longtable}\vspace*{-2ex}

\begin{remark}\label{rmk:graphbigraph}
We are using the following notations, see \cite{barotQuadraticFormsCombinatorics2019,simsonCoxeterGramClassificationPositive2013,simsonSymbolicAlgorithmsComputing2016,SimZaj_intmms}.
\begin{enumerate}[label={\textnormal{(\alph*)}},wide]
 \item\label{rmk:graphbigraph:graphasbigraph} 
A simple graph $G=(V,E)$ is viewed as the signed graph $\Delta_G=(V,E,\sgn)$  with a sign function $\sgn(e)\eqdef-1$ for every $e\in E$, i.e., signed graph with \textit{negative} edges only.
\item\label{rmk:graphbigraph:bigraphdraw} We denote \textit{positive} edges  by dotted lines and \textit{negative} as full~ones, see~\cite{barotQuadraticFormsCombinatorics2019,simsonCoxeterGramClassificationPositive2013}.
\end{enumerate}
\end{remark}

By setting $\Dyn_I\eqdef \Dyn_{\Delta_I} $ one associates a Dynkin diagram with an arbitrary connected  non-negative poset~$I$.

In the present work, we give a complete description of connected non-negative posets $I=(V,\leq_I)$ of Dynkin type $\Dyn_I=\AA_m$ in terms of their \textit{Hasse digraphs} $\CH(I)$, where $\CH(I)$ is  the transitive reduction of the acyclic digraph  $\CD(I)=(V, A_I)$, with  $i\to j\in A_I$ iff  $i<_I j$ (see also Definition~\ref{df:hassedigraph}). The main result of the manuscript is the following theorem that establishes the correspondence between combinatorial and algebraic properties of non-negative posets of Dynkin type $\AA_m$.\pagebreak

\begin{theorem}\label{thm:a:main}
Assume that $I$ is a connected poset of size $n$ and $\CH(I)$ is its Hasse digraph.
\begin{enumerate}[label=\normalfont{(\alph*)}]
        \item\label{thm:a:main:posit} $I$ is non-negative of Dynkin type $\Dyn_I=\AA_n$ if and only if
        $\ov \CH(I)$ is a path graph.
        \item\label{thm:a:main:princ} $I$ is non-negative of Dynkin type $\Dyn_I=\AA_{n-1}$ if and only if $\ov\CH(I)$ is a cycle graph and $\CH(I)$ has at least two sinks.
        \item\label{thm:a:main:crkbiggeri} If $I$ is non-negative of Dynkin type $\Dyn_I=\AA_{m}$, then $m\in \{n,n-1\}$.

\end{enumerate}
\end{theorem}

In particular, we confirm Conjecture 6.4 stated in~\cite{gasiorekAlgorithmicCoxeterSpectral2020} by showing that in the case of  connected non-negative posets of Dynkin type $\AA_m$, there is a one-to-one correspondence between positive  posets and connected digraphs whose underlying graph is a path. We give a similar description of principal posets: there is a one-to-one correspondence between such posets and connected digraphs with at least two sinks, whose underlying graph is a cycle. We show that this characterization is complete: there are no connected non-negative posets of rank $m<n-1$.

Moreover, using the results of Theorem~\ref{thm:a:main}, we devise a formula for the number of all, up to isomorphism, connected non-negative posets of Dynkin type $\AA_m$.
\begin{theorem}\label{thm:typeanum}
        Let $Nneg(n,\AA)$ be the number of all non-negative posets $I$ of size $n\geq1$ and Dynkin type $\Dyn_I=\AA_{m}$. Then
        \begin{equation}\label{thm:typeanum:eq}
                Nneg(n, \AA)=
                \frac{1}{2n} \sum_{d\mid n}\big(2^{\frac{n}{d}}\varphi(d)\big) + 
\big\lfloor 2^{n - 2} + 2^{\lceil\frac{n}{2}-2\rceil} - \tfrac{n+1}{2}\big\rfloor,
        \end{equation}
        where $\varphi$ is Euler's totient function.
\end{theorem}

\section{Preliminaries}

Throughout, we mainly use the terminology and notation introduced in~\cite{gasiorekOnepeakPosetsPositive2012,gasiorekAlgorithmicStudyNonnegative2015,Si92,SimZaj_intmms} (in regard to posets), \cite{barotQuadraticFormsCombinatorics2019,simsonIncidenceCoalgebrasIntervally2009,simsonCoxeterGramClassificationPositive2013,simsonSymbolicAlgorithmsComputing2016} (quadratic forms), and~\cite{diestelGraphTheory2017} (graph theory). In particular, by $\NN\subseteq\ZZ\subseteq\RR$ we denote the set of non-negative integers, the ring of integers, and the real number fields, respectively.  We use a row notation for vectors $v=[v_1,\ldots,v_n]$ and write $v^{tr}$ to denote column vectors.

Two (directed) graphs $G=(V,E)$ and $G'=(V',E')$ are called \textit{isomorphic} $G\simeq G'$ if there exist a bijection $f\colon V\to V'$ that preserves (arcs) edges. By degree $\deg_G(v)$ of a vertex $v\in V$  we mean the number of edges incident with $v$. We call $G$ a \textit{path graph} if $|V|=1$ and $|E|=0$ or \mbox{$G\simeq\,P(u,v)\eqdef \,  u\scriptstyle \bullet\,\rule[1.5pt]{22pt}{0.4pt}\,\bullet\,\rule[1.5pt]{22pt}{0.4pt}\,\,\hdashrule[1.5pt]{12pt}{0.4pt}{1pt}\,\rule[1.5pt]{22pt}{0.4pt}\,\bullet \displaystyle v$}, where $u\neq v$. We call $G$ a \textit{cycle}, if $|G|\geq 3$ and $G\simeq\!\!\!\!\smash{%
\tikzsetnextfilename{cyclepic_graph}\begin{tikzpicture}[baseline=(n11.base),label distance=-2pt,xscale=0.65, yscale=0.74]
        \node[circle, fill=black, inner sep=0pt, minimum size=3.5pt, label={[name=n11]left:\phantom{$|$}}] (n1) at (0, 0  ) {};
        \node[circle, fill=black, inner sep=0pt, minimum size=3.5pt] (n2) at (1, 0  ) {};
        \node (n3) at (2, 0  ) {$ $};
        \node[circle, fill=black, inner sep=0pt, minimum size=3.5pt] (n4) at (5, 0  ) {};
        \node (n5) at (3, 0  ) {$ $};
        \node[circle, fill=black, inner sep=0pt, minimum size=3.5pt] (n6) at (4, 0  ) {};
        \draw[shorten <= 2.50pt, shorten >= 2.50pt] (n1) .. controls (0.2,0.45) and (4.8,0.45) .. (n4);
        \draw [line width=1.2pt, line cap=round, dash pattern=on 0pt off 5\pgflinewidth, -, shorten <= -2.50pt, shorten >= -2.50pt] (n3) to  (n5);
        \foreach \x/\y in {1/2, 6/4}
        \draw [-, shorten <= 2.50pt, shorten >= 2.50pt] (n\x) to  (n\y);
        \draw [-, shorten <= 2.50pt, shorten >= -2.50pt] (n2) to  (n3);
        \draw [-, shorten <= -2.50pt, shorten >= 2.50pt] (n5) to  (n6);
\end{tikzpicture}}$. 
A graph $G$ is \textit{connected} if $P(u, v)\subseteq G$ for every $u\neq v\in V$. By \textit{underlying graph} $\ov \CD$  of a digraph $\CD$ we mean a graph obtained from $\CD$  by forgetting the orientation of its arcs. A digraph $\CD$ is connected if  $\ov \CD$ is connected. A connected graph $G$  is called a \textit{tree} if $G$ does not contain any cycle. A digraph $\CD$ is called \textit{acyclic} if it contains no induced subdigraph isomorphic to\!\!\!\! \smash{%
\tikzsetnextfilename{cyclepic_orient}\begin{tikzpicture}[baseline=(n11.base),label distance=-2pt,xscale=0.65, yscale=0.74]
        \node[circle, fill=black, inner sep=0pt, minimum size=3.5pt, label={[name=n11]left:\phantom{$|$}}] (n1) at (0  , 0  ) {};
        \node[circle, fill=black, inner sep=0pt, minimum size=3.5pt] (n2) at (1  , 0  ) {};
        \node (n3) at (2  , 0  ) {$ $};
        \node[circle, fill=black, inner sep=0pt, minimum size=3.5pt] (n4) at (5  , 0  ) {};
        \node (n5) at (3  , 0  ) {$ $};
        \node[circle, fill=black, inner sep=0pt, minimum size=3.5pt] (n6) at (4  , 0  ) {};
        \draw[<-,shorten <= 2.50pt, shorten >= 3.50pt] (n1) .. controls (0.8,0.4) and (4.,0.4) .. (n4);
        \draw [line width=1.2pt, shorten <= 2.50pt, line cap=round, dash pattern=on 0pt off 5\pgflinewidth, -, shorten <= .50pt, shorten >= -4.50pt] (n3) to  (n5);
        \foreach \x/\y in {1/2, 6/4}
        \draw [->, shorten <= 2.50pt, shorten >= 2.50pt] (n\x) to  (n\y);
        \draw [->, shorten <= 2.50pt, shorten >= -2.50pt] (n2) to  (n3);
        \draw [->, shorten <= -2.50pt, shorten >= 2.50pt] (n5) to  (n6);
\end{tikzpicture}}.
We call a vertex $v$ of a digraph $\CD=(V, A)$ a \textit{source} (minimum) if it is not a target of any arc $\alpha\in A$. Analogously, we call $v\in \CD$ a \textit{sink} (maximum) if it is not a source of any arc.\smallskip

Given two elements $x,y\in V$ of a poset $I=(V,\leq_I)$ we write:
\begin{itemize}
\item $x <_I y$ when $x\leq_I y$ and $x\neq y$;
\item $x\lessdot_I y$ when $y$ covers $x$, i.e., $x<_I y$ and there is no such $z\in V,$ that $x<_I z<_I y$. 
\end{itemize}
Moreover, by $N_I(x)\eqdef\{z\in I;\, x \lessdot_{I} z\}\cup \{z\in I;\, z \lessdot_{I} x\}$ we denote the set of elements that either cover $x$ or are covered by $x$ in $I$.
\begin{definition}\label{df:hassedigraph}
\textit{Hasse digraph}  of a poset $I=(V,\leq_I)$ is a simple directed graph $\CH(I)=(V,A)$ with the set of arcs defined as follows: $x\to y\in A$ iff $x\lessdot_I y$.
\end{definition}

We call a  poset $I$ \textit{connected} if the graph $\ov \CH(I)\eqdef \ov{\CH(I)}$ is connected (equivalently, when the signed graph $\Delta_I$ is connected), and we note that every minimal (maximal) element in $I$ corresponds to a source (sink) in the Hasse digraph $\CH(I)$. We say that $I$ is a \textit{one-peak poset} if $I$ has exactly one maximal element.  Every finite acyclic digraph $\CD=(V,A)$ defines the poset $I_\CD\eqdef(V, {\leq_\CD)}$, where \mbox{$a \leq_\CD b$} if $a=b$ or there is an oriented path $\vec P(a,b)\eqdef\, a\scriptstyle \bullet \raisebox{-1.5pt}{\parbox{25pt}{\rightarrowfill}} \bullet \raisebox{-1.5pt}{\parbox{25pt}{\rightarrowfill}} \,\hdashrule[1.5pt]{12pt}{0.4pt}{1pt} \raisebox{-1.5pt}{\parbox{25pt}{\rightarrowfill}} \bullet \displaystyle b\subseteq \CD$. We note that $\CH(I_\CD)\neq \CD$ in general, see Example~\ref{ex:definitions}. By $I^{(k_1,\ldots,k_r)}\eqdef I\setminus \{k_1,\ldots,k_r\}$ we denote the induced subposet of $I$ whose set of elements equals $\{1,\ldots,n\}\setminus\{k_1,\ldots,k_r\}$.\smallskip

Following~\cite{gasiorekOnepeakPosetsPositive2012,simsonCoxeterGramClassificationPositive2013,SimZaj_intmms}, we associate with a poset $I$ of size $n$:
\begin{itemize}
\item the unit quadratic form $q_I\colon\ZZ^n\to\ZZ$ defined by the formula
\begin{equation}\label{eq:quadratic_form}
 q_I(x):=\sum_{\mathclap{i\in\{1,\ldots,n\}}}x_i^2 + \sum_{\mathclap{i\,<_I\, j}}x_i x_j = v\cdot G_I\cdot v^{tr},
\end{equation}
\item and its kernel 
\begin{equation}\label{eq:kernel}
\Ker q_I \eqdef \{v \in\ZZ^n;\ q_I(v)=0\}\subseteq\ZZ^n,
\end{equation}
\end{itemize}
where $G_I\in\MM_n(\ZZ)$ is the symmetric Gram matrix of $I$. It is known that a poset $I$ is non-negative of rank $m$ if and only if the quadratic form $q_I$ is positive semi-definite (i.e., $q_I(v)\geq 0$ for every $v\in\ZZ^n$) and its kernel $\Ker q_I\subseteq \ZZ^n$ is a free abelian subgroup of rank $n-m$, see~\cite{simsonCoxeterGramClassificationPositive2013}. We call a non-negative poset $I$  \textit{positive} if $m=n$, \textit{principal} if $m=n-1$, and \textit{indefinite} if its symmetric Gram matrix $G_I$ is not positive/negative semidefinite.\smallskip

\begin{remark}\label{rmk:indef}
A poset $I$ is indefinite if and only if there exist such vectors $v,w\in\ZZ^n$ that $q_I(v)>0$ and $q_I(w)<0$. Since $q_I([1,0,\ldots,0])=1>0$ for every poset $I$, to show that given $I$ is indefinite, it suffices to show that $q_I(w)<0$ for some $w\in\ZZ^n$.
\end{remark}

\begin{example}\label{ex:definitions}
To illustrate the definitions, we consider digraph
$\CD=(\{1,2,3,4\}, \{{2\to 1},\ab 2\to 3, 2\to 4,\ab 1\to 3, 4\to 3\})$. 

We have $I_\CD=(\{1,2,3,4\}, \{2 \leq_\CD 1, 2 \leq_\CD 3, 2 \leq_\CD 4, 1 \leq_\CD 3, 4 \leq_\CD 3\})$,
{\newcommand{\mnum}[1]{\vbox to 11pt {\vfil\hbox to 9pt{\hfill$\scriptstyle #1$}\vfil }}\begin{align*}
\CD&=\tikzsetnextfilename{expic_1}\begin{tikzpicture}[baseline={([yshift=-2.75pt]current bounding box)},label distance=-2pt,xscale=0.66, yscale=0.6]
        \node[circle, fill=black, inner sep=0pt, minimum size=3.5pt, label=left:$\scriptscriptstyle 1$] (n1) at (1, 2  ) {};
        \node[circle, fill=black, inner sep=0pt, minimum size=3.5pt, label=above:$\scriptscriptstyle 2$] (n2) at (0, 1  ) {};
        \node[circle, fill=black, inner sep=0pt, minimum size=3.5pt, label=above:$\scriptscriptstyle 3$] (n3) at (2, 1  ) {};
        \node[circle, fill=black, inner sep=0pt, minimum size=3.5pt, label=right:$\scriptscriptstyle 4$] (n4) at (1, 0  ) {};
        \foreach \x/\y in {1/3, 2/1, 2/4, 4/3}
        \draw [-stealth, shorten <= 2.50pt, shorten >= 2.50pt] (n\x) to  (n\y);
        \draw [-stealth, shorten <= 2.00pt, shorten >= 2.10pt] (n2) to  (n3);
\end{tikzpicture},           &  \CH(I_\CD)&=\tikzsetnextfilename{expic_2}\begin{tikzpicture}[baseline={([yshift=-2.75pt]current bounding box)},label distance=-2pt,
        xscale=0.66, yscale=0.6]
        \node[circle, fill=black, inner sep=0pt, minimum size=3.5pt, label=left:$\scriptscriptstyle 1$] (n1) at (1, 2  ) {};
        \node[circle, fill=black, inner sep=0pt, minimum size=3.5pt, label=above:$\scriptscriptstyle 2$] (n2) at (0, 1  ) {};
        \node[circle, fill=black, inner sep=0pt, minimum size=3.5pt, label=above:$\scriptscriptstyle 3$] (n3) at (2, 1  ) {};
        \node[circle, fill=black, inner sep=0pt, minimum size=3.5pt, label=right:$\scriptscriptstyle 4$] (n4) at (1, 0  ) {};
        \foreach \x/\y in {1/3, 2/1, 2/4, 4/3}
        \draw [-stealth, shorten <= 2.50pt, shorten >= 2.50pt] (n\x) to  (n\y);
\end{tikzpicture},              &  \Delta_{I_\CD}&=
\tikzsetnextfilename{expic_3}\begin{tikzpicture}[baseline={([yshift=-2.75pt]current bounding box)},label distance=-2pt,
        xscale=0.66, yscale=0.6]
        \node[circle, fill=black, inner sep=0pt, minimum size=3.5pt, label=left:$\scriptscriptstyle 1$] (n1) at (1, 2  ) {};
        \node[circle, fill=black, inner sep=0pt, minimum size=3.5pt, label=above:$\scriptscriptstyle 2$] (n2) at (0, 1  ) {};
        \node[circle, fill=black, inner sep=0pt, minimum size=3.5pt, label=above:$\scriptscriptstyle 3$] (n3) at (2, 1  ) {};
        \node[circle, fill=black, inner sep=0pt, minimum size=3.5pt, label=right:$\scriptscriptstyle 4$] (n4) at (1, 0  ) {};
        \foreach \x/\y in {1/3, 2/1, 2/3, 2/4, 4/3}
        \draw [-, densely dashed, shorten <= 2.50pt, shorten >= 2.50pt] (n\x) to  (n\y);
\end{tikzpicture},\\
G_{I_\CD}&=\begin{bsmallmatrix*}[r]
        \mnum{1} & \mnum{\frac{1}{2}} & \mnum{\frac{1}{2}} & \mnum{0}\\
        \mnum{\frac{1}{2}} & \mnum{1} & \mnum{\frac{1}{2}} & \mnum{\frac{1}{2}}\\
        \mnum{\frac{1}{2}} & \mnum{\frac{1}{2}} & \mnum{1} & \mnum{\frac{1}{2}}\\
        \mnum{0} & \mnum{\frac{1}{2}} & \mnum{\frac{1}{2}} & \mnum{1}
\end{bsmallmatrix*}=G_{\Delta_{I_\CD}},         &  C_{I_\CD}&=\begin{bsmallmatrix*}[r]
        \mnum{1} & \mnum{0} & \mnum{1} & \mnum{0}\\
        \mnum{1} & \mnum{1} & \mnum{1} & \mnum{1}\\
        \mnum{0} & \mnum{0} & \mnum{1} & \mnum{0}\\
        \mnum{0} & \mnum{0} & \mnum{1} & \mnum{1}
\end{bsmallmatrix*},   &  \Ad_{I_\CD}&=\begin{bsmallmatrix*}[r]
        \mnum{0} & \mnum{1} & \mnum{1} & \mnum{0}\\
        \mnum{1} & \mnum{0} & \mnum{1} & \mnum{1}\\
        \mnum{1} & \mnum{1} & \mnum{0} & \mnum{1}\\
        \mnum{0} & \mnum{1} & \mnum{1} & \mnum{0}
\end{bsmallmatrix*},
\end{align*}}%
where we denote the positive edges of a signed graph by dotted lines. Moreover:
\begin{itemize}
        \item $2$ is minimal in $I_\CD$ and $3$ is maximal in $I_\CD$, equivalently: $2$ is a source in $\CH(I_\CD)$ and $3$ is a sink in $\CH(I_\CD)$,
        \item $q_{I_\CD}(x)= x_{1}^{2} + x_{2}^{2} + x_{3}^{2} + x_{4}^{2} +  
        x_{1}x_{2} +  x_{2}x_{3} + x_{3}x_{4}  + x_{1}x_{3} + x_{2}x_{4}=
        \left(x_{1} \!+\! \tfrac{1}{2}x_{2} \!+\! \tfrac{1}{2}x_{3}\right)^{2} \!+\!
        \tfrac{3}{4}\! \left(\tfrac{1}{3}x_{2} \!+\! x_{3} \!+\! \frac{2}{3} x_{4}\right)^{2} \!+\!
        \tfrac{2}{3}\! \left(x_{2} \!+\! \tfrac{1}{2}x_{4}\right)^{2} \!+\! \tfrac{1}{2}x_{4}^{2}$,
        \item $\Ker q_{I_\CD}=\{0\}\subseteq\ZZ^4$ and poset $I_\CD$  is non-negative of rank $4$, i.e., positive.
\end{itemize}\pagebreak
Since $G_{\DD_4}=B^{tr}G_{\Delta_{I_\CD}}B$, where

\begin{center}
\newcommand{\mnum}[1]{\vbox to 11pt {\vfil\hbox to 13pt{\hfill$\scriptstyle #1$}\vfil }}
$\DD_4=$
\tikzsetnextfilename{expic_4}\begin{tikzpicture}[baseline={([yshift=-2.75pt]current bounding box)},label distance=-2pt,
        xscale=0.66, yscale=0.6]
        \node[circle, fill=black, inner sep=0pt, minimum size=3.5pt, label=above:$\scriptscriptstyle 1$] (n1) at (0, 0  ) {};
        \node[circle, fill=black, inner sep=0pt, minimum size=3.5pt, label=right:$\scriptscriptstyle 2$] (n2) at (1, 1  ) {};
        \node[circle, fill=black, inner sep=0pt, minimum size=3.5pt, label=above right:$\scriptscriptstyle 3$] (n3) at (1, 0  ) {};
        \node[circle, fill=black, inner sep=0pt, minimum size=3.5pt, label=above:$\scriptscriptstyle 4$] (n4) at (2, 0  ) {};
        \foreach \x/\y in {1/3, 2/3,  3/4}
        \draw [-, shorten <= 2.50pt, shorten >= 2.50pt] (n\x) to  (n\y);
\end{tikzpicture},\quad
$G_{\DD_4}=\begin{bsmallmatrix*}[r]
        \mnum{1} & \mnum{0} & \mnum{-\frac{1}{2}} & \mnum{0}\\
        \mnum{0} & \mnum{1} & \mnum{-\frac{1}{2}} & \mnum{0}\\
        \mnum{-\frac{1}{2}} & \mnum{-\frac{1}{2}} & \mnum{1} & \mnum{-\frac{1}{2}}\\
        \mnum{0} & \mnum{0} & \mnum{-\frac{1}{2}} & \mnum{1}
\end{bsmallmatrix*}=\frac{1}{2}\Ad_{\DD_4} + \mathrm{id}_n$
and
$B=\begin{bsmallmatrix*}[r]
        \mnum{0} & \mnum{0} & \mnum{-1} & \mnum{0}\\
        \mnum{0} & \mnum{0} & \mnum{-1} & \mnum{0}\\
        \mnum{-1} & \mnum{-1} & \mnum{0} & \mnum{-1}\\
        \mnum{0} & \mnum{0} & \mnum{-1} & \mnum{0}
\end{bsmallmatrix*}$,
\end{center}
we conclude that $\Dyn_{I_\CD}=\DD_4$. To finish the example, we note that elements of the adjacency matrix $\Ad_{\DD_4}$ are negative because we view graph $\DD_4$ as a signed graph.
\end{example}

We recall from Section~\ref{sec:intro}, that the Dynkin type of a connected non-negative poset $I$ of size $n$ and rank $m$ is such a simply-laced Dynkin diagram $\Dyn_I\in\{\AA_m,\ab \DD_m,\ab \EE_6,\ab \EE_7,\ab \EE_8\},$ that the signed graph $\Delta_I$ is weakly Gram $\ZZ$-congruent with the canonical $r$-vertex extension of $\Dyn_I$, where $r=n-m$. Equivalently, Dynkin type  can  be defined without referring to canonical $r$-vertex extensions, see~\cite{gasiorekAlgorithmicStudyNonnegative2015,simsonSymbolicAlgorithmsComputing2016,zajacStructureLoopfreeNonnegative2019,zajacPolynomialTimeInflation2020} and \cite{barotQuadraticFormsCombinatorics2019}. Although such a definition is more technical, it yields better insight into the combinatorial structure of non-negative posets. First, we need the following fact.
\begin{fact}\label{fact:specialzbasis} 
Assume that $I$ is a connected non-negative poset of size $n$ and rank $m$, $r=n-m$, and $q_I\colon\ZZ^n\to\ZZ$~\eqref{eq:quadratic_form} is the quadratic form associated with $I$.
\begin{enumerate}[label=\normalfont{(\alph*)}]
    \item\label{fact:specialzbasis:existance} There exists such a basis $h^{k_1},\ldots, h^{k_r}$ of the free abelian group $\Ker q_I\subseteq \ZZ^n$, that $h^{k_i}_{k_i} = 1$ and $h^{k_i}_{k_j} = 0$ for $1 \leq i,j \leq r$ and $i \neq j$, where $1 \leq k_1 < \ldots  < k_r \leq n$.

    \item\label{fact:specialzbasis:subbigraph} $I^{(a_1,\ldots,a_s)}$ is a connected and non-negative poset of size $n-s$ and rank $m$, for every $\{a_1,\ldots,a_s\}\subseteq \{k_1,\ldots,k_r\}$ and $1\leq s\leq r$.
\end{enumerate}
\end{fact}
\begin{proof}
Apply \cite[Lemma 2.7]{zajacPolynomialTimeInflation2020} and \cite[Theorem 2.1]{zajacStructureLoopfreeNonnegative2019} to the bigraph $\Delta_I$.
\end{proof}

The Dynkin type of a connected non-negative poset $I$ is defined to be a Dynkin diagram $\Dyn_I \in \{\AA_m,\DD_m,\EE_6,\EE_7,\EE_8\}$ that determines $I$ uniquely, up to a weak Gram $\ZZ$-congruence.
\begin{definition}\label{df:Dynkin_type}
Assume that $I$ is a connected non-negative poset  of rank $m$ and size $n$. The Dynkin type $\Dyn_I  \in \{\AA_m,\DD_m,\EE_6,\EE_7,\EE_8\}$ is  the unique simply-laced Dynkin diagram, viewed as a signed graph, such that
\[
\Delta_{I'} \sim_\ZZ \Dyn_I,
\]
where $I'\eqdef I$ if $m=n$ and $I'\eqdef I^{(k_1,\ldots,k_r)}$ (see \Cref{fact:specialzbasis}\ref{fact:specialzbasis:subbigraph}) otherwise.\smallskip
\end{definition}

We note that Dynkin type $\Dyn_I$ can be calculated  using the inflation algorithm~\cite[Algorithm 3.1]{simsonCoxeterGramClassificationPositive2013} applied to the bigraph $\Delta_{I'}$.

{\tikzexternaldisable\newcommand{\mxs}{0.9}
        \begin{longtable}{@{}r@{\,}l@{\,}l@{\quad}r@{\,}l@{}}
                ${}_{p}\AA^*_{r}\colon$ & \grafcrkzA{\mxs}{1} &  & $\wh\DD^*_{p} \diamond \AA_{r-p}\colon$ & \grafcrkzDii{\mxs}{1}\\[.6cm]
                $\DD^*_{r}\colon$ & \grafcrkzDi{\mxs}{1} &  & ${}_{s}\DD^*_{p} \diamond \AA_{r-p}\colon$ & \grafcrkzDiii{0.7}{0.8}\\[-.1cm]
                \caption{One-peak positive posets of Dynkin type $\AA_{r+1}$ and $\DD_{r+1}$}\label{tbl:onepeak_posit}
\end{longtable}\tikzexternalenable}%

The aim of this manuscript is to give a full structural characterization of connected non-negative posets $I$ of Dynkin type $\Dyn_I=\AA_m$. We note that such a result is known in the case of one-peak positive and principal posets, see \cite[Theorem 5.2]{gasiorekOnepeakPosetsPositive2012} and \cite[Theorem 3.5]{gasiorekCoxeterTypeClassification2019}. In particular, we have the following.
\begin{theorem}\label{thm:onepeak_posit}
One-peak poset $I$\ of size $n$ is positive if and only if its Hasse digraph $\CH(I)$ is isomorphic~to:
\begin{enumerate}[label=\normalfont{(\alph*)}]  
        \item the digraph ${}_{p}\AA^*_{n-1}$ \textnormal{(}in this case the Dynkin type equals $\Dyn_I=\AA_{n}$\textnormal{);}
        \item one of the digraphs $\wh\DD^*_{p} \diamond \AA_{n-p-1}$, $\DD^*_{n-1}$ or ${}_{s}\DD^*_{p} \diamond \AA_{n-p-1}$ \textnormal{(}$\Dyn_I=\DD_{n}$\textnormal{);}  
        \item one of the digraphs
        $\PP_1,\ldots,\PP_{16}$  \textnormal{(}$\Dyn_I=\EE_{6}$\textnormal{)},
        $\PP_{17},\ldots,\PP_{72}$  \textnormal{(}$\Dyn_I=\EE_{7}$\textnormal{)},
        $\PP_{73},\ldots,\PP_{193}$  \textnormal{(}$\Dyn_I=\EE_{8}$\textnormal{)}         presented in \cite[Tables 6.1–6.3]{gasiorekOnepeakPosetsPositive2012}.\vspace*{-2ex}
\end{enumerate}
\end{theorem}
\section{Hanging path in a Hasse digraph}

In this section, we formalize a very useful observation  that changing the orientation of arcs on the ``hanging path'' in the Hasse digraph $\CH(I)$ does not change the definiteness nor rank of a poset $I$. Inspired by the ideas of~\cite[Prop. 16.15]{Si92} and~\cite{gasiorekOnepeakPosetsPositive2012}, we introduce the following definition.

\begin{definition}
Let $I_p\subseteq I$ be a connected subposet of a poset $I$ and $p$ be a point of $I_p$. The subposet $I_p$ is called \textit{$p$-anchored path} if:
\begin{enumerate}[label=\normalfont{(\alph*)}]  
    \item for every $a\in I_p^{(p)}$ we have $N_I(a)\subseteq I_p$ and $|N_I(a)|\in\{1,2\}$,
    \item  $|N_I(p)\cap I_p|=1$.
\end{enumerate} 
\end{definition}

The following picture illustrates the definition of a $p$-anchored path.
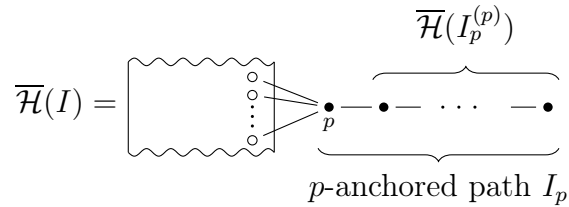
\begin{center}
$\ov\CH(I)=$
\tikzsetnextfilename{anchpth_1}\begin{tikzpicture}[baseline={([yshift=-2.75pt]current bounding box)},label distance=-2pt, xscale=0.48, yscale=0.4]
\coordinate (n1) at (4, 3  ) {};
\coordinate (n3) at (4, 0  ) {};
        \node[circle, draw, inner sep=0pt, minimum size=3.5pt] (n11) at (3.4  , 2.5  ) {};
        \node[circle, draw, inner sep=0pt, minimum size=3.5pt] (n2) at (3.4  , 1.9  ) {};
        \node[circle, draw, inner sep=0pt, minimum size=3.5pt] (n31) at (3.4  , 0.4  ) {};

        \node[circle, fill=black, inner sep=0pt, minimum size=3.5pt, label=below:$\scriptstyle p$] (n4) at (5.50, 1.50) {};

        \coordinate (n5) at (0, 3  ) {};
        \coordinate (n6) at (0, 0  ) {};
        \node[circle, fill=black, inner sep=0pt, minimum size=3.5pt] (n7) at (7, 1.50) {};
        \node (n8) at (8.50, 1.50) {};
        \node (n9) at (10, 1.50) {};
        \node[circle, fill=black, inner sep=0pt, minimum size=3.5pt] (n10) at (11.50, 1.50) {};
        \draw[decoration={complete sines,segment length=3mm, amplitude=1mm},decorate] (n1) -- (n5);
        \draw (n5) -- (n6);
        \draw (n1) -- (n3);
        \draw[decoration={complete sines,segment length=3mm, amplitude=1mm,mirror},decorate] (n3) -- (n6);
        \foreach \x/\y in {4/7, 7/8, 9/10}
        \draw [-, shorten <= 2.50pt, shorten >= 2.50pt] (n\x) to  (n\y);
        \foreach \x/\y in {11/4, 2/4, 31/4}
        \draw [-, shorten <= 2.50pt, shorten >= 2.50pt] (n\x) to  (n\y);
        \draw [line width=1.2pt, line cap=round, dash pattern=on 0pt off 3\pgflinewidth, -, shorten <= 2.50pt, shorten >= 2.50pt] (n2) to  (n31);
        \draw [line width=1.2pt, line cap=round, dash pattern=on 0pt off 5\pgflinewidth, -, shorten <= -2.50pt, shorten >= -2.50pt] (n8) to  (n9);
        \draw [decorate,decoration={brace,amplitude=5pt,mirror,raise=3ex}]
        (5.20, 1.50) -- (11.80, 1.50) node[midway,yshift=-6ex]{$p$-anchored path $I_p$};
        \draw [decorate,decoration={brace,amplitude=5pt,mirror,raise=2ex}]
        (11.80, 1.50) -- (6.70, 1.50) node[midway,yshift=6ex]{$\ov\CH(I^{(p)}_p)$};
\end{tikzpicture}
\end{center}

\begin{definition}
Let $I_p\subseteq I$ be a $p$-anchored path of a poset $I$. The $I_p$-reflection $\delta_{I_p} I$  is  the poset  defined  by the Hasse digraph $\CH(\delta_{I_p}I)$ obtained from $\CH(I)$ by reversing all the arcs in the subdigraph $\CH(I_p)\subseteq \CH(I)$.

\end{definition}
We call $I_p\subseteq I$  an \textit{inward} ($\mkern-2mu$\textit{outward}) $p$-anchored path if  $p\in I_p$ is a unique maximal (minimal) point in $I_p$. For example, given an outward $p$-anchored path $I_p\subseteq I$ consisting of $\{p,s_1,\dots,s_k\}$ elements, we have the following.

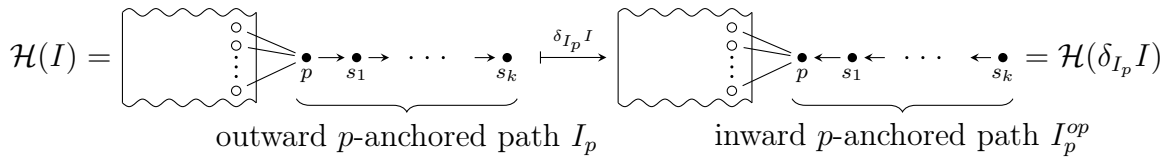
\begin{center}
$\CH(I)=$
\tikzsetnextfilename{anchpth_2}\begin{tikzpicture}[baseline=(n22.base),label distance=-2pt, xscale=0.44, yscale=0.4]
        \coordinate (n1) at (4, 3  ) {};
        \coordinate (n3) at (4, 0  ) {};
        \node[circle, draw, inner sep=0pt, minimum size=3.5pt] (n11) at (3.4  , 2.5  ) {};
        \node[circle, draw, inner sep=0pt, minimum size=3.5pt] (n2) at (3.4  , 1.9  ) {};
        \node[circle, draw, inner sep=0pt, minimum size=3.5pt] (n31) at (3.4  , 0.4  ) {};

        \node[circle, fill=black, inner sep=0pt, minimum size=3.5pt, label=below:$\scriptstyle p$] (n4) at (5.50, 1.50) {};
        \coordinate (n22) at (0, 1.2) {};
        \coordinate (n5) at (0, 3  ) {};
        \coordinate (n6) at (0, 0  ) {};
        \node[circle, fill=black, inner sep=0pt, minimum size=3.5pt, label=below:$\scriptstyle s_1$] (n7) at (7, 1.50) {};
        \node (n8) at (8.50, 1.50) {};
        \node (n9) at (10, 1.50) {};
        \node[circle, fill=black, inner sep=0pt, minimum size=3.5pt, label=below:$\scriptstyle s_k$] (n10) at (11.50, 1.50) {};
        \draw[decoration={complete sines,segment length=3mm, amplitude=1mm},decorate] (n1) -- (n5);
        \draw (n5) -- (n6);
        \draw (n1) -- (n3);
        \draw[decoration={complete sines,segment length=3mm, amplitude=1mm,mirror},decorate] (n3) -- (n6);
        \foreach \x/\y in {4/7, 7/8, 9/10}
        \draw [-stealth, shorten <= 2.50pt, shorten >= 2.50pt] (n\x) to  (n\y);
        \foreach \x/\y in {11/4, 2/4, 31/4}
        \draw [-, shorten <= 2.50pt, shorten >= 2.50pt] (n\x) to  (n\y);
        \draw [line width=1.2pt, line cap=round, dash pattern=on 0pt off 3\pgflinewidth, -, shorten <= 2.50pt, shorten >= 2.50pt] (n2) to  (n31);
        \draw [line width=1.2pt, line cap=round, dash pattern=on 0pt off 5\pgflinewidth, -, shorten <= -2.50pt, shorten >= -2.50pt] (n8) to  (n9);
\draw [decorate,decoration={brace,amplitude=5pt,mirror,raise=3ex}]
        (5.20, 1.50) -- (11.80, 1.50) node[midway,yshift=-6ex]{$\mathclap{\textnormal{outward }p\textnormal{-anchored path } I_p}$};
\end{tikzpicture}
\tikzsetnextfilename{pic_ar1}\begin{tikzpicture}[baseline={([yshift=-9.5pt]current bounding box)},
        label distance=-2pt,xscale=0.35, yscale=0.5]
        \coordinate (n1) at (0, 1.50);
        \coordinate (n2) at (2.50, 1.50);
        \draw [|-stealth] (n1) to  node[above=-2.0pt, pos=0.5] {$\scriptscriptstyle \delta_{I_p} I$} (n2);
\end{tikzpicture}
\tikzsetnextfilename{anchpth_3}\begin{tikzpicture}[baseline=(n22.base),label distance=-2pt, xscale=0.44, yscale=0.4]
        \coordinate (n1) at (4, 3  ) {};
        \coordinate (n3) at (4, 0  ) {};
        \node[circle, draw, inner sep=0pt, minimum size=3.5pt] (n11) at (3.4  , 2.5  ) {};
        \node[circle, draw, inner sep=0pt, minimum size=3.5pt] (n2) at (3.4  , 1.9  ) {};
        \node[circle, draw, inner sep=0pt, minimum size=3.5pt] (n31) at (3.4  , 0.4  ) {};

        \node[circle, fill=black, inner sep=0pt, minimum size=3.5pt, label=below:$\scriptstyle p$] (n4) at (5.50, 1.50) {};
        \coordinate (n22) at (0, 1.2) {};
        \coordinate (n5) at (0, 3  ) {};
        \coordinate (n6) at (0, 0  ) {};
        \node[circle, fill=black, inner sep=0pt, minimum size=3.5pt, label=below:$\scriptstyle s_1$] (n7) at (7, 1.50) {};
        \node (n8) at (8.50, 1.50) {};
        \node (n9) at (10, 1.50) {};
        \node[circle, fill=black, inner sep=0pt, minimum size=3.5pt, label=below:$\scriptstyle s_k$] (n10) at (11.50, 1.50) {};
        \draw[decoration={complete sines,segment length=3mm, amplitude=1mm},decorate] (n1) -- (n5);
        \draw (n5) -- (n6);
        \draw (n1) -- (n3);
        \draw[decoration={complete sines,segment length=3mm, amplitude=1mm,mirror},decorate] (n3) -- (n6);
        \foreach \x/\y in {4/7, 7/8, 9/10}
        \draw [stealth-, shorten <= 2.50pt, shorten >= 2.50pt] (n\x) to  (n\y);
        \foreach \x/\y in {11/4, 2/4, 31/4}
        \draw [-, shorten <= 2.50pt, shorten >= 2.50pt] (n\x) to  (n\y);
        \draw [line width=1.2pt, line cap=round, dash pattern=on 0pt off 3\pgflinewidth, -, shorten <= 2.50pt, shorten >= 2.50pt] (n2) to  (n31);
        \draw [line width=1.2pt, line cap=round, dash pattern=on 0pt off 5\pgflinewidth, -, shorten <= -2.50pt, shorten >= -2.50pt] (n8) to  (n9);
\draw [decorate,decoration={brace,amplitude=5pt,mirror,raise=3ex}]
        (5.20, 1.50) -- (11.80, 1.50) node[midway,yshift=-6ex]{$\mathclap{\textnormal{inward }p\textnormal{-anchored path } I_p^{op}}$};
\end{tikzpicture}$=\CH(\delta_{I_p} I)$
\end{center}

The $I_p$-reflection $I\mapsto\delta_{I_p} I$, a combinatorial operation defined at the digraph level, has an algebraic interpretation. We need one more definition from~\cite{simsonCoxeterGramClassificationPositive2013} to state it formally.
\begin{definition}
	Two finite posets $I$ and $J$ are called \textit{strong Gram $\ZZ$-congruent} $I\approx_\ZZ J$ if there exists such a matrix $B\in\Gl_n(\ZZ)$, that $C_J=B^{tr} C_I B.$
\end{definition}
It is straightforward to check that strong Gram $\ZZ$-congruence of posets implies a week one and the inverse implication is not true in general~\cite{gasiorekAlgorithmicStudyNonnegative2015}, although it holds in the case of one-peak positive~\cite{gasiorekOnepeakPosetsPositive2012} and principal~\cite{gasiorekCoxeterTypeClassification2019} posets. The reader is referred to~\cite{gasiorekCongruenceRationalMatrices2023} for a further discussion on $\ZZ$-congruence and its applications.\smallskip

\begin{lemma}\label{lemma:preflection}
If $I_p\subseteq I$ is an inward or outward $p$-anchored path, then $I\approx_\ZZ \delta_{I_p} I$. In par\-ticular, $I$ is non-negative of rank $m$ if and only if  $\delta_{I_p} I$ is non-negative of rank~$m$.
\end{lemma}
\begin{proof}
First, we note that the strong Gram $\ZZ$-congruence of posets implies a weak Gram $\ZZ$-congruence. Since congruent matrices have the same definiteness and rank, it suffices to show that $I\approx_\ZZ \delta_{I_p} I$.\smallskip

Let $I_p\subseteq I$ be an inward $p$-anchored path of a poset $I$ and let $J\eqdef\delta_{I_p} I$. Then:
\begin{itemize}
\item $J_p\eqdef I_p^{op}$ is an outward $p$-anchored path in $J$,
\item $I\setminus I^{(p)}_p = J\setminus J^{(p)}_p$. 
 \end{itemize}
 Without loss of generality, we may assume that Hasse digraphs  $\CH(I)$ and $\CH(J)$ have the forms
\begin{center}
$\CH(I)=$
\tikzsetnextfilename{anchpth_4}\begin{tikzpicture}[baseline=(n22.base),label distance=-2pt,xscale=0.48, yscale=0.4]
    \coordinate (n1) at (4, 3  ) {};
    \coordinate (n3) at (4, 0  ) {};
     \coordinate (n22) at (0, 1.2) {};
    \coordinate (n113) at (5.9, 1.5  ) {};
    \node[circle, draw, inner sep=0pt, minimum size=3.5pt] (n11) at (3.4  , 2.5  ) {};
    \node[circle, draw, inner sep=0pt, minimum size=3.5pt] (n2) at (3.4  , 1.9  ) {};
    \node[circle, draw, inner sep=0pt, minimum size=3.5pt] (n31) at (3.4  , 0.4  ) {};
    \node[circle, fill=black, inner sep=0pt, minimum size=3.5pt, label=below:$\scriptstyle \phantom{+}\!\!\!p$] (n4) at (5.50, 1.50) {};
    \coordinate (n5) at (0, 3  ) {};
    \coordinate (n6) at (0, 0  ) {};
    \node[circle, fill=black, inner sep=0pt, minimum size=3.5pt, label=below:$\scriptstyle p+1$] (n7) at (7, 1.50) {};
    \node (n8) at (8.50, 1.50) {};
    \node (n9) at (10, 1.50) {};
    \node[circle, fill=black, inner sep=0pt, minimum size=3.5pt, label={below:$\scriptstyle p+k-1=n$}] (n10) at (11.50, 1.50) {};
    \draw[decoration={complete sines,segment length=3mm, amplitude=1mm},decorate] (n1) -- (n5);
    \draw (n5) -- (n6);
    \draw[rounded corners] (n1) -- (n113) -- (n3);
    \draw[decoration={complete sines,segment length=3mm, amplitude=1mm},decorate] (n6) -- (n3);
    \path (n5) to  node[pos=0.5,xshift=-1ex] {{$I\setminus I^{(p)}_p$}} (n3);
    \foreach \x/\y in {4/7, 7/8, 9/10}
    \draw [stealth-, shorten <= 2.50pt, shorten >= 2.50pt] (n\x) to  (n\y);
    \foreach \x/\y in {11/4, 2/4, 31/4}
    \draw [-, shorten <= 2.50pt, shorten >= 2.50pt] (n\x) to  (n\y);
    \draw [line width=1.2pt, line cap=round, dash pattern=on 0pt off 3\pgflinewidth, -, shorten <= 2.50pt, shorten >= 2.50pt] (n2) to  (n31);
    \draw [line width=1.2pt, line cap=round, dash pattern=on 0pt off 5\pgflinewidth, -, shorten <= -2.50pt, shorten >= -2.50pt] (n8) to  (n9);
	\draw [decorate,decoration={brace,amplitude=5pt,mirror,raise=3ex}]
    (5.20, 1.50) -- (11.80, 1.50) node[midway,yshift=-6ex]{$I_p$};
\end{tikzpicture}
$\CH(J)=$
\tikzsetnextfilename{anchpth_5}\begin{tikzpicture}[baseline=(n22.base),label distance=-2pt,xscale=0.48, yscale=0.4]
    \coordinate (n1) at (4, 3  ) {};
    \coordinate (n3) at (4, 0  ) {};
     \coordinate (n22) at (0, 1.2) {};
    \coordinate (n113) at (5.9, 1.5  ) {};
    \node[circle, draw, inner sep=0pt, minimum size=3.5pt] (n11) at (3.4  , 2.5  ) {};
    \node[circle, draw, inner sep=0pt, minimum size=3.5pt] (n2) at (3.4  , 1.9  ) {};
    \node[circle, draw, inner sep=0pt, minimum size=3.5pt] (n31) at (3.4  , 0.4  ) {};
    \node[circle, fill=black, inner sep=0pt, minimum size=3.5pt, label=below:$\scriptstyle \phantom{+}\!\!\!p$] (n4) at (5.50, 1.50) {};
    \coordinate (n5) at (0, 3  ) {};
    \coordinate (n6) at (0, 0  ) {};
    \node[circle, fill=black, inner sep=0pt, minimum size=3.5pt, label=below:$\scriptstyle p+1$] (n7) at (7, 1.50) {};
    \node (n8) at (8.50, 1.50) {};
    \node (n9) at (10, 1.50) {};
    \node[circle, fill=black, inner sep=0pt, minimum size=3.5pt, label={below:$\scriptstyle p+k-1=n$}] (n10) at (11.50, 1.50) {};
    \draw[decoration={complete sines,segment length=3mm, amplitude=1mm},decorate] (n1) -- (n5);
    \draw (n5) -- (n6);
    \draw[rounded corners] (n1) -- (n113) -- (n3);
    \draw[decoration={complete sines,segment length=3mm, amplitude=1mm},decorate] (n6) -- (n3);
    \path (n5) to  node[pos=0.5,xshift=-1ex] {{$J\setminus J^{(p)}_p$}} (n3);
    \foreach \x/\y in {4/7, 7/8, 9/10}
    \draw [-stealth, shorten <= 2.50pt, shorten >= 2.50pt] (n\x) to  (n\y);
    \foreach \x/\y in {11/4, 2/4, 31/4}
    \draw [-, shorten <= 2.50pt, shorten >= 2.50pt] (n\x) to  (n\y);
    \draw [line width=1.2pt, line cap=round, dash pattern=on 0pt off 3\pgflinewidth, -, shorten <= 2.50pt, shorten >= 2.50pt] (n2) to  (n31);
    \draw [line width=1.2pt, line cap=round, dash pattern=on 0pt off 5\pgflinewidth, -, shorten <= -2.50pt, shorten >= -2.50pt] (n8) to  (n9);
	\draw [decorate,decoration={brace,amplitude=5pt,mirror,raise=3ex}]
    (5.20, 1.50) -- (11.80, 1.50) node[midway,yshift=-6ex]{$J_p=I_p^{op}$};
\end{tikzpicture}
\end{center}
and the incidence matrices of posets $I$ and $J$ are as follows:
\begin{center}
$C_{I}=$
\tikzsetnextfilename{matrix_1}\begin{tikzpicture} [baseline=(n6.base), every node/.style={outer sep=0pt,inner sep=0pt},every left delimiter/.style={xshift=1pt}, every right delimiter/.style={xshift=-1pt}]
    \matrix (m1) [matrix of nodes, ampersand replacement=\&, nodes={minimum height=1.6em,minimum width=1.6em,text depth=-.25ex,text height=0.6ex,inner xsep=0.0pt,inner ysep=0.0pt, execute at begin node=$\scriptscriptstyle, execute at end node=$}, column sep={0pt,between borders},row sep={0pt,between borders}, left delimiter=\lbrack, right delimiter=\rbrack]
    {
            |(n1)|          \&         \& |(n22)|                  \& |(n3)| c_{1,p}         \& |(n14)|   \&          \& |(n23)|  \\
            \& |(n8)|  \&                          \&                        \&           \&          \&          \\
            \&         \& |(n4)|                   \& |(n2)| c_{p\mppms 1,p} \&           \&          \& |(n24)|  \\
            |(n6)| c_{p,1}  \&         \& |(n5)| c_{p, p\mppms 1}  \& |(n7)| 1               \& |(n16)| 0 \&          \& |(n15)| 0\\
            |(n19)| c_{p,1} \&         \& |(n18)| c_{p, p\mppms 1} \& |(n17)| 1              \& |(n9)| 1  \&          \& |(n25)|  \\
            \&         \&                          \&                        \&           \& |(n10)|  \&          \\
            |(n21)| c_{p,1} \&         \& |(n20)| c_{p, p\mppms 1} \& |(n13)| 1              \& |(n12)| 1 \&          \& |(n11)| 1\\
    };
    \path (n1) to  node[pos=0.5, scale=1.5] {{$\scriptstyle C_{I\setminus I_p}$}} (n4);
    \draw[-] ([xshift=-2.4pt]n3.north west) to ([xshift=-2.4pt]n13.south west);
    \draw[-] ([xshift=2.4pt]n13.south east) to ([xshift=2.4pt]n3.north east);
    \draw[-] (n6.south west) to (n15.south east);
    \draw[-] (n15.north east) to (n6.north west);
    \path (n3.north east) to  node[pos=0.5, scale=2] {{\scriptsize $0$}} (n15.north east);
    \path (n12.south west) to  node[pos=0.7, scale=1.5] {{\scriptsize $0$}} (n15.south east);
    \foreach \x/\y in {6/5, 9/11, 12/11, 16/15, 19/18, 21/20}
    \draw[line width=1.2pt, line cap=round, dash pattern=on 0pt off 5\pgflinewidth] (n\x) to (n\y);
    \foreach \x/\y in {9/12, 17/13}
    \draw[line width=1.2pt, line cap=round, dash pattern=on 0pt off 5\pgflinewidth, shorten <= 2pt] (n\x) to (n\y);
    \foreach \x/\y in {3/2, 18/20, 19/21}
    \draw[line width=1.2pt, line cap=round, dash pattern=on 0pt off 5\pgflinewidth, shorten <= 3pt] (n\x) to (n\y);
    \foreach \i/\j in {1/1, 22/p-1, 3/\phantom{+}p\phantom{+}, 14/p+1, 23/n}
    \node[gray, scale=0.7] at ([yshift=2mm]n\i.north) {$\scriptstyle \j$};
    \foreach \i/\j in {23/1, 24/p-1, 15/p, 25/p+1, 11/n}
    \node[gray, scale=0.7,text width=1.5em,anchor=west] at ([xshift=3mm]n\i.east) {$\scriptstyle \j$};
\end{tikzpicture},\quad
$C_{J}=$
\tikzsetnextfilename{matrix_2}\begin{tikzpicture} [baseline=(n6.base),  every node/.style={outer sep=0pt,inner sep=0pt},every left delimiter/.style={xshift=1pt}, every right delimiter/.style={xshift=-1pt}]
    \matrix (m1) [matrix of nodes, ampersand replacement=\&, nodes={minimum height=1.6em,minimum width=1.6em,text depth=-.25ex,text height=0.6ex,inner xsep=0.0pt,inner ysep=0.0pt, execute at begin node=$\scriptscriptstyle, execute at end node=$},column sep={0pt,between borders}, row sep={0pt,between borders}, left delimiter=\lbrack, right delimiter=\rbrack]
    {
            |(n1)|         \&         \& |(n21)|                 \& |(n3)| c_{1,p}         \& |(n18)| c_{1,p}         \&          \& |(n17)| c_{1,p}        \\
            \& |(n8)|  \&                         \&                        \&                         \&          \&                        \\
            \&         \& |(n4)|                  \& |(n2)| c_{p\mppms 1,p} \& |(n20)| c_{p\mppms 1,p} \&          \& |(n19)| c_{p\mppms 1,p}\\
            |(n6)| c_{p,1} \&         \& |(n5)| c_{p, p\mppms 1} \& |(n7)| 1               \& |(n15)| 1               \&          \& |(n14)| 1              \\
            \&         \&                         \& |(n16)| 0              \& |(n9)| 1                \&          \& |(n12)| 1              \\
            \&         \&                         \&                        \&                         \& |(n10)|  \&                        \\
            \&         \&                         \& |(n13)| 0              \&                         \&          \& |(n11)| 1              \\
    };
    \path (n1) to  node[pos=0.5, scale=1.5] {{$\scriptstyle C_{J\setminus J_p}$}} (n4);
    \draw[-] ([xshift=-2.4pt]n3.north west) to ([xshift=-2.4pt]n13.south west);
    \draw[-] ([xshift=2.4pt]n13.south east) to ([xshift=2.4pt]n3.north east);
    \draw[-] (n6.south west) to (n14.south east);
    \draw[-] (n14.north east) to (n6.north west);
    \path (n6.south west) to  node[pos=0.5, scale=2] {{\scriptsize $0$}} (n13.south west);
    \path (n13.south east) to  node[pos=0.30000000000000004, scale=1.5] {{\scriptsize $0$}} (n12.north east);
    \foreach \x/\y in {6/5, 9/11, 12/11, 15/14, 18/17, 20/19}
    \draw[line width=1.2pt, line cap=round, dash pattern=on 0pt off 5\pgflinewidth] (n\x) to (n\y);
    \foreach \x/\y in {9/12, 16/13}
    \draw[line width=1.2pt, line cap=round, dash pattern=on 0pt off 5\pgflinewidth, shorten <= 2pt] (n\x) to (n\y);
    \foreach \x/\y in {3/2, 17/19, 18/20}
    \draw[line width=1.2pt, line cap=round, dash pattern=on 0pt off 5\pgflinewidth, shorten <= 3pt] (n\x) to (n\y);
    \foreach \i/\j in {1/1, 21/p-1, 3/\phantom{+}p\phantom{+}, 18/p+1, 17/n}
    \node[gray, scale=0.7] at ([yshift=2mm]n\i.north) {$\scriptstyle \j$};
    \foreach \i/\j in {17/1, 19/p-1, 14/p, 12/p+1, 11/n}
    \node[gray, scale=0.7,text width=1.5em,anchor=west] at ([xshift=3mm]n\i.east) {$\scriptstyle \j$};
\end{tikzpicture}.
\end{center}
It is straightforward to check that $C_{J} = B^{tr} C_{I} B$, where 
\begin{equation}\label{lemma:preflection:bmat}
B\eqdef\!\!
\tikzsetnextfilename{matrix_3}\begin{tikzpicture} [baseline=(n6.base),  every node/.style={outer sep=0pt,inner sep=0pt},every left delimiter/.style={xshift=1pt}, every right delimiter/.style={xshift=-1pt}]
        \matrix (m1) [matrix of nodes, ampersand replacement=\&, nodes={minimum height=1.2em,minimum width=1.2em,text depth=-.25ex,text height=0.6ex,inner xsep=0.0pt,inner ysep=0.0pt, execute at begin node=$\scriptscriptstyle, execute at end node=$},column sep={0pt,between borders}, row sep={0pt,between borders}, left delimiter=\lbrack, right delimiter=\rbrack]
        {
                |(n1)| 1 \&         \& |(n17)|  \& |(n3)| 0  \& |(n13)|   \&          \& |(n18)|  \\
                \& |(n8)|  \&          \&           \&           \&          \&          \\
                |(n21)|       \&         \& |(n4)| 1 \& |(n2)| 0  \&           \&          \& |(n20)|  \\
                |(n6)| 0 \&         \& |(n5)| 0 \& |(n7)| 1  \& |(n15)| \phantom{-}1 \&          \& |(n14)| \phantom{-}1\\
                \&         \&          \& |(n16)| 0 \&           \&          \& |(n9)| -1 \\
                \&         \&          \&           \&           \& |(n10)|  \&          \\
                \&         \&          \& |(n12)| 0 \& |(n11)| -1 \&          \& |(n19)|  \\
        };
        \draw[-] (n3.north west) to (n12.south west);
        \draw[-] (n12.south east) to (n3.north east);
        \draw[-] (n6.south west) to (n14.south east);
        \draw[-] (n14.north east) to (n6.north west);
        \draw[line width=1.2pt, line cap=round, dash pattern=on 0pt off 5\pgflinewidth, shorten <= 2pt, shorten >= -3pt] (n15) to (n14);
        \path (n3.north east) to  node[pos=0.5, scale=2] {{\scriptsize $0$}} (n14.north east);
        \path (n6.south west) to  node[pos=0.5, scale=2] {{\scriptsize $0$}} (n12.south west);
        \path (n21.south west) to  node[pos=0.7, scale=1.2] {{\scriptsize $0$}} (n17.north east);
        \path (n17.north east) to  node[pos=0.7, scale=1.2] {{\scriptsize $0$}} (n21.south west);
        \foreach \x/\y in {1/4, 9/11}
        \draw[line width=1.2pt, line cap=round, dash pattern=on 0pt off 5\pgflinewidth, shorten <= 2pt] (n\x) to (n\y);
        \foreach \x/\y in {3/2, 6/5, 16/12}
        \draw[line width=1.2pt, line cap=round, dash pattern=on 0pt off 5\pgflinewidth, shorten <= -1pt, shorten >= -2pt] (n\x) to (n\y);
        \foreach \x/\y in {7/19, 19/7}
        \path (n\x.south east) to  node[pos=0.7, scale=1.2] {{\scriptsize $0$}} (n\y.south east);
        \foreach \i/\j in {1/1, 17/p-1, 3/\phantom{+}p\phantom{+}, 13/p+1, 18/n}
        \node[gray, scale=0.7] at ([yshift=2mm]n\i.north) {$\scriptstyle \j$};
        \foreach \i/\j in {18/1, 20/p-1, 14/p, 9/p+1, 19/n}
        \node[gray, scale=0.7,text width=1.5em,anchor=west] at ([xshift=3mm]n\i.east) {$\scriptstyle \j$};
\end{tikzpicture}\in\Gl_n(\ZZ)
\end{equation}
is an involutory matrix. It follows that posets $I$ and $J$ are strong Gram $\ZZ$-congruent and 
\[
\delta_{J_p}J=I\approx_\ZZ J=\delta_{I_p}I.
\]  
Now, assume $I_p\subseteq I$ be an outward $p$-anchored path of a poset $I$. By using arguments analogous to the previous case, one easily shows that  the matrix $B$ defines the strong Gram $\ZZ$-congruence $I\approx_\ZZ \delta_{I_p}I$.
\end{proof}
\begin{remark}\label{rmk:arbitrary_path_refl}
The matrix $B\in\Gl_n(\ZZ)$, given in \eqref{lemma:preflection:bmat}, defines the strong Gram $\ZZ$-congruence $I\approx_\ZZ \delta_{I_p}I$, with the assumption that  $I_p\subseteq I$ is such an inward or outward $p$-anchored path, that $I_p=p\,\rule[1.5pt]{22pt}{0.4pt}\, p\!+\!1\,\rule[1.5pt]{22pt}{0.4pt} \ldots\rule[1.5pt]{22pt}{0.4pt}\, n$. Assume now that $I_p$ is numbered arbitrarily, i.e., $I_p=p\,\rule[1.5pt]{22pt}{0.4pt}\, s_1\,\rule[1.5pt]{22pt}{0.4pt} \ldots\rule[1.5pt]{22pt}{0.4pt}\, s_k$, where $k=|I_p|-1$, and consider the permutation $\pi\colon \{s_1,\ldots, s_k\}\to\{s_1,\ldots, s_k\}$ with $\pi(s_i)\eqdef s_{k-i+1}$. One checks that for $B_{I_p}\in\Gl_n(\ZZ)$ composed of columns $b_1,\ldots, b_n$ with 
\[
b_j=\begin{cases}
e_p-e_{\pi(s_i)}, & \textnormal{if }j=s_i\in\{s_1, \ldots, s_k\},\\
e_j, &\textnormal{otherwise},\\
\end{cases}
\]
where $e_i$ is the $i$-th standard basis vector in $\ZZ^n$, we have $C_{\delta I_p I} = B_{I_p}^{tr} C_I B_{I_p}$, that is, the matrix $B_{I_p}$ defines the strong Gram $\ZZ$-congruence $I\approx_\ZZ \delta_{I_p}I$.
\end{remark}

Interchanging inward $p$-anchored paths with the outward ones (an operation defined at the Hasse digraph level) yields a strong Gram $\ZZ$-congruence of posets (defined at the level of incidence matrices). It is easy to generalize this fact to all, not necessarily inward/outward, $p$-anchored paths. 

\begin{corollary}\label{corr:ahchoredpath}
Let  $I_p\subseteq I$ be an arbitrary $p$-anchored path of a poset $I$. 
\begin{enumerate}[label=\normalfont{(\alph*)}]
   \item\label{corr:ahchoredpath:bilinear:out} $I\approx_\ZZ J$, where $J$ is a poset with such an outward $p$-anchored path that $I\setminus I^{(p)}_p = J\setminus J^{(p)}_p$ and $\ov\CH(J_p)=\ov\CH(I_p)$.
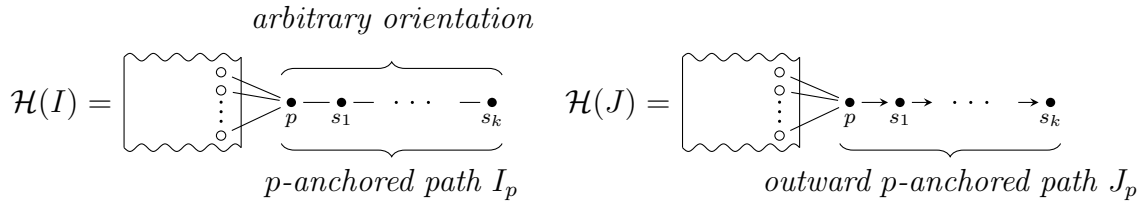
\begin{center}
$\CH(I)=$
\tikzsetnextfilename{apth_1}\begin{tikzpicture}[baseline=(n22.base),label distance=-2pt, xscale=0.44, yscale=0.4]
    \coordinate (n1) at (4, 3  ) {};
    \coordinate (n3) at (4, 0  ) {};
    \node[circle, draw, inner sep=0pt, minimum size=3.5pt] (n11) at (3.4  , 2.5  ) {};
    \node[circle, draw, inner sep=0pt, minimum size=3.5pt] (n2) at (3.4  , 1.9  ) {};
    \node[circle, draw, inner sep=0pt, minimum size=3.5pt] (n31) at (3.4  , 0.4  ) {};

    \node[circle, fill=black, inner sep=0pt, minimum size=3.5pt, label=below:$\scriptstyle p$] (n4) at (5.50, 1.50) {};
    \coordinate (n22) at (0, 1.2) {};
    \coordinate (n5) at (0.5, 3  ) {};
    \coordinate (n6) at (0.5, 0  ) {};
    \node[circle, fill=black, inner sep=0pt, minimum size=3.5pt, label=below:$\scriptstyle s_1$] (n7) at (7, 1.50) {};
    \node (n8) at (8.50, 1.50) {};
    \node (n9) at (10, 1.50) {};
    \node[circle, fill=black, inner sep=0pt, minimum size=3.5pt, label=below:$\scriptstyle s_k$] (n10) at (11.50, 1.50) {};
    \draw[decoration={complete sines,segment length=3mm, amplitude=1mm},decorate] (n1) -- (n5);
    \draw (n5) -- (n6);
    \draw (n1) -- (n3);
    \draw[decoration={complete sines,segment length=3mm, amplitude=1mm,mirror},decorate] (n3) -- (n6);
    \foreach \x/\y in {4/7, 7/8, 9/10}
    \draw [-, shorten <= 2.50pt, shorten >= 2.50pt] (n\x) to  (n\y);
    \foreach \x/\y in {11/4, 2/4, 31/4}
    \draw [-, shorten <= 2.50pt, shorten >= 2.50pt] (n\x) to  (n\y);
    \draw [line width=1.2pt, line cap=round, dash pattern=on 0pt off 3\pgflinewidth, -, shorten <= 2.50pt, shorten >= 2.50pt] (n2) to  (n31);
    \draw [line width=1.2pt, line cap=round, dash pattern=on 0pt off 5\pgflinewidth, -, shorten <= -2.50pt, shorten >= -2.50pt] (n8) to  (n9);
\draw [decorate,decoration={brace,amplitude=5pt,mirror,raise=3ex}]
    (5.20, 1.50) -- (11.80, 1.50) node[midway,yshift=-6ex]{$p$-anchored path $I_p$};
\draw [decorate,decoration={brace,amplitude=5pt,mirror,raise=2ex}]
    (11.80, 1.50) -- (5.20, 1.50) node[midway,yshift=6ex]{arbitrary orientation};
\end{tikzpicture}
\hfill $\CH(J)=$
\tikzsetnextfilename{apth_2}\begin{tikzpicture}[baseline=(n22.base),label distance=-2pt, xscale=0.44, yscale=0.4]
    \coordinate (n1) at (4, 3  ) {};
    \coordinate (n3) at (4, 0  ) {};
    \node[circle, draw, inner sep=0pt, minimum size=3.5pt] (n11) at (3.4  , 2.5  ) {};
    \node[circle, draw, inner sep=0pt, minimum size=3.5pt] (n2) at (3.4  , 1.9  ) {};
    \node[circle, draw, inner sep=0pt, minimum size=3.5pt] (n31) at (3.4  , 0.4  ) {};

    \node[circle, fill=black, inner sep=0pt, minimum size=3.5pt, label=below:$\scriptstyle p$] (n4) at (5.50, 1.50) {};
    \coordinate (n22) at (0, 1.2) {};
    \coordinate (n5) at (0.5, 3  ) {};
    \coordinate (n6) at (0.5, 0  ) {};
    \node[circle, fill=black, inner sep=0pt, minimum size=3.5pt, label=below:$\scriptstyle s_1$] (n7) at (7, 1.50) {};
    \node (n8) at (8.50, 1.50) {};
    \node (n9) at (10, 1.50) {};
    \node[circle, fill=black, inner sep=0pt, minimum size=3.5pt, label=below:$\scriptstyle s_k$] (n10) at (11.50, 1.50) {};
    \draw[decoration={complete sines,segment length=3mm, amplitude=1mm},decorate] (n1) -- (n5);
    \draw (n5) -- (n6);
    \draw (n1) -- (n3);
    \draw[decoration={complete sines,segment length=3mm, amplitude=1mm,mirror},decorate] (n3) -- (n6);
    \foreach \x/\y in {4/7, 7/8, 9/10}
    \draw [-stealth, shorten <= 2.50pt, shorten >= 2.50pt] (n\x) to  (n\y);
    \foreach \x/\y in {11/4, 2/4, 31/4}
    \draw [-, shorten <= 2.50pt, shorten >= 2.50pt] (n\x) to  (n\y);
    \draw [line width=1.2pt, line cap=round, dash pattern=on 0pt off 3\pgflinewidth, -, shorten <= 2.50pt, shorten >= 2.50pt] (n2) to  (n31);
    \draw [line width=1.2pt, line cap=round, dash pattern=on 0pt off 5\pgflinewidth, -, shorten <= -2.50pt, shorten >= -2.50pt] (n8) to  (n9);
	\draw [decorate,decoration={brace,amplitude=5pt,mirror,raise=3ex}]
    (5.20, 1.50) -- (11.80, 1.50) node[midway,yshift=-6ex]{outward $p$-anchored path $J_p$};
\end{tikzpicture}
\end{center}

    \item\label{corr:ahchoredpath:anyorient} For every orientation of arcs in $\CH(I_p)\subseteq\CH(I)$, the resulting poset $\wt I$  is non-negative of rank $m$ if and only if  $I$ is non-negative of rank~$m$.
\end{enumerate}
\end{corollary}
\begin{proof}
It is easy to see that for every orientation of edges of the path graph $\ov\CH(I_{p})$, there exists a series of $I'_{p'}$-reflections, where $p'\in I'_{p'}\subseteq I_p$, that carries the $p$-anchored path $I_p$ into outward $p$-anchored path $J_p$, therefore~\ref{corr:ahchoredpath:bilinear:out} follows by Lemma~\ref{lemma:preflection}.

Since~\ref{corr:ahchoredpath:anyorient} follows from~\ref{corr:ahchoredpath:bilinear:out}, the proof is finished.
\end{proof}

Summing up,  changing the orientation of arcs in a $p$-anchored path does not change the non-negativity nor the rank.

\begin{example}
Consider the following triple of posets: $I$, $J$ and $J'$.
\begin{center}
        \hfill
        $\CH(I)=$
        \tikzsetnextfilename{ex16_1}\begin{tikzpicture}[baseline=(n7.base),label distance=-2pt,xscale=0.65, yscale=0.74]
                \node[circle, fill=black, inner sep=0pt, minimum size=3.5pt, label=below:$\scriptscriptstyle 1$] (n1) at (4, 0  ) {};
                \node[circle, fill=black, inner sep=0pt, minimum size=3.5pt, label=below:$\scriptscriptstyle 2$] (n2) at (1, 0  ) {};
                \node[circle, fill=black, inner sep=0pt, minimum size=3.5pt, label=below:$\scriptscriptstyle 3$] (n3) at (3, 0  ) {};
                \node[circle, fill=black, inner sep=0pt, minimum size=3.5pt, label=below:$\scriptscriptstyle 4$] (n4) at (0, 0  ) {};
                \node[circle, fill=black, inner sep=0pt, minimum size=3.5pt, label=below:$\scriptscriptstyle 5$] (n5) at (2, 0  ) {};
                \node (n7) at (3, 0  ) {$\mathclap{\phantom{7}}$};
                \foreach \x/\y in {1/3, 4/2, 5/2, 5/3}
                \draw [-stealth, shorten <= 2.50pt, shorten >= 2.50pt] (n\x) to  (n\y);
        \end{tikzpicture}
        \hfill
        $\CH(J)=$
        \tikzsetnextfilename{ex16_2}\begin{tikzpicture}[baseline=(n7.base),label distance=-2pt,xscale=0.65, yscale=0.74]
                \node[circle, fill=black, inner sep=0pt, minimum size=3.5pt, label=below:$\scriptscriptstyle 1$] (n1) at (4, 0  ) {};
                \node[circle, fill=black, inner sep=0pt, minimum size=3.5pt, label=below:$\scriptscriptstyle 2$] (n2) at (1, 0  ) {};
                \node[circle, fill=black, inner sep=0pt, minimum size=3.5pt, label=below:$\scriptscriptstyle 3$] (n3) at (3, 0  ) {};
                \node[circle, fill=black, inner sep=0pt, minimum size=3.5pt, label=below:$\scriptscriptstyle 4$] (n4) at (0, 0  ) {};
                \node[circle, fill=black, inner sep=0pt, minimum size=3.5pt, label=below:$\scriptscriptstyle 5$] (n5) at (2, 0  ) {};
                \node (n7) at (3, 0  ) {$\mathclap{\phantom{7}}$};
                \foreach \x/\y in {2/5, 3/1, 4/2, 5/3}
                \draw [-stealth, shorten <= 2.50pt, shorten >= 2.50pt] (n\x) to  (n\y);
        \end{tikzpicture}
        \hfill
        $\CH(J')=$
        \tikzsetnextfilename{ex16_3}\begin{tikzpicture}[baseline=(n7.base),label distance=-2pt,xscale=0.65, yscale=0.74]
                \node[circle, fill=black, inner sep=0pt, minimum size=3.5pt, label=below:$\scriptscriptstyle 1$] (n1) at (4, 0  ) {};
                \node[circle, fill=black, inner sep=0pt, minimum size=3.5pt, label=below:$\scriptscriptstyle 2$] (n2) at (1, 0  ) {};
                \node[circle, fill=black, inner sep=0pt, minimum size=3.5pt, label=below:$\scriptscriptstyle 3$] (n3) at (3, 0  ) {};
                \node[circle, fill=black, inner sep=0pt, minimum size=3.5pt, label=below:$\scriptscriptstyle 4$] (n4) at (0, 0  ) {};
                \node[circle, fill=black, inner sep=0pt, minimum size=3.5pt, label=below:$\scriptscriptstyle 5$] (n5) at (2, 0  ) {};
                \node (n7) at (3, 0  ) {$\mathclap{\phantom{7}}$};
                \foreach \x/\y in {1/3, 2/4, 3/5, 5/2}
                \draw [-stealth, shorten <= 2.50pt, shorten >= 2.50pt] (n\x) to  (n\y);
        \end{tikzpicture}
        \hfill
        \mbox{}
\end{center}
We have $I\approx_\ZZ J$ and $I\approx_\ZZ J'$, since 
$J\! =\! \delta_{\{2, 5,3,1\}} \delta_{\{5,3,1\}} \delta_{\{3,1\}} I$ and $J'\! =\! \delta_{\{3, 5, 2, 4\}} \delta_{\{5, 2, 4\}} \delta_{\{2, 4\}} I$. Moreover, using the description given in \Cref{lemma:preflection} and Remark~\ref{rmk:arbitrary_path_refl} we get the equality $C_{J}=B_1^{tr}C_IB_1$, where
\begin{equation*}
        B_1=
        \begin{bsmallmatrix*}[r]
                \shortminus 1 & \phantom{\shortminus}0 & \phantom{\shortminus}0 & \phantom{\shortminus}0 & \phantom{\shortminus}0\\
0 & 1 & 0 & 0 & 0\\
1 & 0 & 1 & 0 & 0\\
0 & 0 & 0 & 1 & 0\\
0 & 0 & 0 & 0 & 1
        \end{bsmallmatrix*}
        \begin{bsmallmatrix*}[r]
                0 & \phantom{\shortminus}0 & \shortminus 1 &\phantom{\shortminus} 0 &\phantom{\shortminus} 0\\
0 & 1 & 0 & 0 & 0\\
\shortminus 1 & 0 & 0 & 0 & 0\\
0 & 0 & 0 & 1 & 0\\
1 & 0 & 1 & 0 & 1
        \end{bsmallmatrix*}
        \begin{bsmallmatrix*}[r] 
                0 &\phantom{\shortminus} 0 & 0 &\phantom{\shortminus} 0 & \shortminus 1\\
1 & 1 & 1 & 0 & 1\\
0 & 0 & \shortminus 1 & 0 & 0\\
0 & 0 & 0 & 1 & 0\\
\shortminus 1 & 0 & 0 & 0 & 0
        \end{bsmallmatrix*}
        =
        \begin{bsmallmatrix*}[r]
                0 &\phantom{\shortminus} 0 & \shortminus 1 &\phantom{\shortminus} 0 & 0\\
1 & 1 & 1 & 0 & 1\\
0 & 0 & 1 & 0 & 1\\
0 & 0 & 0 & 1 & 0\\
\shortminus 1 & 0 & \shortminus 1 & 0 & \shortminus 1
        \end{bsmallmatrix*}.
\end{equation*}
Analogously, one can calculate such a matrix $B_2\in\Gl_5(\ZZ)$, that $C_{J'}=B_2^{tr}C_IB_2$. 
\end{example}
\begin{remark}\label{remark:hassepath:anyorient}
In view of \Cref{corr:ahchoredpath}\ref{corr:ahchoredpath:anyorient}, we usually omit the orientation of the edges in $p$-anchored paths when presenting Hasse digraphs of finite posets. 
\end{remark}

\section{Main results}

The main result of this work is a complete description of connected non-negative posets $I$ of Dynkin type $\Dyn_I=\AA_m$ in terms of their  Hasse digraphs $\CH(I)$. First, we show that \Cref{thm:a:main} holds for ``trees'', i.e.,  posets $I$ with graph $\ov \CH(I)$ being a tree.

\begin{lemma}\label{lemma:trees}
If $I=(V,\leq_I)$ is such a connected poset of size $n$ that the graph $\ov\CH(I)$ is a tree, then exactly one of the following conditions holds.

\begin{enumerate}[label=\normalfont{(\alph*)}]
    \item\label{lemma:trees:posit} The poset $I$ is non-negative of rank $n$  and $\ov \CH(I)$ is isomorphic to a simply-laced Dynkin diagram $\Dyn_I\in\{\AA_n, \DD_n, \EE_6, \EE_7, \EE_8 \}$ of Table~\ref{tbl:Dynkin_diagrams}.
    \item\label{lemma:trees:nneg} The poset $I$ is non-negative of rank $n-1$  and $\ov \CH(I)$ is isomorphic to a simply-laced Euclidean diagram $\widetilde{\mathrm{D}}\mathrm{yn}_I \in \{\widetilde{\DD}_n,\ab \widetilde{\EE}_6,\ab \widetilde{\EE}_7,\ab \widetilde{\EE}_8\}$ of Table~\ref{tbl:Euklid_diag}.
    \item\label{lemma:trees:indef} The poset $I$ is indefinite, i.e.,  the symmetric Gram matrix $G_I$ is indefinite.
\end{enumerate}
\end{lemma}

\begin{proof} 
Assume that  $I$ is such a connected poset that the graph $\ov\CH(I)$ is a tree, where $\CH(I)=(V, A)$ is the Hasse digraph of $I$, and let $C_I\in\MM_n(\ZZ)$ be its incidence matrix \eqref{df:incmat}. By \cite[Proposition 2.12]{simsonIncidenceCoalgebrasIntervally2009}, the matrix $C_{I}^{-1}=[c_{ab}]\in\MM_n(\ZZ)$ has coefficients
\begin{equation*}
c_{ab}=\begin{cases}
        \phantom{-}1,\textnormal{ iff }a=b,\\
        -1,\textnormal{ iff } a\to b \in A,\\
        \phantom{-}0, \textnormal{ otherwise},
\end{cases} 
\end{equation*}
i.e., the matrix $C_I^{-1}$ uniquely encodes $\CH(I)$. It is straightforward to check that
\begin{equation}\label{eq:digraph_euler_q}
    q_{\CH(I)}(x)\eqdef \tfrac{1}{2}x(C_{I}^{-1} +\ C_{I}^{-tr}) x^{tr} = \sum_{i\in V} x_i^2 
    - \sum_{a\to b\,\in A} x_{a}x_{b}
\end{equation}
is the \textit{Euler quadratic form} of $\CH(I)$ in the sense of~\cite[Section VII.4]{ASS}. Moreover, we have
\[
C_I^{tr} G_{\CH(I)}  C_I=\tfrac{1}{2} C_I^{tr} (C_{I}^{-1} +\ C_{I}^{-tr})  C_I = \tfrac{1}{2}(C_I^{tr}+C_I) = G_I,
\]
where $G_{\CH(I)}\eqdef \tfrac{1}{2}(C_{I}^{-1} +\ C_{I}^{-tr})$ is the symmetric Gram matrix of the Euler quadratic form $q_{\CH(I)}$ \eqref{eq:digraph_euler_q}. Hence, the lemma follows by~\cite[Corollary 2.4]{gasiorekOnepeakPosetsPositive2012} and~\cite[Proposition VII.4.5]{ASS}.
\end{proof}

One of the key observations used in the proof of \Cref{thm:a:main} is that the graph $\ov\CH(I)$, where $I$ is a connected non-negative-poset of Dynkin type $\Dyn_I=\AA_m$, has no vertices of degree larger than $2$. In the following theorem, we give a characterization of such posets.

\begin{theorem}\label{thm:digraphdegmax2}
Assume that $\CD=(V,A)$, where $V=\{1,\ldots,n\}$, is such a connected acyclic digraph that $\deg_{\ov D}(v)\leq 2$ for every $v\in V$. Exactly
one of the following conditions holds.
\begin{enumerate}[label=\normalfont{(\alph*)}]
    \item\label{thm:digraphdegmax2:path} $\ov \CD\simeq \AA_n$ and $\CD$ is the Hasse digraph of the positive poset $I_\CD$
    with $\Dyn_{I_\CD}=\AA_n$.
    \item\label{thm:digraphdegmax2:cycle} $\ov \CD$ is a cycle graph. Moreover:
    \begin{enumerate}[label=\normalfont{(b\arabic*)}, leftmargin=4ex]
\item\label{thm:digraphdegmax2:cycle:onesink} $\CD$ has exactly one sink and
\begin{enumerate}[label=\normalfont{(\roman*)}, leftmargin=1ex]
        \item\label{thm:digraphdegmax2:cycle:onesink:posita} 
        $\Dyn_{I_\CD}=\AA_n$, $I_\CD$ is positive, $ \CD\simeq\!\!\!
        \smash{\tikzsetnextfilename{th20_p1}\begin{tikzpicture}[baseline={([yshift=-10.75pt]current bounding box)},label distance=-2pt,xscale=0.65, yscale=0.74]
            \node[circle, fill=black, inner sep=0pt, minimum size=3.5pt, label=left:$\scriptscriptstyle 1$] (n1) at (0, 0  ) {};
            \node[circle, fill=black, inner sep=0pt, minimum size=3.5pt, label=above:$\scriptscriptstyle 2$] (n2) at (1, 0.50) {};
            \node (n3) at (2, 0.50) {$\scriptscriptstyle $};
            \node (n4) at (3, 0.50) {$\scriptscriptstyle $};
            \node[circle, fill=black, inner sep=0pt, minimum size=3.5pt, label=above:$\scriptscriptstyle n\mppmss 2$] (n5) at (4, 0.50) {};
            \node[circle, fill=black, inner sep=0pt, minimum size=3.5pt, label=above:$\scriptscriptstyle n\mppmss 1$] (n6) at (5, 0.50) {};
            \node[circle, fill=black, inner sep=0pt, minimum size=3.5pt, label=right:$\scriptstyle n$] (n7) at (6, 0  ) {};
            \foreach \x/\y in {1/2, 1/7, 5/6, 6/7}
            \draw [-stealth, shorten <= 2.50pt, shorten >= 2.50pt] (n\x) to  (n\y);
            \draw [-stealth, shorten <= 2.50pt, shorten >= -2.50pt] (n2) to  (n3);
            \draw [line width=1.2pt, line cap=round, dash pattern=on 0pt off 5\pgflinewidth, -, shorten <= -1.00pt, shorten >= -2.50pt] (n3) to  (n4);
            \draw [-stealth, shorten <= -2.50pt, shorten >= 2.50pt] (n4) to  (n5);
        \end{tikzpicture}}$\!, and $\CH(I_\CD)\neq \CD$,

        \item\label{thm:digraphdegmax2:cycle:onesink:positd}
        $\Dyn_{I_\CD}=\DD_n$, $I_\CD$ is positive, $\CD\simeq$\!\!\!
        \tikzsetnextfilename{th20_p2}\begin{tikzpicture}[baseline={([yshift=-2.75pt]current bounding box)},label distance=-2pt,xscale=0.65, yscale=0.74]
        \node[circle, fill=black, inner sep=0pt, minimum size=3.5pt, label=left:$\scriptscriptstyle 1$] (n1) at (0, 0.25 ) {};
        \node[circle, fill=black, inner sep=0pt, minimum size=3.5pt, label=above:$\scriptscriptstyle 2$] (n2) at (1, 0.50) {};
        \node (n3) at (2, 0.50) {$\scriptscriptstyle $};
        \node (n4) at (3, 0.50) {$\scriptscriptstyle $};
        \node[circle, fill=black, inner sep=0pt, minimum size=3.5pt, label=above:$\scriptscriptstyle n\mppmss 3$] (n5) at (4, 0.50) {};
        \node[circle, fill=black, inner sep=0pt, minimum size=3.5pt, label=above:$\scriptscriptstyle n\mppmss 2$] (n6) at (5, 0.50) {};
        \node[circle, fill=black, inner sep=0pt, minimum size=3.5pt, label=right:$\scriptstyle n$] (n7) at (6, 0.25 ) {};
        \node[circle, fill=black, inner sep=0pt, minimum size=3.5pt, label=below:$\scriptscriptstyle n\mppmss 1$] (n8) at (3, 0  ) {};
        \foreach \x/\y in {1/2, 1/8, 5/6, 6/7, 8/7}
        \draw [-stealth, shorten <= 2.50pt, shorten >= 2.50pt] (n\x) to  (n\y);
        \draw [-stealth, shorten <= 2.50pt, shorten >= -2.50pt] (n2) to  (n3);
        \draw [line width=1.2pt, line cap=round, dash pattern=on 0pt off 5\pgflinewidth, -, shorten <= -1.00pt, shorten >= -2.50pt] (n3) to  (n4);
        \draw [-stealth, shorten <= -2.50pt, shorten >= 2.50pt] (n4) to  (n5);
        \end{tikzpicture}\!,  and $\CH(I_\CD)= \CD$,
        
        \item\label{thm:digraphdegmax2:cycle:onesink:posite}
        $\Dyn_{I_\CD}=\EE_\star$, $\CH(I_\CD)=\CD$ and
        \begin{enumerate}[label=\textnormal{\textbullet}, leftmargin=1ex] 
            \item $I_\CD$ is positive,
            $\CD\simeq$\!\!\!
            \tikzsetnextfilename{th20_p3}\begin{tikzpicture}[baseline={([yshift=-2.75pt]current bounding box)},
            label distance=-2pt,xscale=0.55, yscale=0.54]
            \node[circle, fill=black, inner sep=0pt, minimum size=3.5pt, label=below:$\scriptscriptstyle 1$] (n1) at (0, 0.50) {};
            \node[circle, fill=black, inner sep=0pt, minimum size=3.5pt, label=above:$\scriptscriptstyle 2$] (n2) at (1, 1  ) {};
            \node[circle, fill=black, inner sep=0pt, minimum size=3.5pt, label=above:$\scriptscriptstyle 3$] (n3) at (2, 1  ) {};
            \node[circle, fill=black, inner sep=0pt, minimum size=3.5pt, label=below:$\scriptscriptstyle 4$] (n4) at (1, 0  ) {};
            \node[circle, fill=black, inner sep=0pt, minimum size=3.5pt, label=below:$\scriptscriptstyle 5$] (n5) at (2, 0  ) {};
            \node[circle, fill=black, inner sep=0pt, minimum size=3.5pt, label=below:$\scriptscriptstyle 6$] (n6) at (3, 0.50) {};
            \foreach \x/\y in {1/2, 1/4, 2/3, 3/6, 4/5, 5/6}
            \draw [-stealth, shorten <= 2.50pt, shorten >= 2.50pt] (n\x) to  (n\y);
            \end{tikzpicture} or\ \  
            $\CD\simeq$\!\!\!
            \tikzsetnextfilename{th20_p4}\begin{tikzpicture}[baseline={([yshift=-2.75pt]current bounding box)},
            label distance=-2pt,xscale=0.55, yscale=0.54]
            \node[circle, fill=black, inner sep=0pt, minimum size=3.5pt, label=below:$\scriptscriptstyle 1$] (n1) at (0, 0.50) {};
            \node[circle, fill=black, inner sep=0pt, minimum size=3.5pt, label=above:$\scriptscriptstyle 2$] (n2) at (1, 1  ) {};
            \node[circle, fill=black, inner sep=0pt, minimum size=3.5pt, label=above:$\scriptscriptstyle 3$] (n3) at (2, 1  ) {};
            \node[circle, fill=black, inner sep=0pt, minimum size=3.5pt, label=above:$\scriptscriptstyle 4$] (n4) at (3, 1  ) {};
            \node[circle, fill=black, inner sep=0pt, minimum size=3.5pt, label=below:$\scriptscriptstyle 5$] (n5) at (1.50, 0  ) {};
            \node[circle, fill=black, inner sep=0pt, minimum size=3.5pt, label=below:$\scriptscriptstyle 6$] (n6) at (2.50, 0  ) {};
            \node[circle, fill=black, inner sep=0pt, minimum size=3.5pt, label=below:$\scriptscriptstyle 7$] (n7) at (4, 0.50) {};
            \foreach \x/\y in {1/2, 1/5, 2/3, 3/4, 4/7, 5/6, 6/7}
            \draw [-stealth, shorten <= 2.50pt, shorten >= 2.50pt] (n\x) to  (n\y);
            \end{tikzpicture} or \ 
            $\CD\simeq$\!\!\!
            \tikzsetnextfilename{th20_p5}\begin{tikzpicture}[baseline={([yshift=-2.75pt]current bounding box)},
            label distance=-2pt,xscale=0.55, yscale=0.54]
            \node[circle, fill=black, inner sep=0pt, minimum size=3.5pt, label=below:$\scriptscriptstyle 1$] (n1) at (0, 0.50) {};
            \node[circle, fill=black, inner sep=0pt, minimum size=3.5pt, label=above:$\scriptscriptstyle 2$] (n2) at (1, 1  ) {};
            \node[circle, fill=black, inner sep=0pt, minimum size=3.5pt, label=above:$\scriptscriptstyle 3$] (n3) at (2, 1  ) {};
            \node[circle, fill=black, inner sep=0pt, minimum size=3.5pt, label=above:$\scriptscriptstyle 4$] (n4) at (3, 1  ) {};
            \node[circle, fill=black, inner sep=0pt, minimum size=3.5pt, label=above:$\scriptscriptstyle 5$] (n5) at (4, 1  ) {};
            \node[circle, fill=black, inner sep=0pt, minimum size=3.5pt, label=below:$\scriptscriptstyle 6$] (n6) at (2, 0  ) {};
            \node[circle, fill=black, inner sep=0pt, minimum size=3.5pt, label=below:$\scriptscriptstyle 7$] (n7) at (3, 0  ) {};
            \node[circle, fill=black, inner sep=0pt, minimum size=3.5pt, label=below:$\scriptscriptstyle 8$] (n8) at (5, 0.50) {};
            \foreach \x/\y in {1/2, 1/6, 2/3, 3/4, 4/5, 5/8, 6/7, 7/8}
            \draw [-stealth, shorten <= 2.50pt, shorten >= 2.50pt] (n\x) to  (n\y);
            \end{tikzpicture} with
            $\Dyn_{I_\CD}=\EE_6$, $\EE_7$ and $\EE_8$, respectively;
            
            \item $I_\CD$ is principal,
            $\CD\simeq$\!\!\!
            \tikzsetnextfilename{th20_p6}\begin{tikzpicture}[baseline={([yshift=-2.75pt]current bounding box)},
            label distance=-2pt,xscale=0.55, yscale=0.54]
            \node[circle, fill=black, inner sep=0pt, minimum size=3.5pt, label=below:$\scriptscriptstyle 1$] (n1) at (0, 0.50) {};
            \node[circle, fill=black, inner sep=0pt, minimum size=3.5pt, label=above:$\scriptscriptstyle 2$] (n2) at (1, 1  ) {};
            \node[circle, fill=black, inner sep=0pt, minimum size=3.5pt, label=above:$\scriptscriptstyle 3$] (n3) at (2, 1  ) {};
            \node[circle, fill=black, inner sep=0pt, minimum size=3.5pt, label=above:$\scriptscriptstyle 4$] (n4) at (3, 1  ) {};
            \node[circle, fill=black, inner sep=0pt, minimum size=3.5pt, label=below:$\scriptscriptstyle 5$] (n5) at (1, 0  ) {};
            \node[circle, fill=black, inner sep=0pt, minimum size=3.5pt, label=below:$\scriptscriptstyle 6$] (n6) at (2, 0  ) {};
            \node[circle, fill=black, inner sep=0pt, minimum size=3.5pt, label=below:$\scriptscriptstyle 7$] (n7) at (3, 0  ) {};
            \node[circle, fill=black, inner sep=0pt, minimum size=3.5pt, label=below:$\scriptscriptstyle 8$] (n8) at (4, 0.50) {};
            \foreach \x/\y in {1/2, 1/5, 2/3, 3/4, 4/8, 5/6, 6/7, 7/8}
            \draw [-stealth, shorten <= 2.50pt, shorten >= 2.50pt] (n\x) to  (n\y);
            \end{tikzpicture} or \ 
            $\CD\simeq$\!\!\!
            \tikzsetnextfilename{th20_p7}\begin{tikzpicture}[baseline={([yshift=-2.75pt]current bounding box)},
            label distance=-2pt,xscale=0.55, yscale=0.54]
            \node[circle, fill=black, inner sep=0pt, minimum size=3.5pt, label=below:$\scriptscriptstyle 1$] (n1) at (0, 0.50) {};
            \node[circle, fill=black, inner sep=0pt, minimum size=3.5pt, label=above:$\scriptscriptstyle 2$] (n2) at (1, 1  ) {};
            \node[circle, fill=black, inner sep=0pt, minimum size=3.5pt, label=above:$\scriptscriptstyle 3$] (n3) at (2, 1  ) {};
            \node[circle, fill=black, inner sep=0pt, minimum size=3.5pt, label=above:$\scriptscriptstyle 4$] (n4) at (3, 1  ) {};
            \node[circle, fill=black, inner sep=0pt, minimum size=3.5pt, label=above:$\scriptscriptstyle 5$] (n5) at (4, 1  ) {};
            \node[circle, fill=black, inner sep=0pt, minimum size=3.5pt, label=above:$\scriptscriptstyle 6$] (n6) at (5, 1  ) {};
            \node[circle, fill=black, inner sep=0pt, minimum size=3.5pt, label=below:$\scriptscriptstyle 7$] (n7) at (2.50, 0  ) {};
            \node[circle, fill=black, inner sep=0pt, minimum size=3.5pt, label=below:$\scriptscriptstyle 8$] (n8) at (3.50, 0  ) {};
            \node[circle, fill=black, inner sep=0pt, minimum size=3.5pt, label=below:$\scriptscriptstyle 9$] (n9) at (6, 0.50) {};
            \foreach \x/\y in {1/2, 1/7, 2/3, 3/4, 4/5, 5/6, 6/9, 7/8, 8/9}
            \draw [-stealth, shorten <= 2.50pt, shorten >= 2.50pt] (n\x) to  (n\y);
            \end{tikzpicture}\\ with
            $\Dyn_{I_\CD}=\EE_7$ and $\EE_8$, respectively;
        \end{enumerate}

        \item\label{thm:digraphdegmax2:cycle:onesink:indef} otherwise, the poset $I_\CD$ is indefinite. 
    \end{enumerate}

    \item\label{thm:digraphdegmax2:cycle:multsink} $\CD$ has  more than one sink,  $I_\CD$ is principal with $\Dyn_{I_\CD}=\AA_{n-1}$ and $\CH(I_\CD)=\CD$. 
    \end{enumerate}
\end{enumerate}
\end{theorem}

\begin{proof} If  $\CD$ is such a connected acyclic digraph, that $\deg_{\ov D}(v)\leq 2$ for every $v\in V$, then  $\ov \CD$ is  either a cycle or a path graph.\smallskip 
        
\ref{thm:digraphdegmax2:path}  
If $\ov \CD$ is a path graph, then clearly $\ov\CD$ is a tree and \ref{thm:digraphdegmax2:path} follows
by \Cref{lemma:trees}\ref{lemma:trees:posit}.\smallskip

\ref{thm:digraphdegmax2:cycle}
Assume that $\ov \CD$ is a cycle. It is easy to see that $\CD$ is composed of $2k$ oriented paths and has exactly $k$ sources and $k$ sinks, where $0\neq k\in\NN$. First, we assume that $k=1$.

\ref{thm:digraphdegmax2:cycle:onesink} Since $\CD$ has exactly one sink, 
it is composed of two oriented paths, and we have
\begin{equation}\label{eq:digraph_Apr}
        \CD \simeq \CA_{p,r}\ \ \eqdef
        \tikzsetnextfilename{th20prf_p1}\begin{tikzpicture}[baseline={([yshift=-2.75pt]current bounding box)},label distance=-2pt,xscale=0.65, yscale=0.74]
                \node[circle, fill=black, inner sep=0pt, minimum size=3.5pt, label=left:$\scriptscriptstyle 1$] (n1) at (0, 0.50) {};
                \node[circle, fill=black, inner sep=0pt, minimum size=3.5pt, label=above:$\scriptscriptstyle 2$] (n2) at (1, 1  ) {};
                \node[circle, fill=black, inner sep=0pt, minimum size=3.5pt, label=below:$\scriptscriptstyle p\mpppss 1$] (n3) at (1.50, 0  ) {};
                \node (n4) at (2, 1  ) {$\scriptscriptstyle $};
                \node (n5) at (2.50, 0  ) {$\scriptscriptstyle $};
                \node (n6) at (3, 1  ) {$\scriptscriptstyle $};
                \node (n7) at (3.50, 0  ) {$\scriptscriptstyle $};
                \node[circle, fill=black, inner sep=0pt, minimum size=3.5pt, label=above:$\scriptscriptstyle p\mppmss 1$] (n8) at (4, 1  ) {};
                \node[circle, fill=black, inner sep=0pt, minimum size=3.5pt, label=above:$\scriptscriptstyle p$] (n9) at (5, 1  ) {};
                \node[circle, fill=black, inner sep=0pt, minimum size=3.5pt, label=below:$\scriptscriptstyle p\mpppss r\mppmss 1$] (n10) at (4.50, 0  ) {};
                \node[circle, fill=black, inner sep=0pt, minimum size=3.5pt, label=right:{$\scriptstyle p\mppps r=n$}] (n11) at (6, 0.50) {};
                \foreach \x/\y in {1/2, 1/3, 8/9, 9/11, 10/11}
                \draw [-stealth, shorten <= 2.50pt, shorten >= 2.50pt] (n\x) to  (n\y);
                \foreach \x/\y in {2/4, 3/5}
                \draw [-stealth, shorten <= 2.50pt, shorten >= -2.50pt] (n\x) to  (n\y);
                \foreach \x/\y in {4/6, 5/7}
                \draw [line width=1.2pt, line cap=round, dash pattern=on 0pt off 5\pgflinewidth, -, shorten <= -1.00pt, shorten >= -2.50pt] (n\x) to  (n\y);
                \foreach \x/\y in {6/8, 7/10}
                \draw [-stealth, shorten <= -2.50pt, shorten >= 2.50pt] (n\x) to  (n\y);
        \end{tikzpicture}
\end{equation}
where $1\leq r\leq p$ and $p+r=n\geq 3$. We note that $I_\CD$ is a one-peak poset and recall that, up to the isomorphism of Hasse digraphs, all such positive and principal posets are classified in~\cite[Theorem 5.2]{gasiorekOnepeakPosetsPositive2012} and~\cite[Theorem~3.5]{gasiorekAlgorithmicStudyNonnegative2015} (see Theorem~\ref{thm:onepeak_posit} for the positive case). This classification shows that the following cases are possible.

\ref{thm:digraphdegmax2:cycle:onesink:posita} 
For $r=1$ (i.e., $\CD \simeq \CA_{n-1,1}$) we have $\CH(I_\CD)={}_{0}\AA^*_{n-1}\neq \CD$ and $\Dyn_{I_\CD}=\AA_n$; \ref{thm:digraphdegmax2:cycle:onesink:positd} for $r=2$ (i.e., $\CD \simeq \CA_{n-2,2}$) we have 
$\CH(I_\CD)=\wh\DD^*_{n-2}\diamond \AA_{1} = \CD$ and $\Dyn_{I_\CD}=\DD_n$, hence thesis follows by~\cite[Theorem 5.2]{gasiorekOnepeakPosetsPositive2012}. 

\ref{thm:digraphdegmax2:cycle:onesink:posite}
If $r>2$, then  exactly $5$ digraphs of the shape~\eqref{eq:digraph_Apr} define non-negative posets: 
\begin{itemize}
    \item $\CA_{3,3}\simeq \PP_{5}$, $\CA_{4,3}\simeq \PP_{24}$ and $\CA_{5,3}\simeq \PP_{93}$ are Hasse digraphs of a positive poset of the Dynkin type $\EE_6$, $\EE_7$ and $\EE_8$, respectively, see~\cite[Tables 6.1–6.3]{gasiorekOnepeakPosetsPositive2012};
    \item $\CA_{4,4}$ and $\CA_{6,3}$ are Hasse digraphs of a principal poset 
    of the Dynkin type $\EE_7$ and $\EE_8$, respectively, see~\cite{gasiorekCoxeterTypeClassification2019}.
\end{itemize}
In each  of the remaining cases,
the poset $I_{\CA_{p,r}}$ contains one of the subposets: $I_{\CA_{7,3}}$ or $I_{\CA_{5,4}}$.
\begin{center}
        $\CA_{7,3}=$
        \tikzsetnextfilename{th20prf_p22}\begin{tikzpicture}[baseline={([yshift=-2.75pt]current bounding box)},label distance=-2pt,xscale=0.65, yscale=0.74]
                \node[circle, fill=black, inner sep=0pt, minimum size=3.5pt, label=below:$\scriptscriptstyle 1$] (n1) at (0, 0.50) {};
                \node[circle, fill=black, inner sep=0pt, minimum size=3.5pt, label=above:$\scriptscriptstyle 2$] (n2) at (1, 1  ) {};
                \node[circle, fill=black, inner sep=0pt, minimum size=3.5pt, label=above:$\scriptscriptstyle 3$] (n3) at (2, 1  ) {};
                \node[circle, fill=black, inner sep=0pt, minimum size=3.5pt, label=above:$\scriptscriptstyle 4$] (n4) at (3, 1  ) {};
                \node[circle, fill=black, inner sep=0pt, minimum size=3.5pt, label=above:$\scriptscriptstyle 5$] (n5) at (4, 1  ) {};
                \node[circle, fill=black, inner sep=0pt, minimum size=3.5pt, label=above:$\scriptscriptstyle 6$] (n6) at (5, 1  ) {};
                \node[circle, fill=black, inner sep=0pt, minimum size=3.5pt, label=above:$\scriptscriptstyle 7$] (n7) at (6, 1  ) {};
                \node[circle, fill=black, inner sep=0pt, minimum size=3.5pt, label=below:$\scriptscriptstyle 8$] (n8) at (3, 0  ) {};
                \node[circle, fill=black, inner sep=0pt, minimum size=3.5pt, label=below:$\scriptscriptstyle 9$] (n9) at (4, 0  ) {};
                \node[circle, fill=black, inner sep=0pt, minimum size=3.5pt, label=below:$\scriptscriptstyle 10$] (n10) at (7, 0.50) {};
                \foreach \x/\y in {1/2, 1/8, 2/3, 3/4, 4/5, 5/6, 6/7, 7/10, 8/9, 9/10}
                \draw [-stealth, shorten <= 2.50pt, shorten >= 2.50pt] (n\x) to  (n\y);
        \end{tikzpicture}\qquad
        $\CA_{5,4}=$
        \tikzsetnextfilename{th20prf_p23}\begin{tikzpicture}[baseline={([yshift=-2.75pt]current bounding box)},label distance=-2pt,xscale=0.65, yscale=0.74]
                \node[circle, fill=black, inner sep=0pt, minimum size=3.5pt, label=below:$\scriptscriptstyle 1$] (n1) at (0, 0.50) {};
                \node[circle, fill=black, inner sep=0pt, minimum size=3.5pt, label=above:$\scriptscriptstyle 2$] (n2) at (1, 1  ) {};
                \node[circle, fill=black, inner sep=0pt, minimum size=3.5pt, label=above:$\scriptscriptstyle 3$] (n3) at (2, 1  ) {};
                \node[circle, fill=black, inner sep=0pt, minimum size=3.5pt, label=above:$\scriptscriptstyle 4$] (n4) at (3, 1  ) {};
                \node[circle, fill=black, inner sep=0pt, minimum size=3.5pt, label=above:$\scriptscriptstyle 5$] (n5) at (4, 1  ) {};
                \node[circle, fill=black, inner sep=0pt, minimum size=3.5pt, label=below:$\scriptscriptstyle 6$] (n6) at (1.50, 0  ) {};
                \node[circle, fill=black, inner sep=0pt, minimum size=3.5pt, label=below:$\scriptscriptstyle 7$] (n7) at (2.50, 0  ) {};
                \node[circle, fill=black, inner sep=0pt, minimum size=3.5pt, label=below:$\scriptscriptstyle 8$] (n8) at (3.50, 0  ) {};
                \node[circle, fill=black, inner sep=0pt, minimum size=3.5pt, label=below:$\scriptscriptstyle 9$] (n9) at (5, 0.50) {};
                \foreach \x/\y in {1/2, 1/6, 2/3, 3/4, 4/5, 5/9, 6/7, 7/8, 8/9}
                \draw [-stealth, shorten <= 2.50pt, shorten >= 2.50pt] (n\x) to  (n\y);
        \end{tikzpicture}
\end{center}
Since
\begin{itemize}
        \item $q_{\CA_{7,3}}([11,\ab -3,\ab -3,\ab -3,\ab -3,\ab -3,\ab -3,\ab -7,\ab -7,\ab 10])=-5$ and
        \item $q_{\CA_{5,4}}([11,\ab -4,\ab -4,\ab -4,\ab -4,\ab -5,\ab -5,\ab -5,\ab 9])=-9$,
\end{itemize}
these posets are indefinite and \ref{thm:digraphdegmax2:cycle:onesink:indef} follows (see Remark~\ref{rmk:indef}).\pagebreak

\ref{thm:digraphdegmax2:cycle:multsink} Assume that $\ov \CD$ is a cycle that 
is composed of $s\eqdef 2k>2$ oriented paths. Without loss of generality,
we can assume that
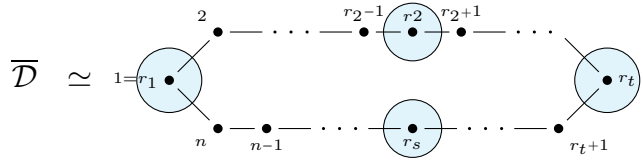
\begin{center}
    $\ov \CD\ \ \simeq$
    \tikzsetnextfilename{th20prf_p2}\begin{tikzpicture}[baseline={([yshift=-2.75pt]current bounding box)},label distance=-2pt,xscale=0.64, yscale=0.64]
            \draw[draw, fill=cyan!10](0,1) circle (19pt);
            \draw[draw, fill=cyan!10](5,2) circle (17pt);
            \draw[draw, fill=cyan!10](9,1) circle (19pt);
            \draw[draw, fill=cyan!10](5,0) circle (17pt);
            \node[circle, fill=black, inner sep=0pt, minimum size=3.5pt, label=left:{$\scriptscriptstyle 1=r_1$}] (n1) at (0, 1  ) {};
            \node[circle, fill=black, inner sep=0pt, minimum size=3.5pt, label=above left:$\scriptscriptstyle 2$] (n2) at (1, 2  ) {};
            \node (n3) at (2, 2  ) {$\scriptscriptstyle $};
            \node (n4) at (3, 2  ) {$\scriptscriptstyle $};
            \node[circle, fill=black, inner sep=0pt, minimum size=3.5pt, label=above:$\scriptscriptstyle r_2\mppmss 1$] (n5) at (4, 2  ) {};
            \node[circle, fill=black, inner sep=0pt, minimum size=3.5pt, label=above:$\scriptscriptstyle r2$] (n6) at (5, 2  ) {};
            \node[circle, fill=black, inner sep=0pt, minimum size=3.5pt, label=above:$\scriptscriptstyle r_2\mpppss 1$] (n7) at (6, 2  ) {};
            \node (n8) at (7, 2  ) {$\scriptscriptstyle $};
            \node (n9) at (8, 2  ) {$\scriptscriptstyle $};
            \node[circle, fill=black, inner sep=0pt, minimum size=3.5pt, label=right:$\scriptscriptstyle r_t$] (n10) at (9, 1  ) {};
            \node[circle, fill=black, inner sep=0pt, minimum size=3.5pt, label=below right:$\scriptscriptstyle r_t\mpppss 1$] (n11) at (8, 0  ) {};
            \node (n12) at (7, 0  ) {$\scriptscriptstyle $};
            \node (n13) at (6, 0  ) {$\scriptscriptstyle $};
            \node[circle, fill=black, inner sep=0pt, minimum size=3.5pt, label=below:$\scriptscriptstyle n\mppmss 1$] (n14) at (2, 0  ) {};
            \node[circle, fill=black, inner sep=0pt, minimum size=3.5pt, label=below left:$\scriptscriptstyle n$] (n15) at (1, 0  ) {};
            \node (n16) at (3, 0  ) {};
            \node (n17) at (4, 0  ) {};
            \node[circle, fill=black, inner sep=0pt, minimum size=3.5pt, label=below:$\scriptscriptstyle r_s$] (n18) at (5, 0  ) {};
            
            \foreach \x/\y in {1/2, 5/6, 6/7, 10/11, 14/15, 15/1}
            \draw [-, shorten <= 2.50pt, shorten >= 2.50pt] (n\x) to  (n\y);
            \foreach \x/\y in {2/3, 7/8, 11/12, 18/17}
            \draw [-, shorten <= 2.50pt, shorten >= -2.50pt] (n\x) to  (n\y);
            \foreach \x/\y in {3/4, 8/9, 12/13, 17/16}
            \draw [line width=1.2pt, line cap=round, dash pattern=on 0pt off 5\pgflinewidth, -, shorten <= -1.00pt, shorten >= -2.50pt] (n\x) to  (n\y);
            \foreach \x/\y in {4/5, 9/10, 13/18, 16/14}
            \draw [-, shorten <= -2.50pt, shorten >= 2.50pt] (n\x) to  (n\y);
    \end{tikzpicture}
\end{center}
where every  $r_1,\ldots,r_s\in \{1,\ldots, n \}$ is either a source or a sink and
\begin{itemize}
    \item $1 = r_1 < r_2 < \cdots < r_t<\cdots < r_s$,
    \item $\{1,\ldots,r_2\}$, $\{r_2,\ldots,r_3\}$, $\ldots$\,, $\{r_{s-1},\ldots,r_s\},\{r_s,\ldots,n, 1\}\simeq\,  \bullet \raisebox{-1.5pt}{\parbox{25pt}{\rightarrowfill}} \,\hdashrule[1.5pt]{12pt}{0.4pt}{1pt} \raisebox{-1.5pt}{\parbox{25pt}{\rightarrowfill}} \bullet$.
\end{itemize}
Since the quadratic form $q_{I_\CD}\colon \ZZ^n\to \ZZ$ \eqref{eq:quadratic_form} is given by the formula:
{\allowdisplaybreaks\begin{align*}
    q_{I_\CD}(x) =& \sum_{i} x_i^2 \,\,+ \sum_{{1\leq t < s}}
    \Big(\sum_{r_t\leq i<j \leq r_{t+1}} x_i x_j\Big)  +
    \sum_{r_s\leq i<j \leq n} x_i x_j + x_1(x_{r_s}+\cdots+ x_n)\\
    =& \sum_{i\not\in\{r_1,\ldots,r_s \}} \frac{1}{2}x_i^2 + \frac{1}{2}
    \sum_{1\leq t < s} \Big(\sum_{r_t\leq i\leq r_{t+1}} x_i\Big)^2 +
    \frac{1}{2}(x_{r_s} + \cdots  + x_{n} + x_1)^2,
\end{align*}}%
the poset $I_\CD$ is non-negative. Consider the vector $h_\CD=[h_1,\ldots,h_n]\in\ZZ^n$, where $h_i=0$ if $i\not\in\{r_1,\ldots,r_s \}$, $h_{r_i}=1$ if $i$ is odd and $-1$ otherwise. That is, $h_\CD\in\ZZ^n$ has $s=2k$ non-zero coordinates (equal $1$ and $-1$ alternately). It is straightforward to check that
\[
q_{I_\CD}(h_\CD) = \tfrac{1}{2}(h_{r_1}+h_{r_2})^2+\cdots+\tfrac{1}{2}(h_{r_{s-1}}+h_{r_s})^2 +\tfrac{1}{2}(h_{r_s}+h_1)^2=0,
\]
i.e.,  $I_\CD$ is not positive and the first coordinate of the vector $h_\CD\in\Ker q_{I_\CD}$ equals $h_1=1$. Since $\smash{\ov \CD^{(1)}}\simeq \AA_{n-1}$, by~\ref{thm:digraphdegmax2:path} the poset $I_{\CD^{(1)}}\subset I_\CD$ is positive and  $\Dyn_{I_{\CD^{(1)}}}\!=\AA_{n-1}$. Hence, we conclude that $I_\CD$ is principal of Dynkin type $\Dyn_{I_\CD}\!=\AA_{n-1}$, see \Cref{df:Dynkin_type}.
\end{proof}

One of the methods to prove that a particular poset $I$ is indefinite is to show that it contains a subposet $J\subseteq I$ that is indefinite (we use this argument in the proof of Theorem~\ref{thm:digraphdegmax2}\ref{thm:digraphdegmax2:cycle:onesink:posite}). The following lemma presents a list of indefinite posets used further in the paper. 

\begin{lemma}\label{lemma:posindef}
If $I$ is such a finite poset, that its Hasse digraph $\CH(I)$  is isomorphic with any of the  digraphs $\CF_1,\ldots,\CF_7$ given in \Cref{tbl:indefposets}, then $I$ is indefinite.\vspace*{-0.3ex}
{\newcommand{\tsep}{\,}
\begin{longtable}{@{}c@{\tsep}c@{\tsep}c@{\tsep}c@{\tsep}c@{\tsep}c@{\tsep}c@{\tsep}c@{\tsep}c@{\tsep}c@{\tsep}c@{\tsep}c@{}}
\multicolumn{3}{@{\tsep}c@{\tsep}}{$\CF_1\colon\!\!\!\!$
        \tikzsetnextfilename{indef_f1}\begin{tikzpicture}[baseline={([yshift=-2.75pt]current bounding box)},label distance=-2pt, xscale=0.65, yscale=1.0]
                \node[circle, draw, inner sep=0pt, minimum size=3.5pt, label=right:$\scriptscriptstyle \mppmss 2$] (n1) at (1, 0  ) {};
                \node[circle, draw, inner sep=0pt, minimum size=3.5pt, label=below:$\scriptscriptstyle 2$] (n2) at (0, 0.50) {};
                \node[circle, draw, inner sep=0pt, minimum size=3.5pt, label=right:$\scriptscriptstyle \mppmss 2$] (n3) at (1, 1  ) {};
                \node[circle, draw, inner sep=0pt, minimum size=3.5pt, label={[xshift=-0.2ex]above right:$\scriptscriptstyle 2$}] (n4) at (2, 0.50) {};
                \node[circle, draw, inner sep=0pt, minimum size=3.5pt, label=below:$\scriptscriptstyle \mppmss 1$] (n5) at (3, 0.50) {};
                \foreach \x/\y in {2/1, 2/3, 4/1, 4/3}
                \draw [-stealth, shorten <= 2.50pt, shorten >= 2.50pt] (n\x) to  (n\y);
                \draw [-, shorten <= 2.50pt, shorten >= 2.50pt] (n4) to  (n5);
\end{tikzpicture}} & \multicolumn{3}{@{\tsep}c@{\tsep}}{$\CF_2\colon\!\!\!\!$
        \tikzsetnextfilename{indef_f2}\begin{tikzpicture}[baseline={([yshift=-2.75pt]current bounding box)},label distance=-2pt, xscale=0.65, yscale=1.00]
                \node[circle, draw, inner sep=0pt, minimum size=3.5pt, label=below:$\scriptscriptstyle \mppmss 4$] (n1) at (3, 0.50) {};
                \node[circle, draw, inner sep=0pt, minimum size=3.5pt, label=below:$\scriptscriptstyle \mppmss 4$] (n2) at (2, 0.50) {};
                \node[circle, draw, inner sep=0pt, minimum size=3.5pt, label=below:$\scriptscriptstyle \mppmss 6$] (n3) at (0, 0.50) {};
                \node[circle, draw, inner sep=0pt, minimum size=3.5pt, label=right:$\scriptscriptstyle 7$] (n4) at (1, 1  ) {};
                \node[circle, draw, inner sep=0pt, minimum size=3.5pt, label=right:$\scriptscriptstyle 5$] (n5) at (1, 0  ) {};
                \foreach \x/\y in {4/2, 4/3, 5/2, 5/3}
                \draw [-stealth, shorten <= 2.50pt, shorten >= 2.50pt] (n\x) to  (n\y);
                \draw [-, shorten <= 2.50pt, shorten >= 2.50pt] (n2) to  (n1);
\end{tikzpicture}} & \multicolumn{3}{@{\tsep}c@{\tsep}}{$\CF_3\colon\!\!\!\!$
        \tikzsetnextfilename{indef_f3}\begin{tikzpicture}[baseline={([yshift=-2.75pt]current bounding box)},label distance=-2pt,
                xscale=0.65, yscale=0.5]
                \node[circle, draw, inner sep=0pt, minimum size=3.5pt, label={[xshift=0.2ex]below:$\scriptscriptstyle \mppmss 20$}] (n1) at (4, 1  ) {};
                \node[circle, draw, inner sep=0pt, minimum size=3.5pt, label=below:$\scriptscriptstyle 4$] (n2) at (3, 1  ) {};
                \node[circle, draw, inner sep=0pt, minimum size=3.5pt, label=below:$\scriptscriptstyle 4$] (n3) at (2, 1  ) {};
                \node[circle, draw, inner sep=0pt, minimum size=3.5pt, label=below:$\scriptscriptstyle 4$] (n4) at (1, 1  ) {};
                \node[circle, draw, inner sep=0pt, minimum size=3.5pt, label=below:$\scriptscriptstyle 4$] (n5) at (0, 1  ) {};
                \node[circle, draw, inner sep=0pt, minimum size=3.5pt, label=below:$\scriptscriptstyle 5$] (n6) at (3, 2  ) {};
                \node[circle, draw, inner sep=0pt, minimum size=3.5pt, label=below:$\scriptscriptstyle 5$] (n7) at (2, 2  ) {};
                \node[circle, draw, inner sep=0pt, minimum size=3.5pt, label=below:$\scriptscriptstyle 5$] (n8) at (1, 2  ) {};
                \node[circle, draw, inner sep=0pt, minimum size=3.5pt, label=right:$\scriptscriptstyle 9$] (n9) at (3, 0  ) {};
                \foreach \x/\y in {1/2, 1/6, 1/9, 2/3, 3/4, 4/5, 6/7, 7/8}
                \draw [-, shorten <= 2.50pt, shorten >= 2.50pt] (n\x) to  (n\y);
\end{tikzpicture}} & \multicolumn{3}{@{\tsep}l@{\tsep}}{$\CF_4\colon\!\!\!\!$
        \tikzsetnextfilename{indef_f4}\begin{tikzpicture}[baseline={([yshift=-2.75pt]current bounding box)},label distance=-2pt,
                xscale=0.65, yscale=0.5]
                \node[circle, draw, inner sep=0pt, minimum size=3.5pt, label={[xshift=0.1ex]below:$\scriptscriptstyle \mppmss 21$}] (n1) at (6, 1  ) {};
                \node[circle, draw, inner sep=0pt, minimum size=3.5pt, label=below:$\scriptscriptstyle 3$] (n2) at (5, 1  ) {};
                \node[circle, draw, inner sep=0pt, minimum size=3.5pt, label=below:$\scriptscriptstyle 3$] (n3) at (4, 1  ) {};
                \node[circle, draw, inner sep=0pt, minimum size=3.5pt, label=below:$\scriptscriptstyle 3$] (n4) at (3, 1  ) {};
                \node[circle, draw, inner sep=0pt, minimum size=3.5pt, label=below:$\scriptscriptstyle 3$] (n5) at (2, 1  ) {};
                \node[circle, draw, inner sep=0pt, minimum size=3.5pt, label=below:$\scriptscriptstyle 3$] (n6) at (1, 1  ) {};
                \node[circle, draw, inner sep=0pt, minimum size=3.5pt, label=below:$\scriptscriptstyle 3$] (n7) at (0, 1  ) {};
                \node[circle, draw, inner sep=0pt, minimum size=3.5pt, label=below:$\scriptscriptstyle 7$] (n8) at (5, 2  ) {};
                \node[circle, draw, inner sep=0pt, minimum size=3.5pt, label=below:$\scriptscriptstyle 7$] (n9) at (4, 2  ) {};
                \node[circle, draw, inner sep=0pt, minimum size=3.5pt, label=right:$\scriptscriptstyle 10$] (n10) at (5, 0  ) {};
                \foreach \x/\y in {1/2, 1/8, 1/10, 2/3, 3/4, 4/5, 5/6, 6/7, 8/9}
                \draw [-, shorten <= 2.50pt, shorten >= 2.50pt] (n\x) to  (n\y);
\end{tikzpicture}}\\[-0.3ex]
\multicolumn{4}{@{\tsep}c@{\tsep}}{$\CF_5\colon\!\!\!\!$
        \tikzsetnextfilename{indef_f5}\begin{tikzpicture}[baseline={([yshift=-2.75pt]current bounding box)},label distance=-2pt,
                xscale=0.52, yscale=0.5]
                \node[circle, draw, inner sep=0pt, minimum size=3.5pt, label=below:$\scriptscriptstyle \mppmss 7$] (n1) at (4, 1  ) {};
                \node[circle, draw, inner sep=0pt, minimum size=3.5pt, label=below:$\scriptscriptstyle 5$] (n2) at (0, 0  ) {};
                \node[circle, draw, inner sep=0pt, minimum size=3.5pt, label=below:$\scriptscriptstyle \mppmss 2$] (n3) at (1, 0  ) {};
                \node[circle, draw, inner sep=0pt, minimum size=3.5pt, label=below:$\scriptscriptstyle \mppmss 2$] (n4) at (2, 0  ) {};
                \node[circle, draw, inner sep=0pt, minimum size=3.5pt, label=below:$\scriptscriptstyle \mppmss 2$] (n5) at (3, 0  ) {};
                \node[circle, draw, inner sep=0pt, minimum size=3.5pt, label=below:$\scriptscriptstyle \mppmss 2$] (n6) at (4, 0  ) {};
                \node[circle, draw, inner sep=0pt, minimum size=3.5pt, label=below:$\scriptscriptstyle \mppmss 2$] (n7) at (5, 0  ) {};
                \node[circle, draw, inner sep=0pt, minimum size=3.5pt, label=below:$\scriptscriptstyle \mppmss 2$] (n8) at (6, 0  ) {};
                \node[circle, draw, inner sep=0pt, minimum size=3.5pt, label=below:$\scriptscriptstyle 5$] (n9) at (7, 0  ) {};
                \node[circle, draw, inner sep=0pt, minimum size=3.5pt, label=below:$\scriptscriptstyle 4$] (n10) at (8, 0  ) {};
                \foreach \x/\y in {1/9, 2/1}
                \draw [bend left=15.0, -stealth, shorten <= 2.50pt, shorten >= 2.50pt] (n\x) to  (n\y);
                \foreach \x/\y in {2/3, 3/4, 4/5, 5/6, 6/7, 7/8, 8/9}
                \draw [-stealth, shorten <= 2.50pt, shorten >= 2.50pt] (n\x) to  (n\y);
                \draw [-, shorten <= 2.50pt, shorten >= 2.50pt] (n9) to  (n10);
\end{tikzpicture}} & \multicolumn{4}{@{\tsep}c@{\tsep}}{$\CF_6\colon\!\!\!\!$
        \tikzsetnextfilename{indef_f6}\begin{tikzpicture}[baseline={([yshift=-2.75pt]current bounding box)},label distance=-2pt,
                xscale=0.52, yscale=0.5]
                \node[circle, draw, inner sep=0pt, minimum size=3.5pt, label=below:$\scriptscriptstyle \mppmss 7$] (n1) at (4, 1  ) {};
                \node[circle, draw, inner sep=0pt, minimum size=3.5pt, label=below:$\scriptscriptstyle 5$] (n2) at (0, 0  ) {};
                \node[circle, draw, inner sep=0pt, minimum size=3.5pt, label=below:$\scriptscriptstyle 5$] (n3) at (1, 0  ) {};
                \node[circle, draw, inner sep=0pt, minimum size=3.5pt, label=below:$\scriptscriptstyle \mppmss 2$] (n4) at (2, 0  ) {};
                \node[circle, draw, inner sep=0pt, minimum size=3.5pt, label=below:$\scriptscriptstyle \mppmss 2$] (n5) at (3, 0  ) {};
                \node[circle, draw, inner sep=0pt, minimum size=3.5pt, label=below:$\scriptscriptstyle \mppmss 2$] (n6) at (4, 0  ) {};
                \node[circle, draw, inner sep=0pt, minimum size=3.5pt, label=below:$\scriptscriptstyle \mppmss 2$] (n7) at (5, 0  ) {};
                \node[circle, draw, inner sep=0pt, minimum size=3.5pt, label=below:$\scriptscriptstyle \mppmss 2$] (n8) at (6, 0  ) {};
                \node[circle, draw, inner sep=0pt, minimum size=3.5pt, label=below:$\scriptscriptstyle \mppmss 2$] (n9) at (7, 0  ) {};
                \node[circle, draw, inner sep=0pt, minimum size=3.5pt, label=below:$\scriptscriptstyle 4$] (n10) at (8, 0  ) {};
                \draw [-, shorten <= 2.50pt, shorten >= 2.50pt] (n2) to  (n3);
                \foreach \x/\y in {3/4, 4/5, 5/6, 6/7, 7/8, 8/9, 9/10}
                \draw [-stealth, shorten <= 2.50pt, shorten >= 2.50pt] (n\x) to  (n\y);
                \foreach \x/\y in {1/10, 3/1}
                \draw [bend left=15.0, -stealth, shorten <= 2.50pt, shorten >= 2.50pt] (n\x) to  (n\y);
\end{tikzpicture}} & \multicolumn{4}{@{\tsep}l@{\tsep}}{$\CF_7\colon\!\!\!\!$
        \tikzsetnextfilename{indef_f7}\begin{tikzpicture}[baseline={([yshift=-2.75pt]current bounding box)},label distance=-2pt,
                xscale=0.52, yscale=0.5]
                \node[circle, draw, inner sep=0pt, minimum size=3.5pt, label=below:$\scriptscriptstyle \mppmss 7$] (n1) at (0, 0  ) {};
                \node[circle, draw, inner sep=0pt, minimum size=3.5pt, label=below:$\scriptscriptstyle 2$] (n2) at (1, 0  ) {};
                \node[circle, draw, inner sep=0pt, minimum size=3.5pt, label=below:$\scriptscriptstyle 2$] (n3) at (2, 0  ) {};
                \node[circle, draw, inner sep=0pt, minimum size=3.5pt, label=below:$\scriptscriptstyle 2$] (n4) at (3, 0  ) {};
                \node[circle, draw, inner sep=0pt, minimum size=3.5pt, label=below:$\scriptscriptstyle 2$] (n5) at (4, 0  ) {};
                \node[circle, draw, inner sep=0pt, minimum size=3.5pt, label=below:$\scriptscriptstyle 2$] (n6) at (5, 0  ) {};
                \node[circle, draw, inner sep=0pt, minimum size=3.5pt, label=below:$\scriptscriptstyle 2$] (n7) at (6, 0  ) {};
                \node[circle, draw, inner sep=0pt, minimum size=3.5pt, label=below:$\scriptscriptstyle \mppmss 7$] (n8) at (7, 0  ) {};
                \node[circle, draw, inner sep=0pt, minimum size=3.5pt, label=below:$\scriptscriptstyle 5$] (n9) at (3, 1  ) {};
                \node[circle, draw, inner sep=0pt, minimum size=3.5pt, label=below:$\scriptscriptstyle 4$] (n10) at (4, 1  ) {};
                \foreach \x/\y in {1/9, 10/8}
                \draw [bend left=15.0, -stealth, shorten <= 2.50pt, shorten >= 2.50pt] (n\x) to  (n\y);
                \foreach \x/\y in {1/2, 2/3, 3/4, 4/5, 5/6, 6/7, 7/8, 9/10}
                \draw [-stealth, shorten <= 2.50pt, shorten >= 2.50pt] (n\x) to  (n\y);
\end{tikzpicture}}\\[-1ex]
\caption{Indefinite posets}\label{tbl:indefposets}
\end{longtable}}
\end{lemma}
\begin{proof}
\vspace*{-0.5ex}First, we recall that  an `unoriented edge' in a diagram means that both arrow orientations are permissible, see \Cref{remark:hassepath:anyorient}. In other words, diagrams $\CF_1,\ldots,\CF_{7}$ describe $777$ different, non-isomorphic posets. 

By \Cref{corr:ahchoredpath}\ref{corr:ahchoredpath:anyorient}, without loss of generality, we may assume that all unoriented edges in diagrams presented in \Cref{tbl:indefposets} are oriented ``from left to right''. That is,  the arcs in the Hasse digraph $\vec \CF_i$, where $i\in\{1,\ldots,7\}$ are oriented in such a way that:
\begin{itemize}
\item $\smash{\vec \CF_1}$ has three sinks,
\item $\smash{\vec \CF_2}$ has two sinks and
\item $\smash{\vec \CF_3,\ldots,\vec \CF_{7}}$ have exactly one sink. 
\end{itemize}

It is straightforward to  check that $q_{I_i}(v^i)<0$, where $I_i\eqdef I_{\vec \CF_i}$ is a poset defined by a digraph $\smash{\vec \CF_i\in\{\vec \CF_1, \ldots,\vec \CF_{7} \}}$ and $v^i\in\ZZ^n$ is an integer vector whose coordinates are given as the labels of vertices of diagram $\CF_i$ in \Cref{tbl:indefposets}. Therefore, by Remark \ref{rmk:indef}, we conclude that poset $I_i$ is indefinite. Since  isomorphic posets are weakly Gram $\ZZ$-congruent (the congruence is defined by a permutation matrix), it follows that every poset that has its Hasse digraph  isomorphic with $\CF_i$ is indefinite.  
\end{proof}

We need one more technical result to prove \Cref{thm:a:main}. In \Cref{lemma:pathext} we describe connected posets $I$ with the following property:  for some $p\in \{1,\ldots,n\}$ the graph $\ov\CH(I^{(p)})$ is isomorphic to a path graph.
\begin{lemma}\label{lemma:pathext}
Assume that $I$ is such a  connected poset, that for some $p\in \{1,\ldots,n\}$ the graph $\ov\CH(I^{(p)})$ is isomorphic to a path graph. Exactly one of the following conditions holds.
\begin{enumerate}[label=\normalfont{(\alph*)}]
\item\label{lemma:pathext:posit} $I$ is positive and:

\begin{enumerate}[label=\normalfont{(a\arabic*)}, leftmargin=3ex]
        \item\label{lemma:pathext:posit:a} $\Dyn_I=\AA_n$ with $\CH(I)\simeq \CA_n=$\tikzsetnextfilename{pth_an}\begin{tikzpicture}[baseline=(n7.base),label distance=-2pt,xscale=0.64, yscale=0.64]
                \node[circle, fill=black, inner sep=0pt, minimum size=3.5pt, label=below:$\scriptscriptstyle 1$] (n1) at (0, 0  ) {};
                \node[circle, fill=black, inner sep=0pt, minimum size=3.5pt, label=below:$\scriptscriptstyle 2$] (n2) at (1, 0  ) {};
                \node[circle, fill=black, inner sep=0pt, minimum size=3.5pt, label=below:$\scriptscriptstyle 3$] (n4) at (2, 0  ) {};
                \node[circle, fill=black, inner sep=0pt, minimum size=3.5pt, label=below:$\scriptscriptstyle n\mppmss 1$] (n5) at (5, 0  ) {};
                \node[circle, fill=black, inner sep=0pt, minimum size=3.5pt, label=below:$\scriptscriptstyle n$] (n6) at (6, 0  ) {};
                \node (n7) at (3, 0  ) {$\mathclap{\phantom{7}}$};
                \node (n8) at (4, 0  ) {};
                \foreach \x/\y in {1/2, 2/4, 5/6}
                \draw [-, shorten <= 2.50pt, shorten >= 2.50pt] (n\x) to  (n\y);
                \draw [-, shorten <= 2.50pt, shorten >= -2.50pt] (n4) to  (n7);
                \draw [line width=1.2pt, line cap=round, dash pattern=on 0pt off 5\pgflinewidth, -, shorten <= -0.50pt, shorten >= -2.50pt] (n7) to  (n8);
                \draw [-, shorten <= -2.50pt, shorten >= 2.50pt] (n8) to  (n5);
        \end{tikzpicture};
        
        \item\label{lemma:pathext:posit:d} $\Dyn_I=\DD_n$ and the Hasse digraph $\CH(I)\simeq \CD_I$ has a shape $\CD_I\in\{\CD_n^{[1]},\CD_{n,s}^{[2]},\CD_{n,s}^{[3]}\}$,
        
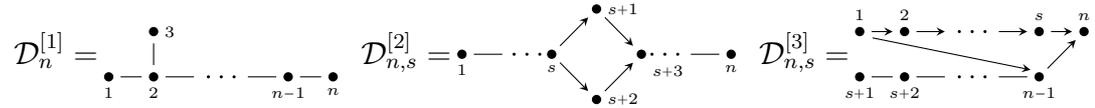
        \begin{center}
                \newcommand{\mxs}{0.59}
            \newcommand{\mys}{0.6}
                $\CD_n^{[1]}=$\!\!\!
                \tikzsetnextfilename{d1n}\begin{tikzpicture}[baseline=(nb.base),label distance=-2pt,xscale=\mxs, yscale=\mys]
                        \node[circle, fill=black, inner sep=0pt, minimum size=3.5pt, label=below:$\scriptscriptstyle 1$] (n1) at (0, 0  ) {};
                        \node[circle, fill=black, inner sep=0pt, minimum size=3.5pt, label=below:$\scriptscriptstyle 2$] (n2) at (1, 0  ) {};
                        \node[circle, fill=black, inner sep=0pt, minimum size=3.5pt, label=right:$\scriptscriptstyle 3$] (n3) at (1, 1  ) {};
                        \node (n4) at (2, 0  ) {};
                        \node[circle, fill=black, inner sep=0pt, minimum size=3.5pt, label=below:$\scriptscriptstyle n\mppmss 1$] (n5) at (4, 0  ) {};
                        \node[circle, fill=black, inner sep=0pt, minimum size=3.5pt, label=below:$\scriptscriptstyle n$] (n6) at (5, 0  ) {};
                        \node (n8) at (3, 0  ) {};
                        \node (nb) at (0.50, 0.50) {\phantom{7}};
                        \foreach \x/\y in {1/2, 2/3,  5/6}
                        \draw [-, shorten <= 2.50pt, shorten >= 2.50pt] (n\x) to  (n\y);
                        \draw [-, shorten <= 2.50pt, shorten >= -2.50pt] (n2) to  (n4);
                        \draw [line width=1.2pt, line cap=round, dash pattern=on 0pt off 4\pgflinewidth, -, shorten <= -0.50pt, shorten >= -2.50pt] (n4) to  (n8);
                        \draw [-, shorten <= -2.50pt, shorten >= 2.50pt] (n8) to  (n5);
                \end{tikzpicture}
                $\CD_{n,s}^{[2]}=$\!\!\!
                \tikzsetnextfilename{d2n}\begin{tikzpicture}[baseline=(n6.base),label distance=-2pt,
                        xscale=\mxs, yscale=\mys]
                        \node[circle, fill=black, inner sep=0pt, minimum size=3.5pt, label=below:$\scriptscriptstyle s$] (n1) at (3, 1  ) {};
                        \node[circle, fill=black, inner sep=0pt, minimum size=3.5pt, label=right:$\scriptscriptstyle s\mpppss 1$] (n2) at (4, 2  ) {};
                        \node[circle, fill=black, inner sep=0pt, minimum size=3.5pt, label=right:$\scriptscriptstyle s\mpppss 2$] (n3) at (4, 0  ) {};
                        \node[circle, fill=black, inner sep=0pt, minimum size=3.5pt, label=below:$\scriptscriptstyle 1$] (n5) at (1, 1  ) {};
                        \node (n6) at (2, 1  ) {$\mathclap{\phantom{7}}$};
                        
                        \node[circle, fill=black, inner sep=0pt, minimum size=3.5pt, label=below right:$\scriptscriptstyle s\mpppss 3$] (n4) at (5, 1  ) {};
                        \node (n9) at (6, 1  ) {};
                        \node[circle, fill=black, inner sep=0pt, minimum size=3.5pt, label=below:$\scriptscriptstyle n$] (n10) at (7, 1  ) {};
                        \foreach \x/\y in {1/2, 1/3, 2/4, 3/4}
                        \draw [-stealth, shorten <= 2.50pt, shorten >= 2.50pt] (n\x) to  (n\y);
                        \draw [-, shorten <= 2.50pt, shorten >= -2.50pt] (n5) to  (n6);
                        \draw [line width=1.2pt, line cap=round, dash pattern=on 0pt off 4\pgflinewidth, -, shorten <= -1.5pt, shorten >= -.50pt] (n6) to  (n1);
                        \draw [line width=1.2pt, line cap=round, dash pattern=on 0pt off 4\pgflinewidth, -, shorten <= 2.pt, shorten >= -1.50pt] (n4) to  (n9);
                        \foreach \x/\y in { 9/10}
                        \draw [-, shorten <= -2.50pt, shorten >= 2.50pt] (n\x) to  (n\y);
                \end{tikzpicture}
                $\CD_{n,s}^{[3]}=$\!\!\!\!\!  
                \tikzsetnextfilename{d3n}\begin{tikzpicture}[baseline=(nb.base),label distance=-2pt,
                        xscale=\mxs, yscale=\mys]
                        \node[circle, fill=black, inner sep=0pt, minimum size=3.5pt, label=above:$\scriptscriptstyle 1$] (n1) at (0, 1  ) {};
                        \node[circle, fill=black, inner sep=0pt, minimum size=3.5pt, label=above:$\scriptscriptstyle 2$] (n2) at (1, 1  ) {};
                        \node[circle, fill=black, inner sep=0pt, minimum size=3.5pt, label=above:$\scriptscriptstyle s$] (n3) at (4, 1  ) {};
                        \node[circle, fill=black, inner sep=0pt, minimum size=3.5pt, label=above:$\scriptscriptstyle n$] (n4) at (5, 1  ) {};
                        \node (n5) at (2, 1  ) {};
                        \node (n6) at (3, 1  ) {};
                        \node[circle, fill=black, inner sep=0pt, minimum size=3.5pt, label=below:$\scriptscriptstyle s\mpppss 1$] (n7) at (0, 0  ) {};
                        \node[circle, fill=black, inner sep=0pt, minimum size=3.5pt, label=below:$\scriptscriptstyle s\mpppss 2$] (n8) at (1, 0  ) {};
                        \node[circle, fill=black, inner sep=0pt, minimum size=3.5pt, label=below:$\scriptscriptstyle n\mppmss 1$] (n9) at (4, 0  ) {};
                        \node (n10) at (2, 0  ) {};
                        \node (n11) at (3, 0  ) {};
                        \node (nb) at (0.50, .50) {\phantom{7}};
                        \foreach \x/\y in {1/2, 3/4, 9/4}
                        \draw [-stealth, shorten <= 2.50pt, shorten >= 2.50pt] (n\x) to  (n\y);
                        \draw [-stealth, shorten <= 2.50pt, shorten >= -2.50pt] (n2) to  (n5);
                        \foreach \x/\y in {5/6, 10/11}
                        \draw [line width=1.2pt, line cap=round, dash pattern=on 0pt off 4\pgflinewidth, -, shorten <= -0.50pt, shorten >= -2.50pt] (n\x) to  (n\y);
                        \draw [-stealth, shorten <= -2.50pt, shorten >= 2.50pt] (n6) to  (n3);
                        \draw [-, shorten <= 2.50pt, shorten >= 2.50pt] (n7) to  (n8);
                        \draw [-, shorten <= 2.50pt, shorten >= -2.50pt] (n8) to  (n10);
                        \draw [-, shorten <= -2.50pt, shorten >= 2.50pt] (n11) to  (n9);
                        \draw [-stealth, shorten <= 3.50pt, shorten >= 1.90pt] ([yshift=-0.5]n1.south east) to  ([yshift=1.5]n9.north west);
                \end{tikzpicture} 
        \end{center}
        where $n\geq 4$ for $\CD_{I}=\CD_n^{[1]}$; $n\geq 4$, $s\geq 1$ for $\CD_{n,s}^{[2]}$, and $n\geq 5$, $s\geq 2$ for $\CD_{n,s}^{[3]}$.
        
        \item\label{lemma:pathext:posit:e} $\Dyn_I=\EE_n$, where $n\in\{6,7,8\}$, and the Hasse digraph $\CH(I)$, up to isomorphism, is one of  $498$ digraphs \textnormal{[}$86$~up to orientation of hanging paths, see \Cref{corr:ahchoredpath}\ref{corr:ahchoredpath:anyorient}\textnormal{]}, i.e.,  there are  $38, 145, 315$ \textnormal{[}$11, 30, 45$\textnormal{]} digraphs with $\Dyn_I=\EE_6, \EE_7$ and $\EE_8$, respectively. In particular, Hasse digraphs $\CH(I)$ of all posets $I$ with $\Dyn_I=\EE_6$ are depicted below.
        \begin{center}
                {\newcommand{\mxscale}{0.65}
                \newcommand{\myscale}{0.55}
                \hfill%
                \tikzsetnextfilename{p_e6_1}
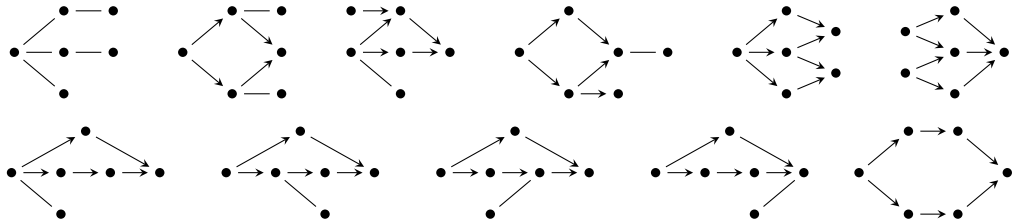
\begin{tikzpicture}[baseline={([yshift=-2.75pt]current bounding box)},label distance=-2pt,xscale=\mxscale, yscale=\myscale]
                        \node[circle, fill=black, inner sep=0pt, minimum size=3.5pt] (n1) at (0, 1  ) {};
                        \node[circle, fill=black, inner sep=0pt, minimum size=3.5pt] (n2) at (1, 2  ) {};
                        \node[circle, fill=black, inner sep=0pt, minimum size=3.5pt] (n3) at (1, 1  ) {};
                        \node[circle, fill=black, inner sep=0pt, minimum size=3.5pt] (n4) at (1, 0  ) {};
                        \node[circle, fill=black, inner sep=0pt, minimum size=3.5pt] (n5) at (2, 2  ) {};
                        \node[circle, fill=black, inner sep=0pt, minimum size=3.5pt] (n6) at (2, 1  ) {};
                        \foreach \x/\y in {1/2, 1/3, 1/4, 2/5, 3/6}
                        \draw [-, shorten <= 2.50pt, shorten >= 2.50pt] (n\x) to  (n\y);
                \end{tikzpicture} \hfill 
                \tikzsetnextfilename{p_e6_2}\begin{tikzpicture}[baseline={([yshift=-2.75pt]current bounding box)},label distance=-2pt,xscale=\mxscale, yscale=\myscale]
                        \node[circle, fill=black, inner sep=0pt, minimum size=3.5pt] (n1) at (0, 1  ) {};
                        \node[circle, fill=black, inner sep=0pt, minimum size=3.5pt] (n2) at (1, 2  ) {};
                        \node[circle, fill=black, inner sep=0pt, minimum size=3.5pt] (n3) at (1, 0  ) {};
                        \node[circle, fill=black, inner sep=0pt, minimum size=3.5pt] (n4) at (2, 2  ) {};
                        \node[circle, fill=black, inner sep=0pt, minimum size=3.5pt] (n5) at (2, 1  ) {};
                        \node[circle, fill=black, inner sep=0pt, minimum size=3.5pt] (n6) at (2, 0  ) {};
                        \foreach \x/\y in {1/2, 1/3, 2/5, 3/5}
                        \draw [-stealth, shorten <= 2.50pt, shorten >= 2.50pt] (n\x) to  (n\y);
                        \foreach \x/\y in {2/4, 3/6}
                        \draw [-, shorten <= 2.50pt, shorten >= 2.50pt] (n\x) to  (n\y);
                \end{tikzpicture} \hfill
                \tikzsetnextfilename{p_e6_3}\begin{tikzpicture}[baseline={([yshift=-2.75pt]current bounding box)},label distance=-2pt,xscale=\mxscale, yscale=\myscale]
                        \node[circle, fill=black, inner sep=0pt, minimum size=3.5pt] (n1) at (0, 1  ) {};
                        \node[circle, fill=black, inner sep=0pt, minimum size=3.5pt] (n2) at (1, 1  ) {};
                        \node[circle, fill=black, inner sep=0pt, minimum size=3.5pt] (n3) at (1, 2  ) {};
                        \node[circle, fill=black, inner sep=0pt, minimum size=3.5pt] (n4) at (1, 0  ) {};
                        \node[circle, fill=black, inner sep=0pt, minimum size=3.5pt] (n5) at (2, 1  ) {};
                        \node[circle, fill=black, inner sep=0pt, minimum size=3.5pt] (n6) at (0, 2  ) {};
                        \foreach \x/\y in {1/2, 1/3, 2/5, 3/5, 6/3}
                        \draw [-stealth, shorten <= 2.50pt, shorten >= 2.50pt] (n\x) to  (n\y);
                        \draw [-, shorten <= 2.50pt, shorten >= 2.50pt] (n1) to  (n4);
                \end{tikzpicture} \hfill
                \tikzsetnextfilename{p_e6_4}\begin{tikzpicture}[baseline={([yshift=-2.75pt]current bounding box)},label distance=-2pt,xscale=\mxscale, yscale=\myscale]
                        \node[circle, fill=black, inner sep=0pt, minimum size=3.5pt] (n1) at (0, 1  ) {};
                        \node[circle, fill=black, inner sep=0pt, minimum size=3.5pt] (n2) at (1, 2  ) {};
                        \node[circle, fill=black, inner sep=0pt, minimum size=3.5pt] (n3) at (2, 1  ) {};
                        \node[circle, fill=black, inner sep=0pt, minimum size=3.5pt] (n4) at (3, 1  ) {};
                        \node[circle, fill=black, inner sep=0pt, minimum size=3.5pt] (n5) at (1, 0  ) {};
                        \node[circle, fill=black, inner sep=0pt, minimum size=3.5pt] (n6) at (2, 0  ) {};
                        \foreach \x/\y in {1/2, 1/5, 2/3, 5/3, 5/6}
                        \draw [-stealth, shorten <= 2.50pt, shorten >= 2.50pt] (n\x) to  (n\y);
                        \draw [-, shorten <= 2.50pt, shorten >= 2.50pt] (n3) to  (n4);
                \end{tikzpicture} \hfill
                \tikzsetnextfilename{p_e6_5}\begin{tikzpicture}[baseline={([yshift=-2.75pt]current bounding box)},label distance=-2pt,xscale=\mxscale, yscale=\myscale]
                        \node[circle, fill=black, inner sep=0pt, minimum size=3.5pt] (n1) at (0, 1  ) {};
                        \node[circle, fill=black, inner sep=0pt, minimum size=3.5pt] (n2) at (1, 2  ) {};
                        \node[circle, fill=black, inner sep=0pt, minimum size=3.5pt] (n3) at (1, 1  ) {};
                        \node[circle, fill=black, inner sep=0pt, minimum size=3.5pt] (n4) at (1, 0  ) {};
                        \node[circle, fill=black, inner sep=0pt, minimum size=3.5pt] (n5) at (2, 1.50) {};
                        \node[circle, fill=black, inner sep=0pt, minimum size=3.5pt] (n6) at (2, 0.50) {};
                        \foreach \x/\y in {1/2, 1/3, 1/4, 2/5, 3/5, 3/6, 4/6}
                        \draw [-stealth, shorten <= 2.50pt, shorten >= 2.50pt] (n\x) to  (n\y);
                \end{tikzpicture} \hfill 
                \tikzsetnextfilename{p_e6_6}\begin{tikzpicture}[baseline={([yshift=-2.75pt]current bounding box)},label distance=-2pt,xscale=\mxscale, yscale=\myscale]
                        \node[circle, fill=black, inner sep=0pt, minimum size=3.5pt] (n1) at (0, 1.50) {};
                        \node[circle, fill=black, inner sep=0pt, minimum size=3.5pt] (n2) at (0, 0.50) {};
                        \node[circle, fill=black, inner sep=0pt, minimum size=3.5pt] (n3) at (1, 2  ) {};
                        \node[circle, fill=black, inner sep=0pt, minimum size=3.5pt] (n4) at (1, 1  ) {};
                        \node[circle, fill=black, inner sep=0pt, minimum size=3.5pt] (n5) at (1, 0  ) {};
                        \node[circle, fill=black, inner sep=0pt, minimum size=3.5pt] (n6) at (2, 1  ) {};
                        \foreach \x/\y in {1/3, 1/4, 2/4, 2/5, 3/6, 4/6, 5/6}
                        \draw [-stealth, shorten <= 2.50pt, shorten >= 2.50pt] (n\x) to  (n\y);
                \end{tikzpicture}\hfill\mbox{} \\[0.34cm]
               \hfill 
               \tikzsetnextfilename{p_e6_7}\begin{tikzpicture}[baseline={([yshift=-2.75pt]current bounding box)},label distance=-2pt,xscale=\mxscale, yscale=\myscale]
                        \node[circle, fill=black, inner sep=0pt, minimum size=3.5pt] (n1) at (0, 1  ) {};
                        \node[circle, fill=black, inner sep=0pt, minimum size=3.5pt] (n2) at (1, 1  ) {};
                        \node[circle, fill=black, inner sep=0pt, minimum size=3.5pt] (n3) at (2, 1  ) {};
                        \node[circle, fill=black, inner sep=0pt, minimum size=3.5pt] (n4) at (3, 1  ) {};
                        \node[circle, fill=black, inner sep=0pt, minimum size=3.5pt] (n5) at (1.50, 2  ) {};
                        \node[circle, fill=black, inner sep=0pt, minimum size=3.5pt] (n6) at (1, 0  ) {};
                        \foreach \x/\y in {1/2, 1/5, 2/3, 3/4, 5/4}
                        \draw [-stealth, shorten <= 2.50pt, shorten >= 2.50pt] (n\x) to  (n\y);
                        \draw [-, shorten <= 2.50pt, shorten >= 2.50pt] (n1) to  (n6);
                \end{tikzpicture} \hfill 
                \tikzsetnextfilename{p_e6_8}\begin{tikzpicture}[baseline={([yshift=-2.75pt]current bounding box)},label distance=-2pt,xscale=\mxscale, yscale=\myscale]
                        \node[circle, fill=black, inner sep=0pt, minimum size=3.5pt] (n1) at (0, 1  ) {};
                        \node[circle, fill=black, inner sep=0pt, minimum size=3.5pt] (n2) at (1, 1  ) {};
                        \node[circle, fill=black, inner sep=0pt, minimum size=3.5pt] (n3) at (2, 1  ) {};
                        \node[circle, fill=black, inner sep=0pt, minimum size=3.5pt] (n4) at (3, 1  ) {};
                        \node[circle, fill=black, inner sep=0pt, minimum size=3.5pt] (n5) at (1.50, 2  ) {};
                        \node[circle, fill=black, inner sep=0pt, minimum size=3.5pt] (n6) at (2, 0  ) {};
                        \foreach \x/\y in {1/2, 1/5, 2/3, 3/4, 5/4}
                        \draw [-stealth, shorten <= 2.50pt, shorten >= 2.50pt] (n\x) to  (n\y);
                        \draw [-, shorten <= 2.50pt, shorten >= 2.50pt] (n2) to  (n6);
                \end{tikzpicture} \hfill 
                \tikzsetnextfilename{p_e6_9}\begin{tikzpicture}[baseline={([yshift=-2.75pt]current bounding box)},label distance=-2pt,xscale=\mxscale, yscale=\myscale]
                        \node[circle, fill=black, inner sep=0pt, minimum size=3.5pt] (n1) at (0, 1  ) {};
                        \node[circle, fill=black, inner sep=0pt, minimum size=3.5pt] (n2) at (1, 1  ) {};
                        \node[circle, fill=black, inner sep=0pt, minimum size=3.5pt] (n3) at (2, 1  ) {};
                        \node[circle, fill=black, inner sep=0pt, minimum size=3.5pt] (n4) at (3, 1  ) {};
                        \node[circle, fill=black, inner sep=0pt, minimum size=3.5pt] (n5) at (1.50, 2  ) {};
                        \node[circle, fill=black, inner sep=0pt, minimum size=3.5pt] (n6) at (1, 0  ) {};
                        \foreach \x/\y in {1/2, 1/5, 2/3, 3/4, 5/4}
                        \draw [-stealth, shorten <= 2.50pt, shorten >= 2.50pt] (n\x) to  (n\y);
                        \draw [-, shorten <= 2.50pt, shorten >= 2.50pt] (n3) to  (n6);
                \end{tikzpicture} \hfill 
                \tikzsetnextfilename{p_e6_10}\begin{tikzpicture}[baseline={([yshift=-2.75pt]current bounding box)},label distance=-2pt,xscale=\mxscale, yscale=\myscale]
                        \node[circle, fill=black, inner sep=0pt, minimum size=3.5pt] (n1) at (0, 1  ) {};
                        \node[circle, fill=black, inner sep=0pt, minimum size=3.5pt] (n2) at (1, 1  ) {};
                        \node[circle, fill=black, inner sep=0pt, minimum size=3.5pt] (n3) at (2, 1  ) {};
                        \node[circle, fill=black, inner sep=0pt, minimum size=3.5pt] (n4) at (3, 1  ) {};
                        \node[circle, fill=black, inner sep=0pt, minimum size=3.5pt] (n5) at (1.50, 2  ) {};
                        \node[circle, fill=black, inner sep=0pt, minimum size=3.5pt] (n6) at (2, 0  ) {};
                        \foreach \x/\y in {1/2, 1/5, 2/3, 3/4, 5/4}
                        \draw [-stealth, shorten <= 2.50pt, shorten >= 2.50pt] (n\x) to  (n\y);
                        \draw [-, shorten <= 2.50pt, shorten >= 2.50pt] (n4) to  (n6);
                \end{tikzpicture}\hfill 
                \tikzsetnextfilename{p_e6_11}\begin{tikzpicture}[baseline={([yshift=-2.75pt]current bounding box)},label distance=-2pt,xscale=\mxscale, yscale=\myscale]
                        \node[circle, fill=black, inner sep=0pt, minimum size=3.5pt] (n1) at (0, 1  ) {};
                        \node[circle, fill=black, inner sep=0pt, minimum size=3.5pt] (n2) at (1, 2  ) {};
                        \node[circle, fill=black, inner sep=0pt, minimum size=3.5pt] (n3) at (2, 2  ) {};
                        \node[circle, fill=black, inner sep=0pt, minimum size=3.5pt] (n4) at (3, 1  ) {};
                        \node[circle, fill=black, inner sep=0pt, minimum size=3.5pt] (n5) at (1, 0  ) {};
                        \node[circle, fill=black, inner sep=0pt, minimum size=3.5pt] (n6) at (2, 0  ) {};
                        \foreach \x/\y in {1/2, 1/5, 2/3, 3/4, 5/6, 6/4}
                        \draw [-stealth, shorten <= 2.50pt, shorten >= 2.50pt] (n\x) to  (n\y);
                \end{tikzpicture}\hfill\mbox{}
        }\end{center}
\end{enumerate}
We note that, up to isomorphism, the first Hasse digraph describes exactly $20$ posets and the second exactly $3$ posets.
\item\label{lemma:pathext:princ} $I$ is principal and:
\begin{enumerate}[label=\normalfont{(b\arabic*)}, leftmargin=3ex]
    \item\label{lemma:pathext:princ:a} $\Dyn_I=\AA_{n-1}$, $\ov\CH(I)$ is a cycle graph and $\CH(I)$ has at least two sinks.
    \item\label{lemma:pathext:princ:e} $\Dyn_I=\EE_{n-1}$, where $n\in\{8,9\}$, and the Hasse digraph $\CH(I)$, up to isomorphism, is one of  $850$ \textnormal{digraphs [}$98$~up to orientation of hanging paths\textnormal{]}, i.e.,  there are exactly  $185, 665$ \textnormal{[}$36,62$\textnormal{]} digraphs with $\Dyn_I= \EE_7$ and $\EE_8$, respectively.
\end{enumerate}
\item\label{lemma:pathext:indef} $I$ is indefinite.\pagebreak
\end{enumerate}
\end{lemma}
\begin{proof}
Assume that $I$ is a connected poset of size $n$. By \Cref{thm:onepeak_posit} and \Cref{corr:ahchoredpath}\ref{corr:ahchoredpath:anyorient}, we know that posets $I$ with $\ov\CH(I)$ isomorphic to $\CA_n,\ab\CD_n^{[1]},\ab\CD_{n,s}^{[2]},\ab\CD_{n,s}^{[3]}$ are positive, with  $\Dyn_I=\AA_n$ if $\ov\CH(I)\simeq \CA_n$ and $\Dyn_I=\DD_n$ otherwise. Furthermore, \Cref{thm:digraphdegmax2}\ref{thm:digraphdegmax2:cycle:multsink} asserts that posets $I$ with $\ov \CH(I)$ isomorphic to a cycle graph  and $\CH(I)$ having at least two sinks are principal with $\Dyn_I=\AA_{n-1}$.\medskip

The proof is divided into two parts. First, we prove the thesis by analyzing all posets having at most $11$ elements. Then, using induction, we prove it for posets $I$ of size $|I|>11$.\smallskip

\textbf{Part $1^\circ$} It is easy to see that all connected  posets $I$ of size $n\leq 3$ satisfy the assumptions: in this case $\CH(I)\in\{\!\!
\tikzsetnextfilename{main_prf_1}\begin{tikzpicture}[baseline={([yshift=-2.75pt]current bounding box)},label distance=-2pt,xscale=0.65, yscale=0.74]
        \node[circle, fill=black, inner sep=0pt, minimum size=3.5pt, label=left:$\scriptstyle 1$] (n1) at (0  , 0  ) {};
        \node[circle, fill=black, inner sep=0pt, minimum size=3.5pt, label=right:$\scriptstyle 2$] (n2) at (1  , 0  ) {};
        \draw [-stealth, shorten <= 2.50pt, shorten >= 2.50pt] (n1) to  (n2);
\end{tikzpicture}\!\!,\!
\tikzsetnextfilename{main_prf_2}\begin{tikzpicture}[baseline={([yshift=-2.75pt]current bounding box)},label distance=-2pt,xscale=0.65, yscale=0.74]
        \node[circle, fill=black, inner sep=0pt, minimum size=3.5pt, label=left:$\scriptstyle 1$] (n1) at (0  , 0  ) {};
        \node[circle, fill=black, inner sep=0pt, minimum size=3.5pt, label=right:$\scriptstyle 2$] (n2) at (1  , 0  ) {};
        \node[circle, fill=black, inner sep=0pt, minimum size=3.5pt, label=right:$\scriptstyle 3$] (n3) at (2.35, 0  ) {};
        \draw [-stealth, shorten <= 2.50pt, shorten >= 2.50pt] (n1) to  (n2);
        \draw [-stealth, shorten <= 2.50pt, shorten >= 9.00pt] (n3) to  (n2);
\end{tikzpicture}\!\!,\!
\tikzsetnextfilename{main_prf_3}\begin{tikzpicture}[baseline={([yshift=-2.75pt]current bounding box)},label distance=-2pt,xscale=0.65, yscale=0.74]
        \node[circle, fill=black, inner sep=0pt, minimum size=3.5pt, label=left:$\scriptstyle 1$] (n1) at (0  , 0  ) {};
        \node[circle, fill=black, inner sep=0pt, minimum size=3.5pt, label=right:$\scriptstyle 2$] (n2) at (1  , 0  ) {};
        \node[circle, fill=black, inner sep=0pt, minimum size=3.5pt, label=right:$\scriptstyle 3$] (n3) at (2.35, 0  ) {};
        \draw [-stealth, shorten <= 2.50pt, shorten >= 2.50pt] (n2) to  (n1);
        \draw [-stealth, shorten <= 9.00pt, shorten >= 2.50pt] (n2) to  (n3);
\end{tikzpicture}\!\!,\!
\tikzsetnextfilename{main_prf_4}\begin{tikzpicture}[baseline={([yshift=-2.75pt]current bounding box)},label distance=-2pt,xscale=0.65, yscale=0.74]
        \node[circle, fill=black, inner sep=0pt, minimum size=3.5pt, label=left:$\scriptstyle 1$] (n1) at (0  , 0  ) {};
        \node[circle, fill=black, inner sep=0pt, minimum size=3.5pt, label=right:$\scriptstyle 2$] (n2) at (1  , 0  ) {};
        \node[circle, fill=black, inner sep=0pt, minimum size=3.5pt, label=right:$\scriptstyle 3$] (n3) at (2.35, 0  ) {};
        \draw [-stealth, shorten <= 2.50pt, shorten >= 2.50pt] (n1) to  (n2);
        \draw [-stealth, shorten <= 9.00pt, shorten >= 2.50pt] (n2) to  (n3);
\end{tikzpicture}\!\!
\}$. Since $\ov\CH(I^{(1)})$ is isomorphic to a path graph, $\Dyn_I=\AA_{n}$ and the thesis follows. Now, using Computer Algebra System (e.g. SageMath, Maple), we  compute all (up to isomorphism) posets $I$ of size at most $11$ using a suitably modified version of~\cite[Algorithm 7.1]{gasiorekOnepeakPosetsPositive2012} (see also~\cite{brinkmannPosets16Points2002} for a different approach). There are exactly $\num{49519383}$ [$\num{46485488}$ connected] posets $I$ of size $4\leq |I| \leq 11$.  Moreover, $\num{58723}$ [$\num{58198}$ connected] posets $I$ have the graph $\ov\CH(I^{(p)})$ isomorphic to a path graph, for some $p\in I$. 

In particular, there are $\num{46749427}$ [$\num{43944974}$ connected] non-isomorphic posets of size $11$. In $\num{39335}$ [$\num{39079}$] cases,  for some $p\in \{1,\ldots,11\}$, the graph $\ov\CH(I^{(p)})$ is isomorphic to a path graph and, up to isomorphism, there are exactly:
\begin{itemize}
	\item $256$ disconnected positive posets $I$ with $\CH(I)\simeq \,  \scriptstyle \bullet\hspace{22pt}\bullet\,\rule[1.5pt]{22pt}{0.4pt}\,\bullet\,\rule[1.5pt]{22pt}{0.4pt}\,\,\hdashrule[1.5pt]{12pt}{0.4pt}{1pt}\,\rule[1.5pt]{22pt}{0.4pt}\,\bullet$,
    \item $\num[group-minimum-digits = 4]{2575}$  connected positive posets $I$, of which $528$, $768$, $1024$ and $255$ have its Hasse digraph $\CH(I)$ isomorphic to $\CA_n$,  $\CD_n^{[1]}$, $\CD_{n,s}^{[2]}$, and $\CD_{n,s}^{[3]}$, respectively;
    \item $88$ connected principal posets $I$, where $\ov\CH(I)$ is a cycle graph and $\CH(I)$ has at least two sinks;
    \item $\num{36416}$ connected indefinite posets.
\end{itemize}

A more precise analysis of connected posets $I$ is given in the following table.

{\sisetup{group-minimum-digits=4}%
\begin{longtable}{lrrrrrrrrrrr}\toprule
    &  & \multicolumn{5}{c}{positive} & \multicolumn{2}{c}{principal} & indefinite
    \\\cmidrule(lr){3-7}\cmidrule(lr){8-9}\cmidrule(ll){10-12}
    $n$& $\# I$ & $\CA_n$ & $\CD_n^{[1]}$ & $\CD_{n,s}^{[2]}$ & $\CD_{n,s}^{[3]}$  & $\EE_n$ & $\AA_{n-1}$ & $\EE_{n-1}$ & $\# I$ \\\midrule
    $4$ & $10$ & $4$ & $4$ & $1$ &  &  & $1$ &  & \\
    $5$ & $34$ & $10$ & $12$ & $4$ & $3$ &  & $1$ &  & $4$\\
    $6$ & $129$ & $16$ & $24$ & $12$ & $7$ & $38$ & $5$ &  & $27$\\
    $7$ & $413$ & $36$ & $48$ & $32$ & $15$ & $145$ & $6$ &  & $131$\\
    $8$ & $\num{1369}$ & $64$ & $96$ & $80$ & $31$ & $315$ & $17$ & $185$ & $581$\\
    $9$ & $\num{4184}$ & $136$ & $192$ & $192$ & $63$ &  & $25$ & $665$ & $\num{2911}$\\
    $10$ & $\num{12980}$ & $256$ & $384$ & $448$ & $127$ &  & $56$ &  & $\num{11709}$\\
    $11$ & $\num{39079}$ & $528$ & $768$ & $\num{1024}$ & $255$ &  & $88$ &  & $\num{36416}$\\
    \bottomrule
\end{longtable}}
\noindent This computer-assisted analysis completes the proof for posets $I$ of size $|I|\leq 11$.\smallskip 

\textbf{Part $2^\circ$} We proceed by induction. Assume that $I$ is such a finite connected poset of size $|I|=n>11$ that for some $p\in \{1,\ldots,n\}$ the graph $\ov \CH(I^{(p)})$ is isomorphic to a path graph, and the thesis holds for posets of size $n-1$. To prove the inductive step we show that $\CH(I)$ has one of the forms described in \ref{lemma:pathext:posit:a}, \ref{lemma:pathext:posit:d} and \ref{lemma:pathext:princ:a}, or is indefinite \ref{lemma:pathext:indef}. Without loss of generality, we may assume that $p=1$ and $\deg_{\ov\CH(I^{(1)})}(n)=1$, i.e., 
\vspace*{-1ex}
 \begin{align*}
\ov\CH(I^{(1)})\simeq P(2,n) &= \!\!\!\tikzsetnextfilename{main_prf_5}\begin{tikzpicture}[baseline=(n2.base),label distance=-2pt]
        \matrix [matrix of nodes, ampersand replacement=\&, nodes={minimum height=1.3em,minimum width=1.3em,
                text depth=0ex,text height=1ex, execute at begin node=$, execute at end node=$}
        , column sep={15pt,between borders}, row sep={10pt,between borders}]
        {
                |(n2)|2  \& |(n3)|3  \& |(n4)| \& |(n5)| \& |(n6)|n-1  \& |(n7)|n \\
        };
        \foreach \x/\y in {2/3, 3/4, 5/6, 6/7}
        \draw [-, shorten <= -2.50pt, shorten >= -2.50pt] (n\x) to  (n\y);
        \draw [line width=1.2pt, line cap=round, dash pattern=on 0pt off 5\pgflinewidth, -, shorten <= -2.50pt, shorten >= -2.50pt] (n4) to  (n5);
\end{tikzpicture}\!\!\!\!
\simeq \CA_{n-1}.\\[-1ex]
\intertext{\vspace*{-1ex}Consider the poset $J\eqdef I^{(n)}$.}\\[-1.2cm]     
\intertext{\indent(A) If $J$ is not connected, then the element $n$ is an articulation point in the graph $\ov\CH(J)$. Since degree of $n$ in $\ov\CH(I^{(1)})$ equals one, we conclude that  $J$ has two connected components: $\{2,\ldots,n-1\}$ and $\{1\}$. Moreover,  the graph $\ov\CH(I)$ have the shape}
\ov\CH(I) &\simeq 
\!\!\!\tikzsetnextfilename{main_prf_6}\begin{tikzpicture}[baseline=(n2.base),label distance=-2pt]
        \matrix [matrix of nodes, ampersand replacement=\&, nodes={minimum height=1.3em,minimum width=1.3em,text depth=0ex,text height=1ex, execute at begin node=$, execute at end node=$},column sep={15pt,between borders}, row sep={10pt,between borders}]
        {
                |(n2)|2  \& |(n3)|3  \& |(n4)| \& |(n5)| \& |(n6)|n-1  \& |(n7)|n  \& |(n1)|1 \\
        };
        \foreach \x/\y in {2/3, 3/4, 5/6, 6/7, 7/1}
        \draw [-, shorten <= -2.50pt, shorten >= -2.50pt] (n\x) to  (n\y);
        \draw [line width=1.2pt, line cap=round, dash pattern=on 0pt off 5\pgflinewidth, -, shorten <= -2.50pt, shorten >= -2.50pt] (n4) to  (n5);
        \draw [thick,decoration={brace,raise=0.35cm},decorate] (n2.west)  -- (n7.east) node [pos=0.5,anchor=north,yshift=0.65cm] {$\smash{\scriptstyle \ov \CH(I^{(p)})}$};
        \draw[draw=black, dashed, fill=green, fill opacity=0.2, rounded corners]([xshift=1pt]n1.north west)--([xshift=-1pt]n1.north east)--([xshift=-1pt,yshift=2pt]n1.south east)--([xshift=1pt,yshift=2pt]n1.south west)--cycle;
        \draw[draw=black, dashed, fill=green, fill opacity=0.2, rounded corners]([xshift=1pt]n2.north west)--([xshift=-1pt]n6.north east)--([xshift=-1pt,yshift=2pt]n6.south east)--([xshift=1pt,yshift=2pt]n2.south west)--cycle;
\end{tikzpicture}\!\!\!\!
\simeq \CA_n,
\end{align*}
where the elements of the poset $J$ are highlighted. Hence the thesis follows.\smallskip

(B) Assume that $J$ is connected. Since $J^{(1)}=I^{(1,n)}$ is such a poset of size $n-1$, that the graph $\ov\CH(J^{(1)})$ is isomorphic to a path graph, by the inductive hypothesis one of the following conditions holds:
\begin{enumerate}[label=\normalfont{(\roman*)}, itemsep=0ex, topsep=1ex, partopsep=1ex, leftmargin=10ex]
 \item $\CH(J)\simeq \CD_{J}$, where 
$\CD_{J}\in\{\CA_n,\CD_n^{[1]},\CD_{n,s}^{[2]},\CD_{n,s}^{[3]}\}$, as described in \ref{lemma:pathext:posit:a} and \ref{lemma:pathext:posit:d};
\item $\ov \CH(J)$ is a cycle graph and $\CH(J)$ has at least two sinks, as described in \ref{lemma:pathext:princ:a};
\item $J$ is indefinite, as described in \ref{lemma:pathext:indef}.
\end{enumerate} 
We analyze these cases one by one. Since the digraphs $\CH(I)$ and $\CH(J)$  are connected, $\ov\CH(I^{(1)})\simeq \CA_{n-1}$ and $\deg_{\ov\CH(I^{(1)})}(n)=1$,  we conclude that the degree of the vertex $n$ in the graph $\ov\CH(I)$ equals two, if elements $1$ and $n$ are in relation, and one otherwise.\smallskip

(i) Assume that  $\CH(J)\simeq \CD_{J}\in\{\CA_n,\ab\CD_n^{[1]}, \ab\CD_{n,s}^{[2]},\ab\CD_{n,s}^{[3]}\}$. We have the following: \begin{enumerate}[label=\normalfont{($\arabic*^\circ$)},wide,labelindent=0pt]
    \item\label{lemma:pathext:prf:b:i:i} if $\CD_J=\CA_{n-1}$, then either $\CH(I)\simeq \CD_I\in\{\CA_n, \CD_n^{[1]}, \CD_{n,s'}^{[3]}\}$, where $2\leq s'\leq n-2$, or $\CH(I)$ has at least two sinks and $\ov\CH(I)$ is a cycle graph;
    
	\item\label{lemma:pathext:prf:b:i:ii} if $\CD_J=\CD_{n-1}^{[1]}$, then either  $\CH(I)\simeq \CD_I\in\{\CD_n^{[1]},\CD_{n,1}^{[2]},\CD_{n,n-3}^{[2]}, \CD_{n,2}^{[3]}\}$, or $I$ is indefinite, as it contains (as a subposet) some $\CF\in\{\CF_1,\ldots, \CF_6\}$;
                        
    \item\label{lemma:pathext:prf:b:i:iii} if $\CD_J=\CD_{n-1,s}^{[2]}$, then either  $\CH(I)\simeq \CD_{n,s'}^{[2]}$, where $1\leq s'\leq n-3$, or $I$ is indefinite, as it contains (as a subposet) some $\CF\in\{\CF_3, \ldots, \CF_6\}$;
    
    \item\label{lemma:pathext:prf:b:i:iiii} if $\CD_J=\CD_{n-1, s}^{[3]}$, then either  $\CH(I)\simeq \CD_{n,s'}^{[3]}$, where $2\leq s'< n-1$, or $I$ is indefinite, as it contains (as a subposet) some $\CF\in\{\CF_2, \ldots, \CF_7\}$.
\end{enumerate}

We describe the case \ref{lemma:pathext:prf:b:i:iii} in detail. One has to consider all such finite posets $I$, that $\ov\CH(I^{(1)})\simeq P(2,n)$, $\CH(I^{(n)})\simeq \CD_{n-1,s}^{[2]}$ and $\deg_{\ov\CH(I^{(1)})}(n)=1$. These are the only possibilities:\medskip

\noindent ($3^\circ a$) $\CH(I)\simeq \CD_{n-1,s}^{[2]}$ and $\CH(I)$\ has one of the following shapes:\medskip

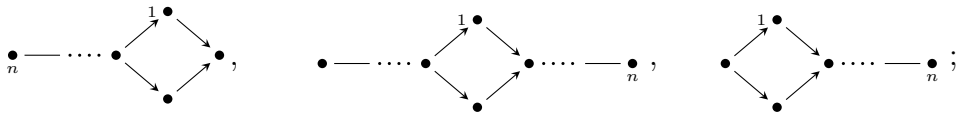
\tikzsetnextfilename{fig3a3}%
\noindent\begin{tikzpicture}[baseline={([yshift=-2.75pt]current bounding box)},label distance=-2pt,
        xscale=0.68, yscale=0.58]
        \node[circle, fill=black, inner sep=0pt, minimum size=3.5pt, label=left:$\scriptscriptstyle 1$] (n1) at (3, 2  ) {};
        \node[circle, fill=black, inner sep=0pt, minimum size=3.5pt, label=below:$\scriptscriptstyle n$] (n2) at (0, 1  ) {};
        \node (n3) at (1, 1  ) {$\scriptscriptstyle $};
        \node[circle, fill=black, inner sep=0pt, minimum size=3.5pt] (n4) at (2, 1  ) {};
        \node[circle, fill=black, inner sep=0pt, minimum size=3.5pt, label=below:$\scriptscriptstyle $] (n5) at (4, 1  ) {};
        \node[circle, fill=black, inner sep=0pt, minimum size=3.5pt, label=below:$\scriptscriptstyle $] (n6) at (3, 0  ) {};
        \draw [-, shorten <= 2.50pt, shorten >= -2.50pt] (n2) to  (n3);
        \draw [line width=1.2pt, line cap=round, dash pattern=on 0pt off 3\pgflinewidth, -, shorten <= -2.0pt, shorten >= 2.0pt] (n3) to  (n4);
        \foreach \x/\y in {1/5, 4/1, 4/6, 6/5}
        \draw [-stealth, shorten <= 2.50pt, shorten >= 2.50pt] (n\x) to  (n\y);
\end{tikzpicture},
\qquad
\tikzsetnextfilename{fig3a1}%
\begin{tikzpicture}[baseline={([yshift=-2.75pt]current bounding box)},label distance=-2pt,
        xscale=0.68, yscale=0.58]
        \node[circle, fill=black, inner sep=0pt, minimum size=3.5pt, label=left:$\scriptscriptstyle 1$] (n1) at (3, 2  ) {};
\node[circle, fill=black, inner sep=0pt, minimum size=3.5pt, label=below:$\scriptscriptstyle $] (n2) at (0, 1  ) {};
        \node (n3) at (1, 1  ) {$\scriptscriptstyle $};
        \node[circle, fill=black, inner sep=0pt, minimum size=3.5pt] (n4) at (2, 1  ) {};
        \node[circle, fill=black, inner sep=0pt, minimum size=3.5pt] (n5) at (4, 1  ) {};
        \node[circle, fill=black, inner sep=0pt, minimum size=3.5pt] (n6) at (3, 0  ) {};
        \node (n7) at (5, 1  ) {$\scriptscriptstyle $};
        \node[circle, fill=black, inner sep=0pt, minimum size=3.5pt, label=below:$\scriptscriptstyle n$] (n8) at (6, 1  ) {};
\draw [-, shorten <= 2.50pt, shorten >= -2.50pt] (n2) to  (n3);
        \draw [line width=1.2pt, line cap=round, dash pattern=on 0pt off 3\pgflinewidth, -, shorten <= -2.0pt, shorten >= 2.0pt] (n3) to  (n4);
        \foreach \x/\y in {1/5, 4/1, 4/6, 6/5}
        \draw [-stealth, shorten <= 2.50pt, shorten >= 2.50pt] (n\x) to  (n\y);
        \draw [line width=1.2pt, line cap=round, dash pattern=on 0pt off 3\pgflinewidth, -, shorten <= 3.5pt, shorten >= -1.5pt] (n5) to  (n7);
        \draw [-, shorten <= -2.50pt, shorten >= 2.50pt] (n7) to  (n8);
\end{tikzpicture},%
\qquad
\tikzsetnextfilename{fig3a2}%
\begin{tikzpicture}[baseline={([yshift=-2.75pt]current bounding box)},label distance=-2pt,
        xscale=0.68, yscale=0.58]
        \node[circle, fill=black, inner sep=0pt, minimum size=3.5pt, label=left:$\scriptscriptstyle 1$] (n1) at (3, 2  ) {};
        \node[circle, fill=black, inner sep=0pt, minimum size=3.5pt] (n4) at (2, 1  ) {};
        \node[circle, fill=black, inner sep=0pt, minimum size=3.5pt] (n5) at (4, 1  ) {};
        \node[circle, fill=black, inner sep=0pt, minimum size=3.5pt] (n6) at (3, 0  ) {};
        \node (n7) at (5, 1  ) {$\scriptscriptstyle $};
        \node[circle, fill=black, inner sep=0pt, minimum size=3.5pt, label=below:$\scriptscriptstyle n$] (n8) at (6, 1  ) {};
        \foreach \x/\y in {1/5, 4/1, 4/6, 6/5}
        \draw [-stealth, shorten <= 2.50pt, shorten >= 2.50pt] (n\x) to  (n\y);
        \draw [line width=1.2pt, line cap=round, dash pattern=on 0pt off 3\pgflinewidth, -, shorten <= 3.5pt, shorten >= -1.5pt] (n5) to  (n7);
        \draw [-, shorten <= -2.50pt, shorten >= 2.50pt] (n7) to  (n8);
\end{tikzpicture};\bigskip

\noindent ($3^\circ b$) $I$ is indefinite, since  $\CF_3\subseteq I$:\medskip

\noindent \tikzsetnextfilename{fig3b}%
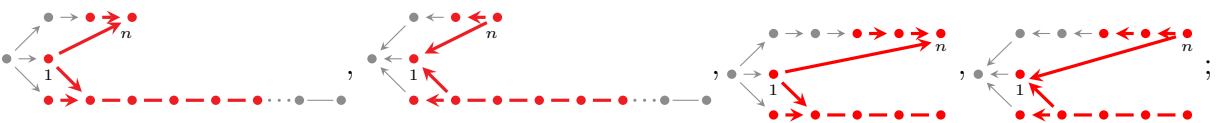
\begin{tikzpicture}[baseline={([yshift=-2.75pt]current bounding box)},label distance=-2pt,
        xscale=0.55, yscale=0.55]
        \node[circle, fill=Red, inner sep=0pt, minimum size=3.5pt, label=below:$\scriptscriptstyle n\phantom{\mppmss }$] (n1) at (11, 2  ) {};
        \node[circle, fill=Red, inner sep=0pt, minimum size=3.5pt] (n2) at (10, 2  ) {};
        \node[circle, fill=, gray!90, inner sep=0pt, minimum size=3.5pt, label=below:$\scriptscriptstyle $] (n3) at (9, 2  ) {};
        \node[circle, fill=, gray!90, inner sep=0pt, minimum size=3.5pt, label=below:$\scriptscriptstyle $] (n4) at (8, 1  ) {};
        \node[circle, fill=, Red, inner sep=0pt, minimum size=3.5pt, label=below:$\scriptscriptstyle 1$] (n5) at (9, 1  ) {};
        \node[circle, fill=Red, inner sep=0pt, minimum size=3.5pt, label=below:$\scriptscriptstyle $] (n6) at (9, 0  ) {};
        \node[circle, fill=Red, inner sep=0pt, minimum size=3.5pt, label=below:$\scriptscriptstyle $] (n7) at (10, 0  ) {};
        \node[circle, fill=Red, inner sep=0pt, minimum size=3.5pt, label=below:$\scriptscriptstyle $] (n8) at (11, 0  ) {};
        \node[circle, fill=Red, inner sep=0pt, minimum size=3.5pt, label=below:$\scriptscriptstyle $] (n9) at (12, 0  ) {};
        \node[circle, fill=Red, inner sep=0pt, minimum size=3.5pt, label=below:$\scriptscriptstyle $] (n10) at (13, 0  ) {};
        \node[circle, fill=Red, inner sep=0pt, minimum size=3.5pt, label=below:$\scriptscriptstyle $] (n11) at (14, 0  ) {};
        \node[circle, fill=, gray!90, inner sep=0pt, minimum size=3.5pt, label=below:$\scriptscriptstyle $] (n12) at (15, 0  ) {};
        \node[circle, fill=, gray!90, inner sep=0pt, minimum size=3.5pt] (n13) at (16, 0  ) {};
        \foreach \x/\y in {2/3, 3/4, 5/4, 6/4}
        \draw [-, gray!90, stealth-, shorten <= 2.50pt, shorten >= 2.50pt] (n\x) to  (n\y);
        \draw [-, , gray!90, -, shorten <= 1pt, shorten >= 1.0pt] (n13) to  (n12);
        \foreach \x/\y in {1/2, 1/5, 7/5, 7/6}
        \draw [stealth-, line width=1.2pt, Red,shorten <= 2.50pt, shorten >= 2.50pt] (n\x) to  (n\y);
        \foreach \x/\y in {8/7, 9/8, 10/9, 11/10}
        \draw [-, line width=1.2pt, Red, shorten <= 2.50pt, shorten >= 2.50pt] (n\x) to  (n\y);
        \draw [line width=1.2pt, gray!90, line cap=round, dash pattern=on 0pt off 3\pgflinewidth, -, shorten <= 2.0pt, shorten >= 1.0pt] (n12) to  (n11);
\end{tikzpicture},\ 
\tikzsetnextfilename{fig3b2}%
 \begin{tikzpicture}[baseline={([yshift=-2.75pt]current bounding box)},label distance=-2pt,
        xscale=0.55, yscale=0.55]
        \node[circle, fill=Red, inner sep=0pt, minimum size=3.5pt, label=below:$\scriptscriptstyle n\phantom{\mppmss }$] (n1) at (11, 2  ) {};
        \node[circle, fill=Red, inner sep=0pt, minimum size=3.5pt] (n2) at (10, 2  ) {};
        \node[circle, fill=, gray!90, inner sep=0pt, minimum size=3.5pt, label=below:$\scriptscriptstyle $] (n3) at (9, 2  ) {};
        \node[circle, fill=, gray!90, inner sep=0pt, minimum size=3.5pt, label=below:$\scriptscriptstyle $] (n4) at (8, 1  ) {};
        \node[circle, fill=, Red, inner sep=0pt, minimum size=3.5pt, label=below:$\scriptscriptstyle 1$] (n5) at (9, 1  ) {};
        \node[circle, fill=Red, inner sep=0pt, minimum size=3.5pt, label=below:$\scriptscriptstyle $] (n6) at (9, 0  ) {};
        \node[circle, fill=Red, inner sep=0pt, minimum size=3.5pt, label=below:$\scriptscriptstyle $] (n7) at (10, 0  ) {};
        \node[circle, fill=Red, inner sep=0pt, minimum size=3.5pt, label=below:$\scriptscriptstyle $] (n8) at (11, 0  ) {};
        \node[circle, fill=Red, inner sep=0pt, minimum size=3.5pt, label=below:$\scriptscriptstyle $] (n9) at (12, 0  ) {};
        \node[circle, fill=Red, inner sep=0pt, minimum size=3.5pt, label=below:$\scriptscriptstyle $] (n10) at (13, 0  ) {};
        \node[circle, fill=Red, inner sep=0pt, minimum size=3.5pt, label=below:$\scriptscriptstyle $] (n11) at (14, 0  ) {};
        \node[circle, fill=, gray!90, inner sep=0pt, minimum size=3.5pt, label=below:$\scriptscriptstyle $] (n12) at (15, 0  ) {};
        \node[circle, fill=, gray!90, inner sep=0pt, minimum size=3.5pt] (n13) at (16, 0  ) {};
        \foreach \x/\y in {3/2, 4/3, 4/5, 4/6}
        \draw [-, gray!90, stealth-, shorten <= 2.50pt, shorten >= 2.50pt] (n\x) to  (n\y);
        \draw [-, , gray!90, -, shorten <= 1pt, shorten >= 1.0pt] (n13) to  (n12);
        \foreach \x/\y in {2/1, 5/1, 5/7, 6/7}
        \draw [stealth-, line width=1.2pt, Red,shorten <= 2.50pt, shorten >= 2.50pt] (n\x) to  (n\y);
        \foreach \x/\y in {8/7, 9/8, 10/9, 11/10}
        \draw [-, line width=1.2pt, Red, shorten <= 2.50pt, shorten >= 2.50pt] (n\x) to  (n\y);
        \draw [line width=1.2pt, gray!90, line cap=round, dash pattern=on 0pt off 3\pgflinewidth, -, shorten <= 2.0pt, shorten >= 1.0pt] (n12) to  (n11);
\end{tikzpicture},\ 
\tikzsetnextfilename{fig3b3}\begin{tikzpicture}[baseline=(n1.base),label distance=-2pt,xscale=0.55, yscale=0.54]
\node[circle, fill=gray!90, inner sep=0pt, minimum size=3.5pt] (n1) at (0  , 1  ) {};
\node[circle, fill=red, inner sep=0pt, minimum size=3.5pt, label=below:$\scriptscriptstyle 1$] (n2) at (1  , 1  ) {};
\node[circle, fill=red, inner sep=0pt, minimum size=3.5pt] (n3) at (1  , 0  ) {};
\node[circle, fill=red, inner sep=0pt, minimum size=3.5pt] (n4) at (2  , 0  ) {};
\node[circle, fill=red, inner sep=0pt, minimum size=3.5pt] (n5) at (3  , 0  ) {};
\node[circle, fill=red, inner sep=0pt, minimum size=3.5pt] (n6) at (4  , 0  ) {};
\node[circle, fill=red, inner sep=0pt, minimum size=3.5pt] (n7) at (5  , 0  ) {};
\node[circle, fill=red, inner sep=0pt, minimum size=3.5pt, label=below:$\scriptscriptstyle n$] (n8) at (5  , 2  ) {};
\node[circle, fill=red, inner sep=0pt, minimum size=3.5pt] (n9) at (4  , 2  ) {};
\node[circle, fill=red, inner sep=0pt, minimum size=3.5pt] (n10) at (3  , 2  ) {};
\node[circle, fill=gray!90, inner sep=0pt, minimum size=3.5pt] (n11) at (1  , 2  ) {};
\node[circle, fill=gray!90, inner sep=0pt, minimum size=3.5pt] (n12) at (2  , 2  ) {};
\foreach \x/\y in {1/2, 1/3, 1/11, 11/12, 12/10}
        \draw [gray!90, solid, -stealth, shorten <= 2.50pt, shorten >= 2.50pt] (n\x) to  (n\y);
\foreach \x/\y in {2/4, 3/4, 9/8, 10/9}
        \draw [very thick, red, -stealth, shorten <= 2.50pt, shorten >= 2.50pt] (n\x) to  (n\y);
\draw [very thick, red, -stealth, shorten <= 2.50pt, shorten >= 3.50pt] (n2) to  ([yshift=-5]n8);
\foreach \x/\y in {4/5, 5/6, 6/7}
        \draw [very thick, red, -, shorten <= 2.50pt, shorten >= 2.50pt] (n\x) to  (n\y);
\end{tikzpicture},\ 
\tikzsetnextfilename{fig3b4}\begin{tikzpicture}[baseline=(n1.base),label distance=-2pt,xscale=0.55, yscale=0.54]
\node[circle, fill=gray!90, inner sep=0pt, minimum size=3.5pt] (n1) at (0  , 1  ) {};
\node[circle, fill=red, inner sep=0pt, minimum size=3.5pt, label=below:$\scriptscriptstyle 1$] (n2) at (1  , 1  ) {};
\node[circle, fill=red, inner sep=0pt, minimum size=3.5pt] (n3) at (1  , 0  ) {};
\node[circle, fill=red, inner sep=0pt, minimum size=3.5pt] (n4) at (2  , 0  ) {};
\node[circle, fill=red, inner sep=0pt, minimum size=3.5pt] (n5) at (3  , 0  ) {};
\node[circle, fill=red, inner sep=0pt, minimum size=3.5pt] (n6) at (4  , 0  ) {};
\node[circle, fill=red, inner sep=0pt, minimum size=3.5pt] (n7) at (5  , 0  ) {};
\node[circle, fill=red, inner sep=0pt, minimum size=3.5pt, label=below:$\scriptscriptstyle n$] (n8) at (5  , 2  ) {};
\node[circle, fill=red, inner sep=0pt, minimum size=3.5pt] (n9) at (4  , 2  ) {};
\node[circle, fill=red, inner sep=0pt, minimum size=3.5pt] (n10) at (3  , 2  ) {};
\node[circle, fill=gray!90, inner sep=0pt, minimum size=3.5pt] (n11) at (1  , 2  ) {};
\node[circle, fill=gray!90, inner sep=0pt, minimum size=3.5pt] (n12) at (2  , 2  ) {};
\foreach \x/\y in {2/1, 3/1, 10/12, 11/1, 12/11}
        \draw [gray!90, solid, -stealth, shorten <= 2.50pt, shorten >= 2.50pt] (n\x) to  (n\y);
\foreach \x/\y in {4/2, 4/3, 8/9, 9/10}
        \draw [very thick, red, -stealth, shorten <= 2.50pt, shorten >= 2.50pt] (n\x) to  (n\y);
\draw [very thick, red, -stealth, shorten <= 2.50pt, shorten >= 3.50pt] (n8) to  ([yshift=-5]n2);
\foreach \x/\y in {4/5, 5/6, 6/7}
        \draw [very thick, red, -, shorten <= 2.50pt, shorten >= 2.50pt] (n\x) to  (n\y);
\end{tikzpicture};\bigskip

\noindent ($3^\circ c$) $I$ is indefinite, since  $\CF_4\subseteq I$:\medskip

\noindent \tikzsetnextfilename{fig3d11}%
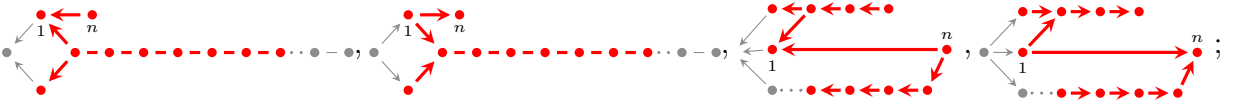
\begin{tikzpicture}[baseline=(n1.base),label distance=-2pt,xscale=0.45, yscale=0.5]
\node[circle, fill=gray!90, inner sep=0pt, minimum size=3.5pt] (n1) at (0  , 1  ) {};
\node[circle, fill=red, inner sep=0pt, minimum size=3.5pt, label=below:$\scriptscriptstyle 1$] (n2) at (1  , 2  ) {};
\node[circle, fill=red, inner sep=0pt, minimum size=3.5pt] (n3) at (1  , 0  ) {};
\node[circle, fill=red, inner sep=0pt, minimum size=3.5pt] (n4) at (2  , 1  ) {};
\node[circle, fill=red, inner sep=0pt, minimum size=3.5pt] (n5) at (3  , 1  ) {};
\node[circle, fill=red, inner sep=0pt, minimum size=3.5pt] (n6) at (4  , 1  ) {};
\node[circle, fill=red, inner sep=0pt, minimum size=3.5pt] (n7) at (5  , 1  ) {};
\node[circle, fill=red, inner sep=0pt, minimum size=3.5pt, label=below:$\scriptscriptstyle n$] (n8) at (2.5  , 2  ) {};
\node[circle, fill=red, inner sep=0pt, minimum size=3.5pt] (n9) at (6  , 1  ) {};
\node[circle, fill=red, inner sep=0pt, minimum size=3.5pt] (n10) at (7  , 1  ) {};
\node[circle, fill=red, inner sep=0pt, minimum size=3.5pt] (n11) at (8  , 1  ) {};
\node[circle, fill=gray!90, inner sep=0pt, minimum size=3.5pt] (n12) at (9  , 1  ) {};
\node[circle, fill=gray!90, inner sep=0pt, minimum size=3.5pt] (n13) at (10  , 1  ) {};
\foreach \x/\y in {2/1, 3/1}
        \draw [gray!90, solid, -stealth, shorten <= 2.50pt, shorten >= 2.50pt] (n\x) to  (n\y);
\draw [line width=1.2pt, line cap=round, dash pattern=on 0pt off 3\pgflinewidth, gray!90, -, shorten <= 2.20pt, shorten >= 1.50pt] (n11) to  (n12);
\draw [gray!90, solid, -, shorten <= 2.50pt, shorten >= 2.50pt] (n12) to  (n13);
\foreach \x/\y in {4/2, 4/3}
        \draw [very thick, red, -stealth, shorten <= 2.50pt, shorten >= 2.50pt] (n\x) to  (n\y);
\foreach \x/\y in {4/5, 5/6, 6/7, 7/9, 9/10, 10/11}
        \draw [very thick, red, -, shorten <= 2.50pt, shorten >= 2.50pt] (n\x) to  (n\y);
\draw [very thick, red, -stealth, shorten <= 2.50pt, shorten >= 1.50pt] (n8) to  (n2);

\end{tikzpicture},\ 
\tikzsetnextfilename{fig3d22}%
\begin{tikzpicture}[baseline=(n1.base),label distance=-2pt,xscale=0.45, yscale=0.5]
\node[circle, fill=gray!90, inner sep=0pt, minimum size=3.5pt] (n1) at (0  , 1  ) {};
\node[circle, fill=red, inner sep=0pt, minimum size=3.5pt, label=below:$\scriptscriptstyle 1$] (n2) at (1  , 2  ) {};
\node[circle, fill=red, inner sep=0pt, minimum size=3.5pt] (n3) at (1  , 0  ) {};
\node[circle, fill=red, inner sep=0pt, minimum size=3.5pt] (n4) at (2  , 1  ) {};
\node[circle, fill=red, inner sep=0pt, minimum size=3.5pt] (n5) at (3  , 1  ) {};
\node[circle, fill=red, inner sep=0pt, minimum size=3.5pt] (n6) at (4  , 1  ) {};
\node[circle, fill=red, inner sep=0pt, minimum size=3.5pt] (n7) at (5  , 1  ) {};
\node[circle, fill=red, inner sep=0pt, minimum size=3.5pt, label=below:$\scriptscriptstyle n$] (n8) at (2.5  , 2  ) {};
\node[circle, fill=red, inner sep=0pt, minimum size=3.5pt] (n9) at (6  , 1  ) {};
\node[circle, fill=red, inner sep=0pt, minimum size=3.5pt] (n10) at (7  , 1  ) {};
\node[circle, fill=red, inner sep=0pt, minimum size=3.5pt] (n11) at (8  , 1  ) {};
\node[circle, fill=gray!90, inner sep=0pt, minimum size=3.5pt] (n12) at (9  , 1  ) {};
\node[circle, fill=gray!90, inner sep=0pt, minimum size=3.5pt] (n13) at (10  , 1  ) {};
\foreach \x/\y in {1/2, 1/3}
        \draw [gray!90, solid, -stealth, shorten <= 2.50pt, shorten >= 2.50pt] (n\x) to  (n\y);
\draw [line width=1.2pt, line cap=round, dash pattern=on 0pt off 3\pgflinewidth, gray!90, -, shorten <= 2.20pt, shorten >= 1.50pt] (n11) to  (n12);
\draw [gray!90, solid, -, shorten <= 2.50pt, shorten >= 2.50pt] (n12) to  (n13);
\foreach \x/\y in {2/4, 3/4}
        \draw [very thick, red, -stealth, shorten <= 2.50pt, shorten >= 2.50pt] (n\x) to  (n\y);
\foreach \x/\y in {4/5, 5/6, 6/7, 7/9, 9/10, 10/11}
        \draw [very thick, red, -, shorten <= 2.50pt, shorten >= 2.50pt] (n\x) to  (n\y);
\draw [very thick, red, -stealth, shorten <= 2.50pt, shorten >= 1.50pt] (n2) to  (n8);
\end{tikzpicture},\ 
\tikzsetnextfilename{fig3d33}%
\begin{tikzpicture}[baseline=(n1.base),label distance=-2pt,xscale=0.51, yscale=0.54]
\\node[circle, fill=gray!90, inner sep=0pt, minimum size=3.5pt] (n1) at (0  , 1  ) {};
\node[circle, fill=red, inner sep=0pt, minimum size=3.5pt] (n2) at (1  , 2  ) {};
\node[circle, fill=red, inner sep=0pt, minimum size=3.5pt] (n3) at (2  , 2  ) {};
\node[circle, fill=red, inner sep=0pt, minimum size=3.5pt] (n4) at (3  , 2  ) {};
\node[circle, fill=red, inner sep=0pt, minimum size=3.5pt] (n5) at (4  , 2  ) {};
\node[circle, fill=red, inner sep=0pt, minimum size=3.5pt, label=below:$\scriptscriptstyle 1$] (n6) at (1  , 1  ) {};
\node[circle, fill=red, inner sep=0pt, minimum size=3.5pt] (n7) at (2  , 0  ) {};
\node[circle, fill=red, inner sep=0pt, minimum size=3.5pt] (n8) at (3  , 0  ) {};
\node[circle, fill=red, inner sep=0pt, minimum size=3.5pt] (n9) at (4  , 0  ) {};
\node[circle, fill=red, inner sep=0pt, minimum size=3.5pt] (n10) at (5  , 0  ) {};
\node[circle, fill=red, inner sep=0pt, minimum size=3.5pt, label=above:$\scriptscriptstyle n$] (n11) at (5.50, 1  ) {};
\node[circle, fill=gray!90, inner sep=0pt, minimum size=3.5pt] (n12) at (1  , 0  ) {};
\foreach \x/\y in {2/1, 6/1, 12/1}
        \draw [gray!90, solid, -stealth, shorten <= 1.50pt, shorten >= 1.50pt] (n\x) to  (n\y);
\draw [line width=1.2pt, line cap=round, dash pattern=on 0pt off 3\pgflinewidth, gray!90, -, shorten <= 1.50pt, shorten >= 1.50pt] (n12) to  (n7);
\foreach \x/\y in {3/2, 4/3, 5/4, 3/6, 11/6, 8/7, 9/8, 10/9, 11/10}
        \draw [very thick, red, -stealth, shorten <= 1.50pt, shorten >= 1.50pt] (n\x) to  (n\y);
\end{tikzpicture},\ 
\tikzsetnextfilename{fig3d44}%
\begin{tikzpicture}[baseline=(n1.base),label distance=-2pt,xscale=0.51, yscale=0.54]
\node[circle, fill=gray!90, inner sep=0pt, minimum size=3.5pt] (n1) at (0  , 1  ) {};
\node[circle, fill=red, inner sep=0pt, minimum size=3.5pt] (n2) at (1  , 2  ) {};
\node[circle, fill=red, inner sep=0pt, minimum size=3.5pt] (n3) at (2  , 2  ) {};
\node[circle, fill=red, inner sep=0pt, minimum size=3.5pt] (n4) at (3  , 2  ) {};
\node[circle, fill=red, inner sep=0pt, minimum size=3.5pt] (n5) at (4  , 2  ) {};
\node[circle, fill=red, inner sep=0pt, minimum size=3.5pt, label=below:$\scriptscriptstyle 1$] (n6) at (1  , 1  ) {};
\node[circle, fill=red, inner sep=0pt, minimum size=3.5pt] (n7) at (2  , 0  ) {};
\node[circle, fill=red, inner sep=0pt, minimum size=3.5pt] (n8) at (3  , 0  ) {};
\node[circle, fill=red, inner sep=0pt, minimum size=3.5pt] (n9) at (4  , 0  ) {};
\node[circle, fill=red, inner sep=0pt, minimum size=3.5pt] (n10) at (5  , 0  ) {};
\node[circle, fill=red, inner sep=0pt, minimum size=3.5pt, label=above:$\scriptscriptstyle n$] (n11) at (5.50, 1  ) {};
\node[circle, fill=gray!90, inner sep=0pt, minimum size=3.5pt] (n12) at (1  , 0  ) {};
\foreach \x/\y in {1/2, 1/6, 1/12}
        \draw [gray!90, solid, -stealth, shorten <= 1.50pt, shorten >= 1.50pt] (n\x) to  (n\y);
\draw [line width=1.2pt, line cap=round, dash pattern=on 0pt off 3\pgflinewidth, gray!90, -, shorten <= 1.50pt, shorten >= 1.50pt] (n12) to  (n7);
\foreach \x/\y in {2/3, 3/4, 4/5, 6/3, 6/11, 7/8, 8/9, 9/10, 10/11}
        \draw [very thick, red, -stealth, shorten <= 1.50pt, shorten >= 1.50pt] (n\x) to  (n\y);
\end{tikzpicture};\bigskip

\noindent ($3^\circ d$) $I$ is indefinite as $\CF_5\subseteq I$ or $\CF_6\subseteq I$, respectively:
\medskip

\noindent \tikzsetnextfilename{fig4d51}%
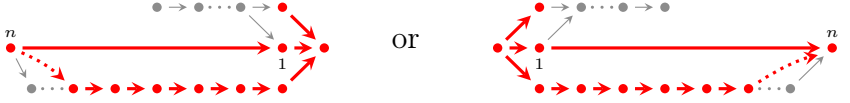
\begin{tikzpicture}[baseline=(n1.base),label distance=-2pt,xscale=0.55, yscale=0.54]
\node[circle, fill=red, inner sep=0pt, minimum size=3.5pt] (n1) at (7.50, 1  ) {};
\node[circle, fill=red, inner sep=0pt, minimum size=3.5pt] (n2) at (6.50, 2  ) {};
\node[circle, fill=red, inner sep=0pt, minimum size=3.5pt, label=below:$\scriptscriptstyle 1$] (n3) at (6.50, 1  ) {};
\node[circle, fill=red, inner sep=0pt, minimum size=3.5pt] (n4) at (6.50, 0  ) {};
\node[circle, fill=red, inner sep=0pt, minimum size=3.5pt] (n5) at (5.50, 0  ) {};
\node[circle, fill=red, inner sep=0pt, minimum size=3.5pt] (n6) at (4.50, 0  ) {};
\node[circle, fill=red, inner sep=0pt, minimum size=3.5pt] (n7) at (3.50, 0  ) {};
\node[circle, fill=red, inner sep=0pt, minimum size=3.5pt] (n8) at (2.50, 0  ) {};
\node[circle, fill=red, inner sep=0pt, minimum size=3.5pt] (n9) at (1.50, 0  ) {};
\node[circle, fill=gray!90, inner sep=0pt, minimum size=3.5pt] (n10) at (5.50, 2  ) {};
\node[circle, fill=gray!90, inner sep=0pt, minimum size=3.5pt] (n11) at (4.50, 2  ) {};
\node[circle, fill=gray!90, inner sep=0pt, minimum size=3.5pt] (n12) at (0.50, 0  ) {};
\node[circle, fill=red, inner sep=0pt, minimum size=3.5pt, label=above:$\scriptscriptstyle n$] (n13) at (0  , 1  ) {};
\node[circle, fill=gray!90, inner sep=0pt, minimum size=3.5pt] (n14) at (3.50, 2  ) {};
\foreach \x/\y in {11/10, 12/9}
        \draw [line width=1.2pt, line cap=round, dash pattern=on 0pt off 3\pgflinewidth, gray!90, -, shorten <= 2.20pt, shorten >= 1.50pt] (n\x) to  (n\y);
\foreach \x/\y in {10/2, 10/3, 13/12, 14/11}
        \draw [gray!90, solid, -stealth, shorten <= 2.50pt, shorten >= 2.50pt] (n\x) to  (n\y);
\draw [bend left=10.0, very thick, red, dotted, -stealth, shorten <= 2.50pt, shorten >= 2.50pt] (n13) to  (n9);
\foreach \x/\y in {2/1, 3/1, 4/1, 5/4, 6/5, 7/6, 8/7, 9/8, 13/3}
        \draw [very thick, red, -stealth, shorten <= 2.50pt, shorten >= 2.50pt] (n\x) to  (n\y);
\end{tikzpicture}\qquad or \qquad
\tikzsetnextfilename{fig4d52}%
\begin{tikzpicture}[baseline=(n1.base),label distance=-2pt,xscale=0.55, yscale=0.54]
\node[circle, fill=red, inner sep=0pt, minimum size=3.5pt] (n1) at (0  , 1  ) {};
\node[circle, fill=red, inner sep=0pt, minimum size=3.5pt] (n2) at (1  , 2  ) {};
\node[circle, fill=red, inner sep=0pt, minimum size=3.5pt, label=below:$\scriptscriptstyle 1$] (n3) at (1  , 1  ) {};
\node[circle, fill=red, inner sep=0pt, minimum size=3.5pt] (n4) at (1  , 0  ) {};
\node[circle, fill=red, inner sep=0pt, minimum size=3.5pt] (n5) at (2  , 0  ) {};
\node[circle, fill=red, inner sep=0pt, minimum size=3.5pt] (n6) at (3  , 0  ) {};
\node[circle, fill=red, inner sep=0pt, minimum size=3.5pt] (n7) at (4  , 0  ) {};
\node[circle, fill=red, inner sep=0pt, minimum size=3.5pt] (n8) at (5  , 0  ) {};
\node[circle, fill=red, inner sep=0pt, minimum size=3.5pt] (n9) at (6  , 0  ) {};
\node[circle, fill=gray!90, inner sep=0pt, minimum size=3.5pt] (n10) at (2  , 2  ) {};
\node[circle, fill=gray!90, inner sep=0pt, minimum size=3.5pt] (n11) at (3  , 2  ) {};
\node[circle, fill=gray!90, inner sep=0pt, minimum size=3.5pt] (n12) at (7  , 0  ) {};
\node[circle, fill=red, inner sep=0pt, minimum size=3.5pt, label=above:$\scriptscriptstyle n$] (n13) at (8  , 1  ) {};
\node[circle, fill=gray!90, inner sep=0pt, minimum size=3.5pt] (n14) at (4  , 2  ) {};
\foreach \x/\y in {9/12, 10/11}
        \draw [line width=1.2pt, line cap=round, dash pattern=on 0pt off 3\pgflinewidth, gray!90, -, shorten <= 2.20pt, shorten >= 1.50pt] (n\x) to  (n\y);
\foreach \x/\y in {2/10, 3/10, 11/14, 12/13}
        \draw [gray!90, solid, -stealth, shorten <= 2.50pt, shorten >= 2.50pt] (n\x) to  (n\y);
\draw [bend left=10.0, very thick, red, dotted, -stealth, shorten <= 2.50pt, shorten >= 5.50pt] (n9) to  ([yshift=-2.6]n13);
\foreach \x/\y in {1/2, 1/3, 1/4, 3/13, 4/5, 5/6, 6/7, 7/8, 8/9}
        \draw [very thick, red, -stealth, shorten <= 2.50pt, shorten >= 2.50pt] (n\x) to  (n\y);
\end{tikzpicture}.\bigskip

The cases \ref{lemma:pathext:prf:b:i:i}, \ref{lemma:pathext:prf:b:i:ii} and \ref{lemma:pathext:prf:b:i:iiii} follow by similar arguments. Details are left to the reader.\medskip

(ii) Assume that $J$  is principal, i.e., $\ov \CH(J)$ is a cycle graph and $\CH(J)$ has at least two sinks. There are two possibilities: degree of vertex $n$ in the graph $\ov\CH(I)$ equals either one or two. In the first case the poset $I$ is indefinite, as it contains a subposet $\CF_1$, $\CF_2$, $\CF_3$ or $\CF_4$. In the second case, $\ov \CH(I)$ is a cycle graph and $\CH(I)$ contains the identical number of sinks as $\CH(J)$. Therefore, by \Cref{thm:digraphdegmax2}\ref{thm:digraphdegmax2:cycle:multsink}, $I$ is principal and  statement \ref{lemma:pathext:princ} follows.\medskip

(iii) Since poset $J\subset I$ is indefinite, the poset $I$ is also indefinite.
\end{proof}

\begin{center}
\textbf{Proof of \Cref{thm:a:main}}
\end{center}

Now we have all the necessary tools to prove the main result of this work.
\begin{proof}[Proof of \Cref{thm:a:main}]
Let $I=(V,\leq_I)$ be a finite connected non-negative poset of size $n$, rank $m$ and Dynkin type $\Dyn_I=\AA_m$.\smallskip

\ref{thm:a:main:posit} Our aim is to show that $m=n$  if and only if
\[
\ov\CH(I)\simeq P(1,n)=1 \,\rule[2.5pt]{22pt}{0.4pt}\,2\,\rule[2.5pt]{22pt}{0.4pt}\,
\hdashrule[2.5pt]{12pt}{0.4pt}{1pt}\,\rule[2.5pt]{22pt}{.4pt}\,n.
\]

Since the implication ``$\Leftarrow$'' is a consequence of \Cref{lemma:trees}\ref{lemma:trees:posit}, it is sufficient to prove ``$\Rightarrow$''. First, we show that for every vertex $v\in V$ we have $\deg_{\ov \CH(I)}(v)\leq 2$. We proceed by contradiction. Assume that there exists a vertex $v$ of degree at least $3$. If that is the case,  there exists such a subposet $J\subseteq I$ that its Hasse digraph $\CH(J)$ has the form 
\begin{center}
$\CH(J)\colon$
\tikzsetnextfilename{scnd_prf_1}\begin{tikzpicture}[baseline={([yshift=-2.75pt]current bounding box)},label distance=-2pt,xscale=0.65, yscale=0.74]
        \node[circle, fill=black, inner sep=0pt, minimum size=3.5pt, label=below:$\scriptscriptstyle v$] (n1) at (1, 0.50) {};
        \node[circle, fill=black, inner sep=0pt, minimum size=3.5pt, label=left:$\scriptscriptstyle $] (n2) at (2, 0.50) {};
        \node[circle, fill=black, inner sep=0pt, minimum size=3.5pt, label=left:$\scriptscriptstyle $] (n3) at (0, 1  ) {};
        \node[circle, fill=black, inner sep=0pt, minimum size=3.5pt, label=left:$\scriptscriptstyle $] (n4) at (0, 0  ) {};
        \foreach \x/\y in {1/2, 3/1, 4/1}
        \draw [-, shorten <= 2.50pt, shorten >= 2.50pt] (n\x) to  (n\y);
\end{tikzpicture}.
\end{center}
By \Cref{lemma:trees}\ref{lemma:trees:posit},  $\Dyn_{J}=\DD_4$. Since $J\subseteq I$, ~\cite[Proposition 2.25]{barotQuadraticFormsCombinatorics2019} yields $\Dyn_I\in\{\DD_n, \EE_n\}$, contrary to our assumptions. We conclude that $\deg_{\ov \CH(I)}(v)\leq 2$ for every vertex $v\in V$ and, in view of \Cref{thm:digraphdegmax2},  statement \ref{thm:a:main:posit} follows.\medskip

\ref{thm:a:main:princ}
We need to show that $m=n-1$ if and only if $\ov \CH(I)$ is a cycle graph and $\CH(I)$ has at least two sinks. To prove ``$\Rightarrow$'' assume that $I$ is a connected principal poset of the Dynkin type $\Dyn_I=\AA_{m}$. By definition, there exists such a $k\in I$, that the poset $J\eqdef I^{(k)}$ is positive of the Dynkin type $\Dyn_J=\AA_{m}$. By~\ref{thm:a:main:posit} we know that $\ov\CH(J)\simeq P(1,n)$ thus, by \Cref{lemma:pathext}\ref{lemma:pathext:princ}, $\ov \CH(I)$ is a cycle graph, $\CH(I)$ has at least two sinks, and ``$\Rightarrow$'' follows.

``$\Leftarrow$'' By \Cref{thm:digraphdegmax2}\ref{thm:digraphdegmax2:cycle:multsink}, every poset $I$ with $\ov \CH(I)$ being a cycle graph and $\CH(I)$ having at least two sinks is  principal of the Dynkin type $\Dyn_I=\AA_{n-1}$.\medskip

\ref{thm:a:main:crkbiggeri} Our aim is to show that the assumption $\Dyn_I=\AA_{m}$ yields $m\in\{n, n-1\}$, i.e., $I$ is either positive or principal. The proof is divided into two parts. First, we show that $m\neq n-2$. Then, using this result, we show that $m> n-2$.\smallskip

\textbf{Part $1^\circ$} 
[$m\neq n-2$] Assume, by contradiction, that $I$ is a connected non-negative poset of rank $n-2$ and Dynkin type $\Dyn_I=\AA_{n-2}$. Since there are no such posets of size $n\leq16$ (see \cite[Corollary 4.4(b)]{gasiorekAlgorithmicStudyNonnegative2015}), without loss of generality, we can assume that $n>16$. 

By \Cref{fact:specialzbasis}\ref{fact:specialzbasis:existance}, there exists such  a basis $h^{k_1}, h^{k_2}$ of the free abelian group $\Ker q_I\subseteq \ZZ^n$, that $h^{k_1}_{k_1} = h^{k_2}_{k_2} = 1$ and $h^{k_1}_{k_2} = h^{k_2}_{k_1} = 0$ where \mbox{$1 \leq k_1<k_2 \leq n$}.  Consider the posets $J_1\eqdef I^{(k_1)}$ and $J_2\eqdef I^{(k_2)}$. Since $\smash{J_1^{(k_2)}=J_2^{(k_1)}=I^{(k_1,k_2)}}$, the posets $J_1$ and $J_2$ are connected principal  of Dynkin type $\AA_{m}$, see \Cref{fact:specialzbasis}\ref{fact:specialzbasis:subbigraph} and \Cref{df:Dynkin_type}. It follows that  $\ov \CH(J_1)$ and $\ov \CH(J_2)$ are cycle graphs and the Hasse digraphs $\CH(J_1)$ and $\CH(J_2)$ have at least two sinks, see~\ref{thm:a:main:princ}. That is, one of the following conditions holds for the poset $I$:
\begin{enumerate}[label=\normalfont{(\roman*)}, itemsep=2.5ex]
    \item\label{thm:a:main:crkbiggeri:prf:i}
    $\ov\CH(I)$ is a cycle graph and $\CH(I)$ has at least two sinks; 
    \item\label{thm:a:main:crkbiggeri:prf:ii} $\CH(I)$ has at least two sinks and is of the shape
    $\smash{\CH(I)\colon
    \tikzsetnextfilename{scnd_prf_2}\begin{tikzpicture}[baseline=(n5.base),label distance=-2pt,xscale=0.55, yscale=0.51]
\node[circle, fill=black, inner sep=0pt, minimum size=3.5pt, label=below:$ $] (n1) at (0  , 2  ) {};
\node[circle, fill=black, inner sep=0pt, minimum size=3.5pt, label=below:$ $] (n2) at (1  , 2  ) {};
\node[circle, fill=black, inner sep=0pt, minimum size=3.5pt, label=below:$ $] (n3) at (2  , 2  ) {};
\node[circle, fill=black, inner sep=0pt, minimum size=3.5pt, label={[yshift=0.2ex]right:$\scriptscriptstyle k_1$}] (n4) at (3  , 3  ) {};
\node[circle, fill=black, inner sep=0pt, minimum size=3.5pt, label=left:$\scriptscriptstyle k_2$] (n5) at (3  , 1  ) {};
\node[circle, fill=black, inner sep=0pt, minimum size=3.5pt, label=below:$ $] (n6) at (4  , 2  ) {};
\node[circle, fill=black, inner sep=0pt, minimum size=3.5pt, label=below:$ $] (n7) at (5  , 2  ) {};
\node[circle, fill=black, inner sep=0pt, minimum size=3.5pt, label=below:$ $] (n8) at (6  , 2  ) {};
\node[circle, fill=black, inner sep=0pt, minimum size=3.5pt, label=below:$ $] (n9) at (6  , 0.4  ) {};
\node[circle, fill=black, inner sep=0pt, minimum size=3.5pt, label=below:$ $] (n10) at (5  , 0.4  ) {};
\node[circle, fill=black, inner sep=0pt, minimum size=3.5pt, label=below:$ $] (n11) at (4  , 0.4  ) {};
\node[circle, fill=black, inner sep=0pt, minimum size=3.5pt, label=below:$ $] (n12) at (2  , 0.4  ) {};
\node[circle, fill=black, inner sep=0pt, minimum size=3.5pt, label=below:$ $] (n13) at (1  , 0.4  ) {};
\node[circle, fill=black, inner sep=0pt, minimum size=3.5pt, label=below:$ $] (n14) at (0  , 0.4  ) {};
\foreach \x/\y in {1/2, 2/3, 6/7, 7/8, 8/9, 9/10, 10/11, 12/13, 13/14, 14/1}
        \draw [-, shorten <= 2.50pt, shorten >= 2.50pt] (n\x) to  (n\y);
\foreach \x/\y in {3/4, 3/5, 4/6, 5/6}
        \draw [-stealth, shorten <= 2.50pt, shorten >= 2.50pt] (n\x) to  (n\y);
\draw [line width=1.2pt, line cap=round, dash pattern=on 0pt off 5\pgflinewidth, -, shorten <= -1.50pt, shorten >= 2.50pt] (n11) to  (n12);
        \end{tikzpicture}}$;
        \item\label{thm:a:main:crkbiggeri:prf:iii} $I$ contains  an indefinite subposet $\CF_1$, $\CF_2$, $\CF_3$ or $\CF_4$.
\end{enumerate}

In the case~\ref{thm:a:main:crkbiggeri:prf:i}, the poset $I$ is principal, i.e., $m=n-1$, which contradicts the assumption. Now, we show that the same goes for~\ref{thm:a:main:crkbiggeri:prf:ii}. Without loss of generality, we may assume that the Hasse digraph $\CH(I)$ has the following form
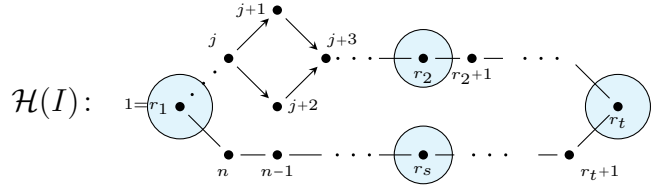
\begin{center}
$\CH(I)\colon$
\tikzsetnextfilename{scnd_prf_3}\begin{tikzpicture}[baseline={([yshift=-7pt]current bounding box)},label distance=-2pt,xscale=0.64, yscale=0.64]
        \draw[draw, fill=cyan!10](0,1) circle (19pt);
        \draw[draw, fill=cyan!10](5,2) circle (17pt);
        \draw[draw, fill=cyan!10](9,1) circle (19pt);
        \draw[draw, fill=cyan!10](5,0) circle (17pt);
        
        \node[circle, fill=black, inner sep=0pt, minimum size=3.5pt, label=above left:$\scriptscriptstyle j$] (n1) at (1, 2  ) {};
        \node[circle, fill=black, inner sep=0pt, minimum size=3.5pt, label=left:$\scriptscriptstyle j\mpppss 1$] (n2) at (2, 3  ) {};
        \node[circle, fill=black, inner sep=0pt, minimum size=3.5pt, label=right:$\scriptscriptstyle j\mpppss 2$] (n3) at (2, 1  ) {};
        \node[circle, fill=black, inner sep=0pt, minimum size=3.5pt, label={[xshift=-0.7ex]above right:$\scriptscriptstyle j\mpppss 3$}] (n4) at (3, 2  ) {};
        \node[circle, fill=black, inner sep=0pt, minimum size=3.5pt, label=left:{$\scriptscriptstyle 1=r_1$}] (n5) at (0, 1  ) {};
        \node (n6) at (4, 2  ) {$\scriptscriptstyle $};
        \node[circle, fill=black, inner sep=0pt, minimum size=3.5pt, label=below:$\scriptscriptstyle \phantom{\mpppss }r_2\phantom{\mpppss }$] (n7) at (5, 2  ) {};
        \node[circle, fill=black, inner sep=0pt, minimum size=3.5pt, label=below:$\scriptscriptstyle r_2\mpppss 1$] (n8) at (6, 2  ) {};
        \node (n9) at (7, 2  ) {$\scriptscriptstyle $};
        \node (n10) at (8, 2  ) {$\scriptscriptstyle $};
        \node[circle, fill=black, inner sep=0pt, minimum size=3.5pt, label=below:$\scriptscriptstyle r_t$] (n11) at (9, 1  ) {};
        \node[circle, fill=black, inner sep=0pt, minimum size=3.5pt, label=below right:$\scriptscriptstyle r_t\mpppss 1$] (n12) at (8, 0  ) {};
        \node (n13) at (7, 0  ) {$\scriptscriptstyle $};
        \node (n14) at (6, 0  ) {$\scriptscriptstyle $};
        \node[circle, fill=black, inner sep=0pt, minimum size=3.5pt, label=below:$\scriptscriptstyle r_s$] (n15) at (5, 0  ) {};
        \node (n16) at (4, 0  ) {$\scriptscriptstyle $};
        \node (n17) at (3, 0  ) {$\scriptscriptstyle $};
        \node[circle, fill=black, inner sep=0pt, minimum size=3.5pt, label=below:$\scriptscriptstyle n\mppmss 1$] (n18) at (2, 0  ) {};
        \node[circle, fill=black, inner sep=0pt, minimum size=3.5pt, label=below:$\scriptscriptstyle n\phantom{\mppmss }$] (n19) at (1, 0  ) {};
        \foreach \x/\y in {1/2, 1/3, 2/4, 3/4}
        \draw [-stealth, shorten <= 2.50pt, shorten >= 2.50pt] (n\x) to  (n\y);
        \draw [line width=1.2pt, line cap=round, dash pattern=on 0pt off 5\pgflinewidth, -, shorten <= 4.50pt, shorten >= 4.50pt] (n5) to  (n1);
        \draw [line width=1.2pt, line cap=round, dash pattern=on 0pt off 5\pgflinewidth, -, shorten <= -2.50pt, shorten >= -2.50pt] (n9) to  (n10);
        
        \draw [line width=1.2pt, line cap=round, dash pattern=on 0pt off 5\pgflinewidth, -, shorten <= 2.50pt, shorten >= -2.50pt] (n4) to  (n6);
        \draw [line width=1.2pt, line cap=round, dash pattern=on 0pt off 5\pgflinewidth, -, shorten <= -3.50pt, shorten >= -3.50pt] (n14) to  (n13);

        \draw [line width=1.2pt, line cap=round, dash pattern=on 0pt off 5\pgflinewidth, -, shorten <= .10pt, shorten >= -2.50pt] (n17) to  (n16);
        
        \foreach \x/\y in {6/7, 10/11, 16/15}
        \draw [-, shorten <= -2.50pt, shorten >= 2.50pt] (n\x) to  (n\y);
        \foreach \x/\y in {5/19, 7/8, 8/9, 12/11, 13/12, 15/14, 19/18}
        \draw [-, shorten <= 2.50pt, shorten >= 2.50pt] (n\x) to  (n\y);
        \draw [-, shorten <= 2.50pt, shorten >= -2.50pt] (n18) to  (n17);
\end{tikzpicture}
\end{center}
where:
\begin{itemize}
    \item every  $r_1,\ldots,r_s\in \{1,\ldots, n\}\setminus \{j+1,j+2\}$ is either a source or a sink,
    \item $1 = r_1 < r_2 < \cdots < r_t<\cdots < r_s$ where $s>2$ is an even number,
    \item subdigraphs $\{1,\ldots,j,j+1,j+3,\ldots r_2\}$, $\{1,\ldots, j,j+2, j+3,\ldots,r_2\}$,  $\{r_s,\ldots,n, 1\}$ and $\{r_{t-1},\ldots,r_{t}\}$, where $3\leq t\leq s$, have exactly one sink.
\end{itemize}
Since the quadratic form  $q_{I}\colon \ZZ^n\to \ZZ$ \eqref{eq:quadratic_form} is given by the formula:
{\allowdisplaybreaks\begin{align*}
    q_{I}(x) =&\sum_{\mathclap{1\leq i\leq n}} x_i^2 \,\,+
    \sum_{\substack{1\leq i < r_2 \\ i\neq j+1}} x_i
    \sum_{\mathclap{i<k\leq r_2}}x_k + x_{j+1}
    \sum_{\mathclap{j+3\leq k \leq r_2}}x_k +
    \sum_{{2\leq t < s}}\left(\sum_{r_t\leq i<k \leq r_{t+1}} x_i x_k\right)  \\
    &+\sum_{r_s\leq i<k \leq n} x_i x_k + x_1(x_{r_s}\!+\cdots+ x_n)\\
    =& 
    \sum_{i\not\in\{r_1,\ldots,r_s,j+1,j+2\}} \frac{1}{2}x_i^2 +
    \frac{1}{2}(x_{j+1}-x_{j+2})^2 +
    \frac{1}{2} \sum_{1\leq t < s} \left(\sum_{r_t\leq i\leq r_{t+1}} x_i\right)^2\\
    & +\frac{1}{2}(x_{r_s} + \cdots + x_{n} + x_1)^2,
\end{align*}}%
then $q_I(v)\geq 0$ for every $v\in\ZZ^n$ and $I$ is non-negative. Consider the non-zero vector $h=[h_1,\ldots,h_n]\in\ZZ^n$, where $h_{r_i}=1$ if $i$ is odd, $-1$ if $i$ is even, and $h_k=0$ for $k\neq r_i$. It is straightforward to check that
\[
q_I(h) = \tfrac{1}{2}(h_{r_1}+h_{r_2})^2+\cdots+\tfrac{1}{2}(h_{r_{s-1}}+h_{r_s})^2 +\tfrac{1}{2}(h_{r_s}+h_1)^2=0,
\]
i.e., the poset $I$ is not positive and $h\in\Ker q_I$. Since the vector $h$ has the first coordinate equal $h_1=1$ and $I^{(1)}\simeq\CD_{n-r,s}^{[2]}$ is a positive poset of Dynkin type $\DD_{n-1}$ (see \Cref{lemma:pathext}\ref{lemma:pathext:posit:d} and \Cref{df:Dynkin_type}),  it follows that $I$ is principal of Dynkin type $\DD_{n-1}$. This contradicts the assumption that $\Dyn_I=\AA_{n-2}$.\smallskip

To finish this part of the proof, we note that every poset $I$ that is not described in~\ref{thm:a:main:crkbiggeri:prf:i} or~\ref{thm:a:main:crkbiggeri:prf:ii} contains (as a subposet) one of the posets  $\CF_1$, $\CF_2$, $\CF_3$ or $\CF_4$ presented in \Cref{tbl:indefposets},  hence is indefinite. This follows by the standard case-by-case inspection, as described in the proof of \Cref{lemma:pathext}. Details are left to the reader.\medskip

\textbf{Part $2^\circ$} [$m> n-2$] Let $I$ be a  connected non-negative poset of   rank $m$ and Dynkin type $\Dyn_I=\AA_{m}$. We show that the assumption  $m\in\{1,\ldots,n-3\}$ yields a contradiction. It follows from \Cref{fact:specialzbasis}\ref{fact:specialzbasis:existance} that there exists such  a basis $h^{k_1},\ldots, h^{k_r}$ of the free abelian group $\Ker q_I\subseteq \ZZ^n$, that $h^{k_i}_{k_i} = 1$ and $h^{k_i}_{k_j} = 0$, for $1 \leq i,j \leq r$ and $i \neq j$, where $r=n-m$ and $1 \leq k_1 < \ldots < k_r \leq n$. Moreover, by \Cref{fact:specialzbasis}\ref{fact:specialzbasis:subbigraph}, the poset $J\eqdef I^{(k_3,\ldots,k_r)}$ is connected non-negative of size $n'=m+2$ and rank $m=n'-2$. Since $J^{(k_1,k_2)} = I^{(k_1,\ldots,k_r)}\sim_\ZZ\AA_m$, it follows that $\Dyn_J=\AA_{n'-2}$ which yields a contradiction with \textbf{Part $1^\circ$}.
\end{proof}

\section{Enumeration of \texorpdfstring{$\AA_n$}{An} Dynkin type non-negative posets}

We finish the manuscript by giving explicit formulae~\eqref{fact:digrphnum:path:eq} and~\eqref{fact:digrphnum:cycle:eq} for the number of all possible orientations of the path and cycle graphs, up to  isomorphism of \textit{unlabeled} digraphs. We apply these results to devise the formula~\eqref{thm:typeanum:eq} for the number of non-negative Dynkin type $\AA_m$ posets of size~$n$.

\begin{fact}\label{fact:digrphnum:path} There are exactly
\begin{equation}\label{fact:digrphnum:path:eq}
	ONum(P_n)=
	\begin{cases}
		2^{\frac{n - 3}{2}} + 2^{n - 2}, & \textnormal{if $n\geq 1$ is odd,}\\
		2^{n-2}, & \textnormal{if $n\geq 2$ is even},\\
	\end{cases}
\end{equation}
digraphs $D$ with $\ov D\simeq P_n\eqdef 1 \,\rule[2.5pt]{22pt}{0.4pt}\,2\,\rule[2.5pt]{22pt}{0.4pt}\,
\hdashrule[2.5pt]{12pt}{0.4pt}{1pt}\,\rule[2.5pt]{22pt}{.4pt}\,n$, up to the isomorphism. 
\end{fact}
\begin{proof}
Here we follow arguments given in the proof of~\cite[Proposition 6.7]{gasiorekAlgorithmicCoxeterSpectral2020}. To calculate the number of non-isomorphic orientations among all $2^{n-1}$ possible orientations of edges of $P_n$, we consider two cases.
\begin{enumerate}[label=\normalfont{(\roman*)},wide]
    \item\label{fact:digrphnum:path:prf:i} First, assume that $|I|=n\geq 2$ is an even number. In this case, every digraph $I$  has exactly two representatives among $2^{n-1}$ edge orientations: one drawn ``from the left'' and the other ``from the right'' (i.e., symmetric along the path). Therefore, the  number $ONum(P_n)$ of all such  non-isomorphic digraphs equals $2^{n-2}$.
    
    \item\label{fact:digrphnum:path:prf:ii} Now we assume that $|I|=n\geq 1$ is an odd number. If $n=1$, then, up to  isomorphism, there exists exactly $1=2^{-1}+2^{-1}$ digraph. Otherwise, $n\geq 3$ and  among all $2^{n-1}$ edge orientations:
    \begin{itemize}
        \item  digraphs that are ``symmetric'' along the path have exactly one representation, and
        \item  the rest of digraphs have exactly two representations, analogously as in~\ref{fact:digrphnum:path:prf:i}.
    \end{itemize} 
    It is straightforward to check that there are  $2^\frac{n-1}{2}$ ``symmetric'' path digraphs, therefore in this case we obtain
    \[
    ONum(P_n)=\frac{2^{n-1}-2^\frac{n-1}{2}}{2}+2^\frac{n-1}{2}=2^{\frac{n - 3}{2}} + 2^{n - 2}.\qedhere
    \]
\end{enumerate}
\end{proof}

\begin{remark}
The formula~\eqref{fact:digrphnum:path:eq} describes the number of various combinatorial objects. For example the number of linear oriented trees with $n$ arrows or unique symmetrical triangle quilt patterns along the diagonal of an $n\times n$ square, see~\cite[OEIS sequence A051437]{oeis_A051437}.
\end{remark}

By \Cref{thm:a:main}\ref{thm:a:main:posit}, the Hasse digraph $\CH(I)$ of every positive connected poset $I$ of Dynkin type $\Dyn_I=\AA_n$ is an oriented path graph, as suggested in~\cite[Conjecture 6.4]{gasiorekAlgorithmicCoxeterSpectral2020}. Hence, \Cref{fact:digrphnum:path} gives an exact formula for the number of all, up to the poset isomorphism, such connected posets $I$ (see also \cite[Proposition 6.7]{gasiorekAlgorithmicCoxeterSpectral2020}). 

\begin{corollary}\label{cor:posit:num:poset}
Given $n\geq 1$, the total number of all finite non-isomorphic connected positive posets $I$ of Dynkin type $\AA_n$ and size $n$ equals $Nneg(n,\AA_n) \eqdef ONum(P_n)$.
\end{corollary}

Similarly, the description given in \Cref{thm:a:main}\ref{thm:a:main:princ} makes it possible to count all connected principal posets $I$ of Dynkin type $\AA_{n-1}$. First, we need to know the exact number of all, up to isomorphism, orientations of the cycle graph $C_n$.

\begin{fact}\label{fact:digrphnum:cycle}
Let $C_n \eqdef$%
\tikzsetnextfilename{cycle_cn}\begin{tikzpicture}[baseline=(n11.base),label distance=-2pt,xscale=0.65, yscale=0.74]
    \node[circle, fill=black, inner sep=0pt, minimum size=3.5pt, label={[name=n11]left:$1$}] (n1) at (0, 0  ) {};
    \node[circle, fill=black, inner sep=0pt, minimum size=3.5pt, label={[yshift=-0.3ex]above:$\scriptscriptstyle 2$}] (n2) at (1, 0  ) {};
    \node (n3) at (2, 0  ) {$ $};
    \node[circle, fill=black, inner sep=0pt, minimum size=3.5pt, label=right:$n$] (n4) at (5, 0  ) {};
    \node (n5) at (3, 0  ) {$ $};
    \node[circle, fill=black, inner sep=0pt, minimum size=3.5pt, label={[yshift=-0.5ex]above:$\scriptscriptstyle n\mppmss 1$}] (n6) at (4, 0  ) {};
    \draw[shorten <= 2.50pt, shorten >= 2.50pt] (n1) .. controls (0.2,0.6) and (4.8,0.6) .. (n4);
    \draw [line width=1.2pt, line cap=round, dash pattern=on 0pt off 5\pgflinewidth, -, shorten <= -2.50pt, shorten >= -2.50pt] (n3) to  (n5);
    \foreach \x/\y in {1/2, 6/4}
    \draw [-, shorten <= 2.50pt, shorten >= 2.50pt] (n\x) to  (n\y);
    \draw [-, shorten <= 2.50pt, shorten >= -2.50pt] (n2) to  (n3);
    \draw [-, shorten <= -2.50pt, shorten >= 2.50pt] (n5) to  (n6);
\end{tikzpicture}%
be the cycle graph on $n\geq 3$ vertices. The number $ONum(C_n)$ of digraphs $\CD$ with $\ov \CD=C_ n$, up to the isomorphism, equals
\begin{equation}\label{fact:digrphnum:cycle:eq}
    ONum(C_n)=
    \begin{cases}
        \frac{1}{2n} \sum_{d\mid n}\left(2^{\frac{n}{d}}\varphi(d)\right), & \textnormal{if $n\geq 3$ is odd,}\\[0.1cm]
        \frac{1}{2n} \sum_{d\mid n}\left(2^{\frac{n}{d}}\varphi(d)\right)+ 2^{\frac{n}{2}-2}, & \textnormal{if $n\geq 4$ is even},\\
    \end{cases}
\end{equation}
where $\varphi$ is the Euler's totient function.
\end{fact}
\begin{proof}
Assume that $\CD$ is such a digraph that $\ov \CD=C_n$. Without loss of generality, we may assume that $\CD$ is depicted in the circle layout (on the plane), and its arrows are labeled with two colors:
\begin{itemize}
    \item \textit{black}: if the arrow is  clockwise oriented, and
    \item \textit{white}: if the arrow is counterclockwise oriented.
\end{itemize} 
That is, every $\CD$ can be viewed as a binary combinatorial necklace $\CN_2(n)$.
For example, for $n=5$ there exist $32$ orientations of edges of the cycle  $C_5$  that yield exactly $8$ different binary necklaces of length $5$ shown in \Cref{tab:cycle_nekl_5}.
\begin{longtable}{|@{}c@{}c@{}|@{}c@{}c@{}|}\hline
    \ringv{a}{0}{0}{0}{0}{0} & \ringv{b}{1}{1}{1}{1}{1} & \ringv{c}{0}{0}{0}{0}{1} & \ringv{d}{0}{1}{1}{1}{1}\\\hline
    \ringv{e}{0}{0}{0}{1}{1} & \ringv{f}{0}{0}{1}{1}{1} & \ringv{g}{0}{0}{1}{0}{1} & \ringv{h}{0}{1}{0}{1}{1}\\\hline
	\caption{Binary combinatorial necklaces of length $5$}\label{tab:cycle_nekl_5}
\end{longtable}

Moreover, up to digraph isomorphism, every $\CD$ has in this case exactly two representations among the necklaces: ``clockwise'' and ''anticlockwise'' ones, as shown in the \Cref{tab:cycle_nekl_5} (isomorphic digraphs are gathered in boxes).\pagebreak

On the other hand, if $|\CD|=n\geq 4$ is an even number, certain digraphs have exactly one representation among necklaces. As an illustration, let us consider the $n=6$ case.
\begin{longtable}{@{}|@{}c@{}c@{}|@{}c@{}c@{}|@{}c@{}c@{}c@{}}\cline{1-6} 
    \ringvi{a}{0}{0}{0}{0}{0}{0} & \ringvi{b}{1}{1}{1}{1}{1}{1} & \ringvi{c}{0}{0}{0}{0}{0}{1} & \ringvi{d}{0}{1}{1}{1}{1}{1} & \ringvi{e}{0}{0}{0}{0}{1}{1} & \multicolumn{1}{@{}c@{}|@{}}{\ringvi{f}{0}{0}{1}{1}{1}{1}} & \ringvi{i}{0}{0}{0}{1}{1}{1}\\\cline{1-6}     
	\ringvi{g}{0}{0}{0}{1}{0}{1} & \ringvi{h}{0}{1}{0}{1}{1}{1} & \ringvi{j}{0}{0}{1}{0}{0}{1} & \ringvi{k}{0}{1}{1}{0}{1}{1} & \ringvi{l}{0}{0}{1}{0}{1}{1} & \ringvi{m}{0}{1}{1}{0}{1}{0} & \ringvi{n}{0}{1}{0}{1}{0}{1}\\\cline{1-4}
    \caption{Binary combinatorial necklaces of length $6$}\label{tab:cycle_nekl_6}
\end{longtable}
Every such a ``rotationally symmetric'' digraph is uniquely determined by a directed path graph of length $\frac{n}{2} + 1$ and, by \Cref{fact:digrphnum:path}, there are exactly $2^{\frac{n}{2}-1}$ such digraphs. 

Now we show that every isomorphism $f\colon\{1,\ldots,n\}\to\{1,\ldots,n\}$ of digraphs $\CD_1$ and $\CD_2$ with $\ov \CD_1=\ov \CD_2=C_n$ has a form of a ``clockwise'' or ``anticlockwise'' \textit{rotation}. Fix a vertex $v_1\in \CD_1$ and consider the sequence $v_1,v_2,\ldots,v_n$ of vertices, where $v_{i+1}$ is a ``clockwise'' neighbour of $v_i$ (i.e., we have either $v_i\to v_{i+1}$ \textit{black} arrow or $v_i\gets v_{i+1}$ \textit{white} arrow in digraph $\CD_1$). One of two possibilities holds: $f(v_2)$ is either a ``clockwise'' neighbor of $f(v_{1})$ or an ``anticlockwise'' one. If the first possibility holds, then $f(v_{i+1})$ is a ``clockwise'' neighbour of $f(v_i)$ for every $2\leq i < n$, hence isomorphism $f$ encode a ``clockwise'' rotation. On the other hand, the assumption that $f(v_2)$ is an ``anticlockwise'' neighbour of $f(v_{1})$ implies that $f(v_{i+1})$ is an ``anticlockwise'' neighbour of $f(v_i)$, i.e.,  $f$ encodes an ``anticlockwise'' rotation.\smallskip

Summing up, if $|\CD|=n\geq 3$ is an odd number, the  digraph $\CD$  has exactly two representatives among binary necklaces $\CN_2(n)$: one  ``clockwise'' and the other ``anticlockwise''. Since $|\CN_2(n)|=\frac{1}{n} \sum_{d\mid n}\left(2^{\frac{n}{d}}\varphi(d)\right)$, see~\cite{riordanCombinatorialSignificanceTheorem1957}, the first part of the equality~\eqref{fact:digrphnum:cycle:eq} follows. Assume now that $|\CD|=n\geq 4$ is an even number. In the formula $|\CN_2(n)|/2$ we count only half (i.e., $2^{\frac{n}{2}-2}$)  of ``rotationally symmetric'' digraphs. Hence, $ONum(C_n)=\frac{1}{2n} \sum_{d\mid n}\left(2^{\frac{n}{d}}\varphi(d)\right)+ 2^{\frac{n}{2}-2}$ in this case.
\end{proof}

In the proof of \Cref{fact:digrphnum:cycle} we show that the number of cyclic graphs with oriented edges, up to the symmetry of the dihedral group, coincides with the number of such digraphs, up to isomorphism of unlabeled digraphs.

\begin{remark}
The formula~\eqref{fact:digrphnum:cycle:eq} is described in~\cite[OEIS sequence A053656]{oeis_A053656} and, among others, counts the number of minimal fibrations of a bidirectional $n$-cycle over the $2$-bouquet (up to precompositions with automorphisms of the $n$-cycle), see~\cite{boldiFibrationsGraphs2002}.
\end{remark}

\begin{corollary}\label{cor:cycle_pos:dag_dyna:num}
Let $n\geq 3$ be an integer. Then, up to isomorphism, there exists exactly:
\begin{enumerate}[label=\normalfont{(\alph*)}]
    \item\label{cor:cycle_pos:dag_dyna:num:cycle} $ONum(C_n)-1$ directed acyclic graphs $\CD$ whose underlying graph is  $\ov \CD=C_n$,
    \item\label{cor:cycle_pos:dag_dyna:num:poset} $Nneg(n,\AA_{n-1})=ONum(C_n)-\lceil\frac{n+1}{2}\rceil$ principal posets $I$ of Dynkin type $\AA_{n-1}$,
\end{enumerate}
where $ONum(C_n)$  is given by the formula \eqref{fact:digrphnum:cycle:eq}.
\end{corollary}

\begin{proof}
Since, up to isomorphism, there exists exactly one cyclic orientation of the $C_n$
graph, \ref{cor:cycle_pos:dag_dyna:num:cycle}  follows directly from \Cref{fact:digrphnum:cycle}.\pagebreak

To prove~\ref{cor:cycle_pos:dag_dyna:num:poset}, we note that by \Cref{thm:a:main}\ref{thm:a:main:princ} it is sufficient to count all oriented cycles that have at least two sinks. Since, among all possible orientations of a cycle, there are:
\begin{itemize}
	 \item $\lfloor\frac{n}{2}\rfloor$ cycles with exactly one sink,
	 \item $1$ oriented cycle
\end{itemize}
and $\lfloor\frac{n}{2}\rfloor+1=\lceil\frac{n+1}{2}\rceil$, the statement~\ref{cor:cycle_pos:dag_dyna:num:poset} follows from \Cref{fact:digrphnum:cycle}.
\end{proof}

\begin{center}
\textbf{Proof of \Cref{thm:typeanum}}
\end{center}

Now, we can devise an exact formula for the total number of non-negative posets of size $n$ and Dynkin type $\AA_m$. 

\begin{proof}[Proof of \Cref{thm:typeanum}]
We note that $Nneg(n,\AA)=Nneg(n,\AA_n)+Nneg(n,\AA_{n-1})$ by \Cref{thm:a:main}, hence by \Cref{cor:posit:num:poset} and \Cref{cor:cycle_pos:dag_dyna:num}\ref{cor:cycle_pos:dag_dyna:num:poset}, for $n\ge 3$ we have
\begin{equation*}
Nneg(n,\AA)= 
\begin{cases}
\frac{1}{2n} \sum_{d\mid n}\left(2^{\frac{n}{d}}\varphi(d)\right)+ 2^{n - 2} + 2^{\frac{n - 3}{2}}-\lceil\frac{n+1}{2}\rceil, & \textnormal{if $n\geq 3$ is odd,}\\
\frac{1}{2n} \sum_{d\mid n}\left(2^{\frac{n}{d}}\varphi(d)\right) + 2^{n-2} + 2^{\frac{n}{2}-2}-\lceil\frac{n+1}{2}\rceil, & \textnormal{if $n\geq 4$ is even}.\\
\end{cases}\end{equation*}
Since $\frac{n - 3}{2} = \lceil\frac{n}{2}-2\rceil$ for odd values of $n$ and $Nneg(1,\AA)=Nneg(2,\AA) = 1$, it follows that
 \begin{equation*}
    Nneg(n,\AA)=\frac{1}{2n} \sum_{d\mid n}\big(2^{\frac{n}{d}}\varphi(d)\big) + 
    \big\lfloor 2^{n - 2} + 2^{\lceil\frac{n}{2}-2\rceil} - \tfrac{n+1}{2}\big\rfloor
\end{equation*}
for any $n\geq 1$.
\end{proof}

Summing up, we get the following asymptotic description of connected non-negative posets $I$ of Dynkin type $\AA_m$.
\begin{corollary}
	Let $Nneg(n,\AA_{n-1})$ be the number of principal posets $I$ of Dynkin type $\AA_{n-1}$ and $Nneg(n,\AA)$ be the number of all non-negative posets $I$ of size $n$ and Dynkin type $\AA_{m}$. Then
	\begin{enumerate}[label=\normalfont{(\alph*)}]
		\item  $\lim_{n\to \infty} \frac{Nneg(n+1,\AA)}{Nneg(n,\AA)}=2$ and $Nneg(n,\AA)\approx 2^{n-2}$, 
		\item the  number of connected non-negative posets of Dynkin type $\AA_{m}$ grows exponentially,
		\item $\lim_{n\to\infty}\frac{Nneg(n,\AA_{n-1})}{Nneg(n,\AA)}=0$, hence almost all such posets are positive.
	\end{enumerate}
\end{corollary}
\begin{proof}
	Apply \Cref{cor:cycle_pos:dag_dyna:num}\ref{cor:cycle_pos:dag_dyna:num:poset} and \Cref{thm:typeanum}.
\end{proof}

\begin{landscape}
	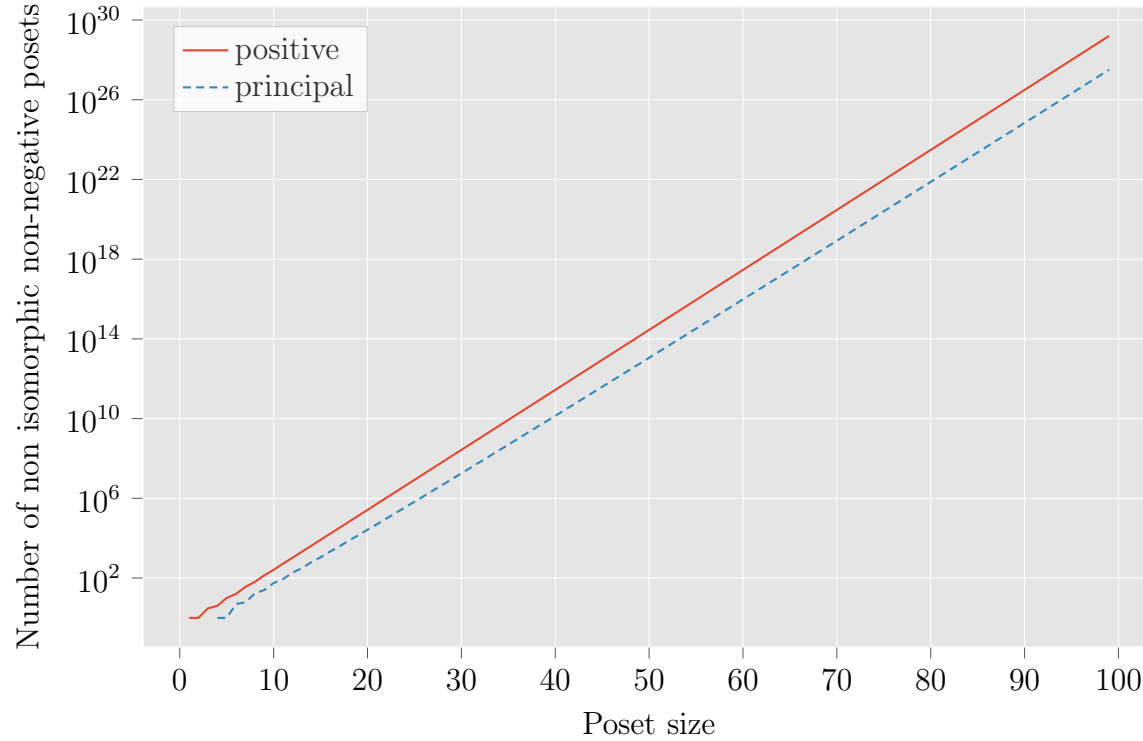
\begin{figure}[H]
        \centering
\tikzsetnextfilename{plot}
\begin{tikzpicture}
        \definecolor{chocolate2267451}{RGB}{226,74,51}
        \definecolor{dimgray85}{RGB}{85,85,85}
        \definecolor{gainsboro229}{RGB}{229,229,229}
        \definecolor{lightgray204}{RGB}{204,204,204}
        \definecolor{steelblue52138189}{RGB}{52,138,189}
        
        \begin{axis}[
                axis background/.style={fill=gainsboro229},
                axis line style={white},
                height=10cm,
                legend cell align={left},
                legend style={
                        fill opacity=0.8,
                        draw opacity=1,
                        text opacity=1,
                        at={(0.03,0.97)},
                        anchor=north west,
                        draw=lightgray204
                },
                log basis y={10},
                tick align=outside,
                tick pos=left,
                width=15cm,
                x grid style={white},
                xlabel={Poset size},
                xmajorgrids,
                xmin=-3.9, xmax=103.9,
                xtick style={color=dimgray85},
                y grid style={white},
                ylabel={Number of non isomorphic non-negative posets},
                ymajorgrids,
                ymin=0.0346740460021202, ymax=4.56988275953839e+30,
                ymode=log,
                ytick style={color=dimgray85},
                ytick={1e-06,0.01,100,1000000,10000000000,100000000000000,1e+18,1e+22,1e+26,1e+30,1e+34,1e+38},
                yticklabels={
                        \(\displaystyle {10^{-6}}\),
                        \(\displaystyle {10^{-2}}\),
                        \(\displaystyle {10^{2}}\),
                        \(\displaystyle {10^{6}}\),
                        \(\displaystyle {10^{10}}\),
                        \(\displaystyle {10^{14}}\),
                        \(\displaystyle {10^{18}}\),
                        \(\displaystyle {10^{22}}\),
                        \(\displaystyle {10^{26}}\),
                        \(\displaystyle {10^{30}}\),
                        \(\displaystyle {10^{34}}\),
                        \(\displaystyle {10^{38}}\)
                }
                ]
                \addplot [thick, chocolate2267451]
                table {%
                        1 1
                        2 1
                        3 3
                        4 4
                        5 10
                        6 16
                        7 36
                        8 64
                        9 136
                        10 256
                        11 528
                        12 1024
                        13 2080
                        14 4096
                        15 8256
                        16 16384
                        17 32896
                        18 65536
                        19 131328
                        20 262144
                        21 524800
                        22 1048576
                        23 2098176
                        24 4194304
                        25 8390656
                        26 16777216
                        27 33558528
                        28 67108864
                        29 134225920
                        30 268435456
                        31 536887296
                        32 1073741824
                        33 2147516416
                        34 4294967296
                        35 8590000128
                        36 17179869184
                        37 34359869440
                        38 68719476736
                        39 137439215616
                        40 274877906944
                        41 549756338176
                        42 1099511627776
                        43 2199024304128
                        44 4398046511104
                        45 8796095119360
                        46 17592186044416
                        47 35184376283136
                        48 70368744177664
                        49 140737496743936
                        50 281474976710656
                        51 562949970198528
                        52 1.12589990684262e+15
                        53 2.25179984723968e+15
                        54 4.5035996273705e+15
                        55 9.00719932184986e+15
                        56 1.8014398509482e+16
                        57 3.60287971531817e+16
                        58 7.20575940379279e+16
                        59 1.44115188344291e+17
                        60 2.88230376151712e+17
                        61 5.76460752840294e+17
                        62 1.15292150460685e+18
                        63 2.30584301028744e+18
                        64 4.61168601842739e+18
                        65 9.22337203900226e+18
                        66 1.84467440737096e+19
                        67 3.68934881517141e+19
                        68 7.37869762948382e+19
                        69 1.47573952598266e+20
                        70 2.95147905179353e+20
                        71 5.90295810375886e+20
                        72 1.18059162071741e+21
                        73 2.36118324146918e+21
                        74 4.72236648286965e+21
                        75 9.44473296580801e+21
                        76 1.88894659314786e+22
                        77 3.77789318630946e+22
                        78 7.55578637259143e+22
                        79 1.51115727452104e+23
                        80 3.02231454903657e+23
                        81 6.04462909807864e+23
                        82 1.20892581961463e+24
                        83 2.41785163923036e+24
                        84 4.83570327845852e+24
                        85 9.67140655691923e+24
                        86 1.93428131138341e+25
                        87 3.86856262276725e+25
                        88 7.73712524553363e+25
                        89 1.54742504910681e+26
                        90 3.09485009821345e+26
                        91 6.18970019642708e+26
                        92 1.23794003928538e+27
                        93 2.4758800785708e+27
                        94 4.95176015714152e+27
                        95 9.90352031428311e+27
                        96 1.98070406285661e+28
                        97 3.96140812571323e+28
                        98 7.92281625142643e+28
                        99 1.58456325028529e+29
                };
                \addlegendentry{positive}
                \addplot [thick, densely dashed, steelblue52138189]
                table {%
                        1 0
                        2 0
                        3 0
                        4 1
                        5 1
                        6 5
                        7 6
                        8 17
                        9 25
                        10 56
                        11 88
                        12 185
                        13 309
                        14 615
                        15 1088
                        16 2113
                        17 3847
                        18 7419
                        19 13788
                        20 26489
                        21 49929
                        22 95873
                        23 182350
                        24 350637
                        25 671079
                        26 1292748
                        27 2485520
                        28 4797871
                        29 9256381
                        30 17904460
                        31 34636818
                        32 67126265
                        33 130150571
                        34 252679814
                        35 490853398
                        36 954506467
                        37 1857283137
                        38 3616952517
                        39 7048151652
                        40 13744170619
                        41 26817356755
                        42 52358246192
                        43 102280151400
                        44 199912301315
                        45 390937468385
                        46 764879842415
                        47 1497207322906
                        48 2932035377905
                        49 5744387279793
                        50 11259007792596
                        51 22076468764166
                        52 43303859993497
                        53 84973577874889
                        54 166800021000937
                        55 327534518354268
                        56 643371444844535
                        57 1.26416831646405e+15
                        58 2.48474476084341e+15
                        59 4.88526061274085e+15
                        60 9.60767948245851e+15
                        61 1.89003525345384e+16
                        62 3.71910168318295e+16
                        63 7.32013653718966e+16
                        64 1.44115189183153e+17
                        65 2.83796062672455e+17
                        66 5.58992246870488e+17
                        67 1.1012981536543e+18
                        68 2.17020518956359e+18
                        69 4.27750587216466e+18
                        70 8.43279729967401e+18
                        71 1.66280509960199e+19
                        72 3.27942117042521e+19
                        73 6.46899518201321e+19
                        74 1.27631526599333e+20
                        75 2.51859545753048e+20
                        76 4.97091208793648e+20
                        77 9.81270957479407e+20
                        78 1.93738112131825e+21
                        79 3.82571461903364e+21
                        80 7.55578637287318e+21
                        81 1.49250101186991e+22
                        82 2.948599560092e+22
                        83 5.82614852826327e+22
                        84 1.15135792345376e+23
                        85 2.27562507221577e+23
                        86 4.4983286311467e+23
                        87 8.89324740865934e+23
                        88 1.75843755580759e+24
                        89 3.47735966091399e+24
                        90 6.87744466270555e+24
                        91 1.36037366954437e+25
                        92 2.69117399844828e+25
                        93 5.32447328724895e+25
                        94 1.05356599088153e+26
                        95 2.08495164511222e+26
                        96 4.12646679761865e+26
                        97 8.16785180559426e+26
                        98 1.61690127580146e+27
                        99 3.20113787936422e+27
                };
                \addlegendentry{principal}
        \end{axis}
	\end{tikzpicture}
	\captionsetup{width=.9\linewidth}
    \caption{Logarithmic scale plot of the number of connected non-negative posets $I$ of Dynkin type $\Dyn_I=\AA_m$}
\label{fig:coxgrow}
\end{figure}
{\newcommand{\msep}{\,\,\,\,}
\begin{longtable}{lr@{\msep}r@{\msep}r@{\msep}r@{\msep}r@{\msep}r@{\msep}r@{\msep}r@{\msep}r@{\msep}r@{\msep}r@{\msep}r@{\msep}r@{\msep}r@{\msep}r@{\msep}r@{\msep}r@{\msep}r@{\msep}r@{\msep}r}\toprule
        $n$ & $1$ & $2$ & $3$ & $4$ & $5$ & $6$ & $7$ & $8$ & $9$ & $10$ & $11$ & $12$ & $13$ & $14$ & $15$ & $16$ & $17$ & $18$ & $19$ & $20$\\\midrule
        positive & $\num{1}$ & $\num{1}$ & $\num{3}$ & $\num{4}$ & $\num{10}$ & $\num{16}$ & $\num{36}$ & $\num{64}$ & $\num{136}$ & $\num{256}$ & $\num{528}$ & $\num{1024}$ & $\num{2080}$ & $\num{4096}$ & $\num{8256}$ & $\num{16384}$ & $\num{32896}$ & $\num{65536}$ & $\num{131328}$ & $\num{262144}$\\
        principal & $\num{0}$ & $\num{0}$ & $\num{0}$ & $\num{1}$ & $\num{1}$ & $\num{5}$ & $\num{6}$ & $\num{17}$ & $\num{25}$ & $\num{56}$ & $\num{88}$ & $\num{185}$ & $\num{309}$ & $\num{615}$ & $\num{1088}$ & $\num{2113}$ & $\num{3847}$ & $\num{7419}$ & $\num{13788}$ & $\num{26489}$\\\midrule
        all & $\num{1}$ & $\num{1}$ & $\num{3}$ & $\num{5}$ & $\num{11}$ & $\num{21}$ & $\num{42}$ & $\num{81}$ & $\num{161}$ & $\num{312}$ & $\num{616}$ & $\num{1209}$ & $\num{2389}$ & $\num{4711}$ & $\num{9344}$ & $\num{18497}$ & $\num{36743}$ & $\num{72955}$ & $\num{145116}$ & $\num{288633}$\\\bottomrule
        \captionsetup{width=.9\linewidth}
        \caption{Number of connected non-negative posets $I$ of Dynkin type $\Dyn_I=\AA_m$ and size $1\leq |I|\leq 20$}
\end{longtable}}%
\end{landscape}

\section{Future work}\label{sec:conclusions}
In the present work, we give a complete description of connected non-negative posets $I$ of Dynkin type $\Dyn_I=\AA_m$ and, in particular, we show that $m\in\{n,n-1\}$. Computer experiments suggest that there is an upper bound for the rank of $\EE_n$ type posets as well.

\begin{conjecture}\label{conj:typeE}
If $I$ is a Dynkin type $\EE_m$ non-negative connected poset, then $m\geq n-3$.
\end{conjecture}
The conjecture yields $|I|\leq 11$, and consequently, we get the following.
\begin{conjecture}
If $I$ is  a non-negative connected poset of size $n>11$ and rank $m<n-1$, then $\Dyn_I=\DD_m$.
\end{conjecture}
In other words, checking the Dynkin type of a connected non-negative poset $I$ that has at least $n\geq 12$ elements is straightforward (compare with \cite{makurackiQuadraticAlgorithmCompute2020} and \cite{zajacPolynomialTimeInflation2020}).
\begin{proposition}
If \Cref{conj:typeE} holds, the Dynkin type of a connected non-negative poset $I$ of size $n\geq 12$ and rank $m$, encoded in the form of the adjacency list of the Hasse digraph~$\CH(I)$, can be calculated in $O(n)$. Moreover, assuming that this adjacency list is sorted by degrees of vertices, $\Dyn_I$ can be calculated in $O(1)$.
\end{proposition}
\begin{proof}
First, we note that the assumptions yield  $\Dyn_I\in\{\AA_m, \DD_m\}$. Moreover,  $\Dyn_I=\AA_m$ if and only if one of the following conditions hold:
\begin{enumerate}[label=\normalfont{(\roman*)}]
	\item $\deg_{\ov \CH(I)}(v)=2$ for all $v\in \CH(I)$, or 
	\item $\deg_{\ov \CH(I)}(v)=2$ for all but two $v_i\in \CH(I)$ with $\deg_{\ov \CH(I)}(v_i)=1$,
\end{enumerate}
see \Cref{thm:a:main}. In the pessimistic case, to verify these conditions, one has to examine the degrees of all $n$ vertices, thus we have $O(n)$ complexity. In the case of the adjacency list sorted by degrees of vertices, this can be simplified to checking degrees of at most two vertices, which yields $O(1)$ complexity.
\end{proof}

Nevertheless, this description does not give any insights into the structure of $\DD_m$ type non-negative connected posets.
\begin{open}
Give a structural description of Hasse digraphs of $\DD_m$ type non-negative connected posets.
\end{open}


\end{document}